\documentclass[11pt,reqno]{amsart}

\usepackage{a4wide}
\usepackage{amsmath}
\usepackage{amsthm}
\usepackage{amsfonts}
\usepackage{amssymb}
\usepackage{mathtools}
\usepackage{eucal}
\usepackage{mathrsfs}
\usepackage{subfig,epsfig,tikz,float}	

\usepackage{hyperref}
\hypersetup{colorlinks=true, linkcolor=red, filecolor=blue, urlcolor=cyan} 
\usepackage{graphicx}

\usepackage{cancel}
\usepackage{soul}

\theoremstyle{plain}
\newtheorem{thm}{Theorem}[section]
\newtheorem{lem}[thm]{Lemma}

\theoremstyle{definition}
\newtheorem{rem}[thm]{Remark}

\newtheorem{defi}[thm]{Definition}

\numberwithin{thm}{section}
\numberwithin{equation}{section}

\newcommand\ddfrac[2]{\frac{\displaystyle #1}{\displaystyle #2}}

\def\supp{\operatorname{supp}}
\def\esup{\operatornamewithlimits{ess\,sup}}

\def\RHS{\operatorname{RHS}}
\def\LHS{\operatorname{LHS}}

\def\Ces{\operatorname{Ces}}
\def\Cop{\operatorname{Cop}}

\allowdisplaybreaks

\begin{document}
\author{Amiran Gogatishvili\and Tu\u{g}\c{c}e \"{U}nver}

\title{Weighted inequalities involving two Hardy operators}

%\today

\address{A. Gogatishvili,
Institute of Mathematical of the  Czech Academy of Sciences, \v Zitn\'a~25, 115~67 Praha~1, Czech Republic}
\email{gogatish@math.cas.cz}

\address{T. \"{U}nver,
Department of Mathematics, Kirikkale University, 71450 Yahsihan, Kirikkale, T\" urkiye}
\email{tugceunver@kku.edu.tr}
		
\thanks{The research of  A.~Gogatishvili 
was partially supported by  the grant project 23-04720S of the Czech Science Foundation (GA\v{C}R),  The Institute of Mathematics, CAS is supported  by RVO:67985840, by  Shota Rustaveli National Science Foundation (SRNSF), grant no: FR22-17770, and by the
grant Ministry of Education and Science of the Republic of Kazakhstan (project no.
AP14869887).
The research of T.~\"{U}nver was supported by the grant of The Scientific and Technological Research Council of Turkey (TUBITAK), Grant No: 1059B192000075.}
  
\begin{abstract}
We find necessary and sufficient conditions on weights
$u_1, u_2, v_1, v_2$, i.e.  measurable, positive, and finite, a.e. on $(a,b)$, for which there exists a positive constant $C$ such that for given $0 < p_1,q_1,p_2,q_2 <\infty$ the inequality
\begin{equation*}
\begin{split}
\bigg(\int_a^b \bigg(\int_a^t f(s)^{p_2} v_2(s)^{p_2} ds\bigg)^{\frac{q_2}{p_2}} u_2(t)^{q_2} dt \bigg)^{\frac{1}{q_2}}& \\ & \hspace{-3cm}\le 
C \bigg(\int_a^b \bigg(\int_a^t f(s)^{p_1} v_1(s)^{p_1} ds\bigg)^{\frac{q_1}{p_1}} u_1(t)^{q_1} dt \bigg)^{\frac{1}{q_1}}
\end{split}
\end{equation*}
holds for every non-negative, measurable function $f$ on $(a,b)$, where $0 \le a <b \le \infty$. The proof is based on a recently developed discretization method that enables us to overcome the restrictions of the earlier results.
\end{abstract}
	
\maketitle

\medskip

\noindent{\textbf{\emph{
Dedicated to Professor Lars-Erik Persson on the occasion of his 80th birthday.}}}
	
\medskip

\baselineskip=\normalbaselineskip

\section{Introduction}

Let $0 \leq a < b \leq \infty$.  Let $\mathcal{M}$ be the set of all measurable functions on $(a,b)$ and $\mathcal{M}^+$ be the collection of all non-negative functions in $\mathcal{M}$. The functions that are measurable, positive, and finite, a.e. on $(a,b)$ will be referred to as weight functions (or simply weights). For parameters $p,q \in (0,\infty)$ and weights $u,v$, the \textit{weighted Ces\`{a}ro function space} $\Ces_{p,q}(u,v)$ is the collection of all functions $f\in \mathcal{M}$ such that
\begin{equation*}
\|f\|_{\Ces_{p,q}(u,v)}= \bigg(\int_a^b \bigg(\int_a^t |f(s)|^p v(s)^p ds\bigg)^{\frac{q}{p}} u(t)^q dt \bigg)^{\frac{1}{q}}< \infty,
\end{equation*}
and the \textit{weighted Copson function space} $\Cop_{p,q}(u,v)$, is the collection of all functions $f\in \mathcal{M}$ such that
\begin{equation*}
\|f\|_{\Cop_{p,q}(u,v)}= \bigg(\int_a^b \bigg(\int_t^b |f(s)|^p v(s)^p ds\bigg)^{\frac{q}{p}} u(t)^q dt \bigg)^{\frac{1}{q}} < \infty.
\end{equation*}
The classical Ces\`{a}ro sequence and function spaces have a lengthy history. For the details on the history of classical Ces\`{a}ro function spaces, refer to the survey study \cite{As-Ma:14} and \cite{GU:24:CesCes_short}. Additionally, the history and the relations of $\Ces_{p,q}(u,v)$ and $\Cop_{p,q}(u,v)$ spaces are thoroughly covered in the recent paper \cite{GU:24:CesCes_short}.  

Our goal in this paper is to give a characterization of the embedding  
$$\Ces_{p_1,q_1}(u_1,v_1)\hookrightarrow \Ces_{p_2,q_2}(u_2,v_2),$$
in other words, to formulate necessary and sufficient conditions on $u_1, u_2, v_1, v_2$ for given $p_1, p_2, q_1, q_2$, so that
\begin{equation}\label{Ces-Ces}
   \|f\|_{\Ces_{p_2,q_2}(u_2,v_2)} \leq c \|f\|_{\Ces_{p_1,q_1}(u_1,v_1)} 
\end{equation}
holds for every $f \in \mathcal{M}^+$. The optimal constant $c$ in \eqref{Ces-Ces} is 
\begin{align}\label{C}
c := \sup_{\substack{f\in \mathcal{M}^+(a,b)\\f\neq 0}} \dfrac{\|f\|_{\Ces_{p_2,q_2}(u_2,v_2)}}{\|f\|_{\Ces_{p_1,q_1}(u_1,v_1)}} =\sup_{\substack{f\in \mathcal{M}^+(a,b)\\f\neq 0}} \ddfrac{\bigg(\int_a^b \bigg(\int_a^t f(s)^{p_2} v_2(s)^{p_2} ds\bigg)^{\frac{q_2}{p_2}} u_2(t)^{q_2} dt \bigg)^{\frac{1}{q_2}}}{\bigg(\int_a^b \bigg(\int_a^t f(s)^{p_1} v_1(s)^{p_1} ds\bigg)^{\frac{q_1}{p_1}} u_1(t)^{q_1} dt \bigg)^{\frac{1}{q_1}}}.
\end{align}

The characterization of the embeddings $\Cop_{p_1,q_1}(u_1,v_1)\hookrightarrow \Ces_{p_2,q_2}(u_2,v_2)$ was given in \cite{Ca-Go-Ma-Pi:08}, \cite{Go-Mu-Un:17},  \cite{GPU-JFA}. It is noteworthy that when $p_1 = q_1$ or $p_2 = q_2$, inequality \eqref{Ces-Ces} reduces to Hardy and reverse Hardy inequalities, respectively.  For a comprehensive history and development of Hardy-type inequalities, refer to \cite{ok}, \cite{Ku-Pe-Sa:17}. The inequality \eqref{Ces-Ces} was initially studied in \cite{Un:20} using the duality argument, which only works in the case when $p_2 \le q_2$. The case $p_2>q_2$ remained open until recently. In \cite{GU:24:CesCes_short}, the characterization without parameter restriction was announced without providing proof. In this paper, we will present a detailed proof of this result.

It was shown in \cite{GU:24:CesCes_short} that \eqref{Ces-Ces} is equivalent to the inequality
\begin{equation}\label{main}
 \bigg(\int_a^b \bigg(\int_a^t f(s)^r v(s) ds \bigg)^{\frac{q}{r}} u(t) dt \bigg)^{\frac{1}{q}} \leq C\bigg(\int_a^b \bigg(\int_a^t f(s) ds \bigg)^{p} w(t) dt\bigg)^{\frac{1}{p}},
\end{equation}
and the optimal constant $ C$ in \eqref{main} satisfies $c^{p_1} =  C$, where $c$ is defined in \eqref{C}. In this paper, we focus on the characterization of \eqref{main}. We improve the discretization technique given in \cite{GMPTU:JFAA, GPU-JFA, Go-Un:22} and adapt it to the inequality \eqref{main}. 

We should mention that by the change of variables, \eqref{Ces-Ces} allows us to characterize the embedding between weighted Copson function spaces (see, \cite{GU:24:CesCes_short}). Moreover, the characterization of \eqref{main} enables us to investigate the embeddings between local Morrey-type spaces. 

Let $0< p, q < \infty$. Suppose that $u$ is a non-negative, measurable function on $(0, \infty)$ and $v$ is a weight 
function on $\mathbb{R}^n$. The weighted local Morrey-type space $LM_{pq,u}(\mathbb{R}^n,v)$ is defined as the collection of all measurable function $f$ on $\mathbb{R}^n$ such that
$$
\|f\|_{LM_{pq,u}(\mathbb{R}^n,v)}  =  \bigg( \int_0^{\infty} \bigg(  \int_{B(0,t)} |f(s)|^{p} v(s)^p ds 
\bigg)^{\frac{q}{p}} u(t)^q \,dt \bigg)^{\frac{1}{q}} < {\infty}.
$$

It was shown in \cite{GU_NS} that the characterization of 
\begin{equation}\label{LM-LM}
\|f\|_{LM_{p_2q_2,u_2}(\mathbb{R}^n,v_2)}\leq c \|f\|_{LM_{p_1q_1,u_1}(\mathbb{R}^n,v_1)}
\end{equation}
follows easily from \eqref{Ces-Ces}. However, since the case $q_2< p_2$ for \eqref{Ces-Ces} was not known at that time, the characterization of \eqref{LM-LM} was given when $p_2 \le q_2$. The full characterization of \eqref{Ces-Ces} allows us to remove the parameter restriction also in \eqref{LM-LM}.

Let us mention the structure of the paper. In Section~\ref{S:main-results}, we formulate the main results. In Section~\ref{S:BackMat}, we provide basic definitions and results that are independently useful.  In Section~\ref{Disc.c}, we describe the elements of the discretization method and provide the discrete characterization of inequality \eqref{main}.  Finally, we present the proofs of the main results in Section~\ref{S:Proofs}.

We will use the following conventions. We write $A \lesssim B$ if there exists a constant $C\in (0, \infty)$ independent of appropriate quantities such that $A \leq C B$. We write $A \approx B$ if we have both $A \lesssim B $ and $B \lesssim A$.  We put $1/ (\pm \infty) =0$, $0/0 =0$, $0\cdot (\pm \infty)=0$. Moreover, we omit arguments of integrands as well as differentials in integrals when appropriate in order to keep the expository as short as possible. We will denote by $\LHS(*)$ and $\RHS(*)$, the left-hand side and right-hand side of an inequality $(*)$, respectively.

\section{Main results}\label{S:main-results}

For $0\leq a \leq x < y \leq  b \leq \infty$, and weights $v,w$ on $(a,b)$, we define
\begin{align}\label{Vr}
V_r(x,y) =
\begin{cases}
\big(\int_{x}^{y} v^{\frac{1}{1-r}}\big)^{\frac{1-r}{r}} & \text{if \, $0<r<1$,}
\\
\esup\limits_{s \in (x, y)} v(s) & \text{if \, $r=1$,}
\end{cases}
\end{align}
and 
\begin{equation}
V_r(x, y+) = \lim_{t \to y+} V_r(x,t).
\end{equation}
Note, that if $0<r<1$, the functions  $V_r(x,y)$ is continuous, i.e.  $V_r(x,y)=V_r(x, y+)$. 

Denote by
\begin{equation*}
\mathcal{W}(x):= \int_x^b w(t) dt, \, x\in [a,b].
\end{equation*}

\begin{thm}\label{T:main}
Let $0 < r \leq 1$, $0 < p, q < \infty$ and assume that $u,v,w$ are weights on $(a,b)$ such that $0 < \mathcal{W}(x)<\infty$ for all $x\in (a,b)$. Then inequality \eqref{main} holds for all $f\in \mathcal{M}^+$ if and only if

{\rm(i)} either  $p \leq r \leq 1 \leq q$ and $C_1< \infty$, where
\begin{equation}\label{C1}
C_1 :=  \esup_{x\in (a,b)} \mathcal{W}(x)^{-\frac{1}{p}} \esup_{t\in (x,b)} \bigg( \int_t^b u \bigg)^{\frac{1}{q}}  V_r(x,t)
\end{equation}

{\rm(ii)} or  $p\leq q <1$, and $p \leq r\leq 1$ and $C_2 < \infty$, where
\begin{equation}\label{C2}
C_2 :=  \sup_{x\in (a,b)} \mathcal{W}(x)^{-\frac{1}{p}} \bigg( \int_x^b \bigg( \int_t^b u \bigg)^{\frac{q}{1-q}} u(t) V_r(x,t)^{\frac{q}{1-q}} dt \bigg)^{\frac{1-q}{q}}
\end{equation}

{\rm(iii)} or  $r < p \leq q$, $r \leq 1 \leq q$ and $\max\{C_1, C_3\} < \infty$, where $C_1$ is defined in \eqref{C1} and
\begin{equation}\label{C3}
C_3 :=  \sup_{x \in (a,b)} \bigg(\int_{x}^b u\bigg)^{\frac{1}{q}} \bigg(\int_a^{x} \mathcal{W}(t)^{-\frac{p}{p-r}} w(t) V_r(t,x)^{\frac{pr}{p-r}} dt \bigg)^{\frac{p-r}{pr}}
\end{equation}

{\rm(iv)} or $r < p \leq q < 1$ and $\max\{C_2, C_3\} < \infty$, where $C_2$ and $C_3$ are defined in \eqref{C2} and \eqref{C3}, respectively,

{\rm(v)} or $q < p \leq r \leq 1$ and $\max\{C_4, C_5\} < \infty$, where
\begin{equation*}
C_4 := \bigg(\int_a^{b} \bigg(\int_x^b u\bigg)^{\frac{q}{p-q}} u(x)  \esup_{t\in (a, x)}  \mathcal{W}(t)^{-\frac{q}{p-q}} V_r(t,x)^{\frac{pq}{p-q}} dx \bigg)^{\frac{p-q}{pq}}
\end{equation*}
and
\begin{align}\label{C5}
C_5 &:=  \bigg(\int_a^b  w(x) \esup_{y\in (a,x)} \mathcal{W}(y)^{-\frac{p}{p-q}}\bigg( \int_{y}^x \bigg( \int_t^x u \bigg)^{\frac{q}{1-q}} u(t) V_r(y,t)^{\frac{q}{1-q}} dt \bigg)^{\frac{p(1-q)}{p-q}} dx \bigg)^{\frac{p-q}{pq}}
\end{align}

{\rm(vi)} or  $q < p$, $q < 1$, $r < p$, $r \leq 1$ and $\max\{C_1, C_5, C_6\} < \infty$, where  $C_1$ and $C_5$  are defined in  \eqref{C1} and \eqref{C5} respectively, and
\begin{align}\label{C6}
C_6 &:= \Bigg(\int_a^{b} w(x) \esup_{ y\in(a, x)} \mathcal{W}(y)^{-1} \Bigg(\int_y^{x} 
\bigg(\int_t^b u\bigg)^{\frac{q}{p-q}} u(t)dt \bigg) \notag\\
& \hskip+3cm \times \bigg(\int_a^y  \mathcal{W}(t)^{-\frac{p}{p-r}} w(t)V_r(t,y)^{\frac{pr}{p-r}}dt\bigg)^{\frac{q(p-r)}{r(p-q)}} dx \Bigg)^{\frac{p-q}{pq}}
\end{align}

{\rm(vii)} or $r \leq 1 \leq q < p$ and $\max\{ C_1, C_6, C_7\} < \infty$, where  $C_1$ and $C_6$ are  defined in  \eqref{C1} and \eqref{C6} respectively, and 
\begin{align*}
C_7 &:= \Bigg(\int_a^{b}  w(x) \esup_{y\in(a,x) } \mathcal{W}(y)^{-\frac{p}{p-q}}  \bigg(\esup_{t\in(y,x) }  \bigg( \int_t^{x} u \bigg)^{\frac{p}{p-q}} V_r(y,t) ^{\frac{pq}{p-q}}\bigg) dx \Bigg)^{\frac{p-q}{pq}}.
\end{align*}

Moreover, the best constant $C$  in~\eqref{main} satisfies
\begin{equation*}
C\approx
\begin{cases}
	C_1, &\text{in the case \textup{(i)},}\\
	C_2, &\text{in the case \textup{(ii)},}\\
	C_1 + C_3, &\text{in the case \textup{(iii)},}\\
	C_2 + C_3,  &\text{in the case \textup{(iv)},}\\
	C_4 + C_5,  &\text{in the case \textup{(v)},}\\
	C_1 + C_5 + C_6, &\text{in the case \textup{(vi)},}\\
	C_1+ C_6 + C_7, &\text{in the case \textup{(vii)},}
\end{cases}
\end{equation*}
and the multiplicative constants depend only on $p,q,r$.
\end{thm}

\begin{rem}
Note that inequality \eqref{main} holds only for trivial  weights if $r>1$. Hence the assumption that $0 < r \leq 1$, which applies throughout the paper constitutes no restriction (cf.~\cite[Lemma~1]{Un:20}).
\end{rem}

In the next theorem, we provide an equivalent condition for Theorem~\ref{T:main} case (vi), when in addition $r\leq q$ and $p<1$. This characterization will be helpful for future applications, such as expressing multipliers between Ces\`{a}ro spaces.

\begin{thm}\label{T:equiv.solut.}
Let $0 < r \leq q <  p < 1 $ and assume that $u,v,w$ are weights on $(a,b)$ such that $0 < \mathcal{W}(x)<\infty$ for all $x\in (a,b)$. Then inequality \eqref{main} holds for all $f\in \mathcal{M}^+$ if and only if $\max\{\mathcal{C}_5,  \mathcal{C}_6\} < \infty$, where
\begin{align*}
\mathcal{C}_5 &:= \bigg(\int_a^{b} \mathcal{W}(x)^{-\frac{p}{p-q}} w(x) \bigg( \int_x^{b} \bigg( \int_t^{b} u \bigg)^{\frac{q}{1-q}} u(t) V_r(x,t)^{\frac{q}{1-q}} dt \bigg)^{\frac{p(1-q)}{p-q}} dx \bigg)^{\frac{p-q}{pq}}
\notag\\
& \hspace{1cm} + \mathcal{W}(a)^{-\frac{1}{p}} \bigg( \int_a^{b} \bigg( \int_t^{b} u \bigg)^{\frac{q}{1-q}} u(t) V_r(a, t)^{\frac{q}{1-q}}  dt \bigg)^{\frac{1-q}{q}}
\end{align*}
and
\begin{equation*}
\mathcal{C}_6 := \bigg(\int_a^{b} \bigg(\int_x^{b} u\bigg)^{\frac{q}{p-q}} u(x)  \bigg(\int_a^x \mathcal{W}(t)^{-\frac{p}{p-r}} w(t) V_r(t, x)^{\frac{pr}{p-r}} dt \bigg)^{\frac{q(p-r)}{r(p-q)}} dx\bigg)^{\frac{p-q}{pq}}.
\end{equation*}
Moreover, the best constant $C$ in \eqref{main} satisfies $C\approx \mathcal{C}_5 + \mathcal{C}_6$, and the multiplicative constants depend only on $p,q,r$.
\end{thm}

\section{Background material and some useful lemmas}\label{S:BackMat}

Let us introduce some basic and supplemental definitions and results that will be used throughout this paper.  

\begin{defi}
Let $N, M \in \mathbb{Z} \cup \{-\infty, +\infty\}$, $N < M$. A sequence of positive numbers $\{a_k\}_{k=N}^M$ is called \textit{strongly decreasing} if 
\begin{equation*}
\sup\bigg\{ \frac{a_{k+1}}{a_k}, \, N \leq k \leq M\bigg\} < 1,
\end{equation*}
and $\{a_k\}_{k=N}^M$  is called  \textit{strongly increasing} if 
\begin{equation*}
\inf\bigg\{\frac{a_{k+1}}{a_k}, \, N \leq k \leq M\bigg\} > 1.
\end{equation*}
\end{defi}

Throughout the paper, the below lemma will be employed several times. The proofs of \eqref{dec.sum-sum} and \eqref{inc.sum-sum} can be found in \cite[Proposition~2.1]{Go-He-St:96}, for the remaining cases see \cite[Lemmas~3.2-3.4]{Go-Pi:03}. Note that in \cite{Go-Pi:03, Go-He-St:96}, the results are given when $N=-\infty$ and $M=\infty$; however, if $N, M <\infty$, they can be proved by using the same argument. Moreover, \eqref{inc.sup-sup} is a direct consequence of the interchanging the supremum. 

\begin{lem} \cite{Go-Pi:03,Go-He-St:96}
Let $\alpha\in (0,\infty)$ and $N, M \in  \mathbb{Z} \cup \{-\infty, +\infty\}$, $N < M$. Assume that $\{a_k\}_{k=N}^M$ and  $\{b_k\}_{k=N}^M$ are sequences of positive numbers.

{\rm(i)} If $\{a_k\}_{k=N}^M$ is strongly decreasing, then
\begin{equation}\label{dec.sum-sum}
\sum_{k=N}^M a_k \bigg(\sum_{i=N}^k b_i \bigg)^\alpha  \approx \sum_{k=N}^M a_k b_k^\alpha,
\end{equation}
\begin{equation} \label{dec.sum-sup}
\sum_{k=N}^M a_k \sup_{N \leq i\leq k} b_i^{\alpha}   \approx \sum_{k=N}^M a_k b_k^{\alpha}
\end{equation}
and
\begin{equation}\label{dec.sup-sum}
\sup_{N \leq k\leq M} a_k \bigg(\sum_{i=N}^k b_i \bigg)^\alpha \approx \sup_{N \leq k\leq M} a_k b_k^\alpha.
\end{equation}

{\rm(ii)} If $\{a_k\}_{k=N}^M$ is non-decreasing, then
\begin{equation} \label{inc.sup-sup}
\sup_{N \leq k \leq M} a_k \sup_{k \leq i \leq M} b_i =  \sup_{N \leq k \leq M} a_k b_k.
\end{equation}
If, in addition, $\{a_k\}_{k=N}^M$ is strongly increasing, then
\begin{equation}\label{inc.sum-sum}
\sum_{k=N}^M a_k \bigg(\sum_{i=k}^M b_i \bigg)^\alpha  \approx \sum_{k=N}^M a_k b_k^\alpha
\end{equation}
\begin{equation}\label{inc.sup-sum}
\sup_{N \leq k\leq M} a_k \bigg(\sum_{i=k}^M b_i \bigg)^\alpha \approx \sup_{N \leq k\leq M} a_k b_k^\alpha
\end{equation}
and
\begin{equation} \label{inc.sum-sup}
\sum_{k=N}^M a_k \sup_{k \leq i\leq M} b_i^{\alpha}   \approx \sum_{k=N}^M a_k b_k^{\alpha}.
\end{equation}
\end{lem}

\begin{lem} \label{lem_3.7}
Let $N, M \in \mathbb{Z}\cup \{-\infty, +\infty\}$ and $N<M$. Let $\beta \geq 0$. Assume that $\{b_k\}_{k=N}^M$ is a positive sequence, $\{a_k\}_{k=N}^M$ is a strongly increasing sequence and $\{d_{k,i}\}_{k,i=N}^M$ is a regular kernel, that is, a sequence of non-negative numbers, non-increasing in the first variable, non-decreasing in the second variable and satisfies
\begin{equation*}
d_{k,i} \lesssim d_{k,j} + d_{j,i} \quad \textit{for every} \quad i,j,k \in \mathbb{Z}, \,\, k\leq j\leq i.
\end{equation*}
Then
\begin{equation}\label{kersup}
\sup_{N\leq k \leq M} a_k \bigg(\sum_{i=k}^M d_{k,i+1} \, b_i\bigg)^{\beta} \approx \sup_{N\leq k \leq M} a_k \, d_{k,k+1}^{\,\beta} \bigg(\sum_{i=k}^M b_i\bigg)^{\beta} 
\end{equation}
and 
\begin{equation}\label{kersum}
\sum_{k=N}^{M} a_k \bigg(\sum_{i=k}^{M} d_{k,i+1} \, b_i \bigg)^{\beta} \approx \sum_{k=N}^{M} a_k \, d_{k,{k+1}}^{\,\beta}  \bigg( \sum_{i=k}^M  \, b_i \bigg)^{\beta}
\end{equation}
hold.
\end{lem}

\begin{proof}
Let $\{d_{k,i}\}_{k,i=N}^M$ be a regular kernel. Using \cite[Lemma~3.1]{Go-St:13}(see also \cite[Lemma~5.1]{GPU:disc}), there exists an $\alpha_c$, $0<\alpha_c<1$, such that for every $0<\alpha \leq \alpha_c$ and $k \leq k+1 \leq \dots \leq i+1$, we have   
\begin{equation*}
d_{k, i+1} \lesssim \bigg(\sum_{j=k}^{i} d_{j,j+1}^{\,\alpha} \bigg)^{\frac{1}{\alpha} }.  
\end{equation*}
Since $1/\alpha >1$, applying Minkowski's inequality, we have for each $k: N\leq k\leq M$
\begin{align}\label{ker-mink}
\sum_{i=k}^M d_{k,i+1} \, b_i \lesssim \sum_{i=k}^M \bigg(\sum_{j=k}^{i} d_{j,j+1}^{\,\alpha} b_i^{\alpha}\bigg)^{\frac{1}{\alpha}} \leq \bigg(\sum_{j=k}^M d_{j,j+1}^{\,\alpha} \bigg(\sum_{i=j}^M  b_i\bigg)^{\alpha}\bigg)^{\frac{1}{\alpha}}.
\end{align}
Thus, using \eqref{ker-mink} and applying \eqref{inc.sup-sum}, we get 
\begin{align*}
\LHS\eqref{kersup}\lesssim  \sup_{N\leq k \leq M} a_k \bigg(\sum_{j=k}^M d_{j,j+1}^{\,\alpha} \bigg(\sum_{i=j}^M  b_i\bigg)^{\alpha}\bigg)^{\frac{\beta}{\alpha}}
\approx \RHS\eqref{kersup}.
\end{align*}
Moreover, using \eqref{ker-mink} once more and applying \eqref{inc.sum-sum}, we obtain
\begin{align*}
\LHS\eqref{kersum} \lesssim    \sum_{k=N}^M a_k \bigg(\sum_{j=k}^M d_{j,j+1}^{\,\alpha} \bigg(\sum_{i=j}^M b_i\bigg)^{\alpha}\bigg)^{\frac{\beta}{\alpha}} \approx \RHS\eqref{kersum}.
\end{align*}
The reverse estimates easily follow from the monotonicity of $\{d_{k,i}\}_{k,i=N}^M$ in the second variable, more precisely, from the fact that for each $k: N\leq k \leq M$
$$
d_{k,k+1} \bigg(\sum_{i=k}^M b_i\bigg) \leq \sum_{i=k}^M d_{k, i+1} b_i.
$$
\end{proof}

The following lemmas will frequently appear in the proofs of the main results. 

\begin{lem}
Let $0 \le\alpha<\infty$, $0 \leq a \leq b \leq c \leq \infty$ and $u$ be a weight. Then, 
\begin{align}\label{u-estimate}
	\int_{a}^{b}  \bigg(\int_{t}^{c} u \bigg)^{\alpha} u(t) dt \le \bigg(\int_{a}^{b} u \bigg) \bigg(\int_{a}^{c} u \bigg)^{\alpha} \le (1+\alpha) \int_{a}^{b}  \bigg(\int_{t}^{c} u \bigg)^{\alpha} u(t) dt
\end{align}
holds. For $-1 < \alpha \leq 0$, the inequalities take the opposite direction. 
\end{lem}

\begin{proof}
For $0<x<y$ and $\alpha >0$, it is well known that
\begin{equation}\label{new-classical}
 (y-x)y^{\alpha} \leq y^{\alpha +1} - x^{\alpha +1} \leq (\alpha +1)  (y-x)y^{\alpha}  
\end{equation}
(see, for instance, the proof of Lemma 7 from \cite{Ben-GrosErd-decreasing}). Then, choosing $y= \int_a^c u$ and $x=\int_b^c u $,
we have 
\begin{align} 
\bigg(\int_a^c u \bigg) ^{\alpha }\left(  \int_a^b u \right)
\le\bigg(\int_a^cu \bigg)^{\alpha +1} -\bigg( \int_b^c u \bigg)^{\alpha+1} \le (\alpha+1) \bigg(\int_a^c u \bigg)^{\alpha }\left(  \int_a^b u \right) \label{alpha>0-2}
\end{align}
Since 
$$ \bigg(\int_a^cu \bigg)^{\alpha +1} -\bigg( \int_b^c u \bigg)^{\alpha+1} 
=(\alpha+1)\int_a^b  \bigg(\int_t^{c} u \bigg)^{\alpha} u(t) dt,
$$ 
we arrive at \eqref{u-estimate}. 
\end{proof}

\begin{lem}\label{cutinglem}
Let $0<\beta<\infty$ and  $0 \leq a \le c\le d\le b \leq \infty$. Assume that $k(x,y)\colon [a,b]^2 \mapsto [0,\infty)$ is a regular kernel, that is a measurable function that is non-increasing in the first variable and non-decreasing in the second variable and satisfies 
\begin{equation*}
k(x,z) \lesssim k(x,y) + k(y,z) \quad \textit{for every} \quad a\leq x<y<z \leq b.
\end{equation*}
Then
\begin{align}\label{lemma 3.5 ineq}
\int_{c}^{d} \left( \int_t^{b}u \right)^{\beta} u(t) k(a,t) dt&\lesssim \int_{c}^{d} \left( \int_t^{d}u \right)^{\beta} u(t)  k(c,t) dt + \left(\int_d^{b}u \right)^{\beta+1}\int_c^d d\big(k(a,t)\big)\notag\\
& \quad  + \lim_{t\to c+} k(a,t) \bigg[ \left( \int_c^{b}u \right)^{\beta+1} - \left( \int_d^{b}u \right)^{\beta+1} \bigg].
\end{align}
\end{lem}
\begin{proof}
Let $\beta\in(0,\infty)$.
Integrating by parts and the equality  $ \int _t^b u =  \int _t^du + \int _d^b u$ yields 
\begin{align*}
\int_{c}^{d} & \left( \int_t^{b}u \right)^{\beta} u(t) k(a,t) dt \\
& \lesssim   
\int_{c}^{d} \left( \int_t^{b}u \right)^{\beta+1} d \big( k(a,t)\big) + \lim_{t\to c+} k(a,t)  \bigg[ \left( \int_c^{b}u \right)^{\beta+1} - \left( \int_d^{b}u \right)^{\beta+1} \bigg]\\
&\approx \int_{c}^{d} \left( \int_t^{d}u \right)^{\beta+1}  d\big(k(a,t)\big) + 
\left( \int_d^{b}u \right)^{\beta+1} \int_c^d d\big(k(a,t)\big)\\
& \quad  + \lim_{t\to c+} k(a,t) \bigg[  \left( \int_c^{b}u \right)^{\beta+1} - \left( \int_d^{b}u \right)^{\beta+1} \bigg].
\end{align*}
Integrating by parts again we obtain
\begin{align*}
\int_{c}^{d} &  \left( \int_t^{b}u \right)^{\beta} u(t) k(a,t) dt \\ 
& \quad \lesssim \int_{c}^{d} \left( \int_t^{d}u \right)^{\beta} u(t)  k(a,t)dt  + \left( \int_d^{b}u \right)^{\beta+1}\int_c^d d\big(k(a,t)\big)\\
& \qquad  + \lim_{t\to c+} k(a,t)  \bigg[  \left( \int_c^{b}u \right)^{\beta+1} - \left( \int_d^{b}u \right)^{\beta+1} \bigg].
\end{align*}
Observe that if $a=c$ it is nothing but the inequality \eqref{lemma 3.5 ineq}. On the other hand if $a < c \leq t$, using the fact that $k(a,t) \lesssim k(a,c) + k(c,t)$, we obtain
\begin{align}\label{second-term}
\int_{c}^{d} &  \left( \int_t^{b}u \right)^{\beta} u(t) k(a,t) dt \notag\\& \quad \lesssim \int_{c}^{d} \left( \int_t^{d}u \right)^{\beta} u(t) k(c,t) dt +  k(a,c) \bigg( \int_c^{d}u \bigg)^{\beta+1} \notag\\
& \quad + \left( \int_d^{b}u \right)^{\beta+1}\int_c^d d\big(k(a,t)\big) +  \lim_{t\to c+} k(a,t) \bigg[  \left( \int_c^{b}u \right)^{\beta+1} - \left( \int_d^{b}u \right)^{\beta+1} \bigg].
\end{align}
On the other hand, in view of \eqref{new-classical}, it is easy to see that
\begin{equation}\label{final-middle}
 \bigg( \int_c^{d}u \bigg)^{\beta+1} \leq \bigg[\int_c^b u - \int_d^b u\bigg]  \bigg( \int_c^{b}u \bigg)^{\beta} \leq    \bigg[\bigg( \int_c^{b}u \bigg)^{\beta+1} - \bigg( \int_d^{b}u \bigg)^{\beta+1} \bigg].  
\end{equation}
Thus, replacing $k(a,c)$ by $\lim_{t\to c+} k(a,t)$ and applying \eqref{final-middle} to the second term in \eqref{second-term}, we obtain
\begin{align*}
\int_{c}^{d}   \left( \int_t^{b}u \right)^{\beta} u(t) k(a,t) dt & \lesssim \int_{c}^{d} \left( \int_t^{d}u \right)^{\beta} u(t) k(c,t) dt  + \left( \int_d^{b}u \right)^{\beta+1}\int_c^d d\big(k(a,t)\big) \\
& \qquad + \lim_{t\to c+} k(a,t) \bigg[  \left( \int_c^{b}u \right)^{\beta+1} - \left( \int_d^{b}u \right)^{\beta+1} \bigg]
\end{align*}
which is the desired estimate. 
\end{proof}

\section{Discrete characterization and some auxiliary results}\label{Disc.c}

In this section, we will go through the basic components of the discretization method and give the discrete characterization of inequality \eqref{main}.

\begin{defi}\cite[Definition 2.5]{Go-Un:22}
Let $w$ be a non-negative measurable function on $(a,b)$ such that $0<\mathcal{W}(x)<\infty$ for all $x\in (a,b$). A strictly increasing sequence  $\{x_k\}_{k=N}^{\infty}\subset [a,b]$ is said to be a discretizing sequence of the function $\mathcal{W}$, if it satisfies $ \mathcal{W}(x_k)  \approx 2^{-k}$, $N \leq k <  \infty $. If $N > -\infty$ then $x_{N} : =a$, otherwise  $x_{-\infty}:= \lim_{k\rightarrow -\infty} x_k = a$.
\end{defi}

Observe that if $N=-\infty$, then $N+1$ is also $-\infty$.

\begin{lem}\cite[Lemma~2.7]{Go-Un:22}\label{L:int-sup-equiv}
Let $\beta >0$ and $N \in \mathbb{Z}\cup \{-\infty\}$. Assume that $w$ is a weight on $(a, b)$ such that $0<\mathcal{W}(x)<\infty$ for all $x\in (a,b)$ and $\{x_k\}_{k=N}^{\infty}$ is a discretizing sequence of the function $\mathcal{W}$. Then for any $n \colon N\leq n$, 
\begin{equation}\label{int.equiv}
\int_{x_n}^b \mathcal{W}(x)^{\beta-1} w(x) h(x)  dx \approx \sum_{k=n+1}^{\infty} 2^{-k\beta} h(x_k)
\end{equation}
and
\begin{equation}\label{sup.equiv}
\esup_{x \in (x_n,b)}  \mathcal{W}(x)^{\beta} h(x) \approx \sup_{n+1\leq k} 2^{-k{\beta}} h(x_k)
\end{equation}
hold for all non-negative and non-decreasing $h$ on $(a, b)$.
\end{lem}

Next, we need some new tools to help us reconcile the distinction between the discrete and continuous worlds.

\begin{lem}
Let $\beta > 0$ and $N,M \in \mathbb{Z}\cup \{-\infty, +\infty\}$. Assume that $w$ is a weight on $(a,b)$ such that $0<\mathcal{W}(x)<\infty$ for all $x\in (a,b)$ and $\{x_k\}_{k=N}^{\infty}$ is a discretizing sequence of $\mathcal{W}$. Then for any $m\colon m\leq M$,
\begin{equation}\label{P1}
\sup_{x\in (a,x_m)} \mathcal{W}(x)^{-\beta} h(x) \approx \sup_{N+1 \leq k \leq m} 2^{k \beta} h(x_{k-1})
\end{equation}
holds for all non-negative and non-increasing function $h$ on $(a,b)$.
\end{lem}

\begin{proof}
By the definition of a discretizing sequence of $\mathcal{W}$ and the monotonicity of $h$, we immediately have
\begin{align*}
\LHS\eqref{P1} &= \sup_{N+1\leq k \leq m} \sup_{x\in(x_{k-1}, x_k)} \mathcal{W}(x)^{-\beta} h(x) \approx \sup_{N+1\leq k \leq m} 2^{k\beta} \sup_{x\in(x_{k-1}, x_k)} h(x) = \RHS\eqref{P1}.     
\end{align*}
\end{proof}

Next, we show that \eqref{main} is equivalent to two more manageable discrete inequalities.

\begin{lem}\label{L:equiv. ineq.}
Assume that $0 < r \leq 1$, $0 < p, q < \infty$ and $u,v,w$ are weights on $(a,b)$ such that $0 < \mathcal{W}(x)<\infty$ for all $x\in (a,b)$. Suppose that $\{x_k\}_{k=N}^{\infty}$ is a discretizing sequence of $\mathcal{W}$. Then, there exists a positive constant $C$ such that \eqref{main} holds for all $f\in \mathcal{M}^+(a,b)$ if and only if there exist positive constants $C'$ and $C''$ such that
\begin{equation}\label{Vr-inequality}
\bigg( \sum_{k=N+1}^{\infty} \bigg( \sum_{i=N+1}^{k} 2^{i\frac{r}{p}} V_r(x_{i-1},x_i)^r a_i^r \bigg)^{\frac{q}{r}} \int_{x_k}^{x_{k+1}} u \bigg)^{\frac{1}{q}} \leq C' \bigg(\sum_{k=N+1}^{\infty} a_k^p \bigg)^{\frac{1}{p}}
\end{equation}	
and
\begin{equation}\label{H-inequality}
\bigg( \sum_{k=N+1}^{\infty} 2^{k\frac{q}{p}} H(x_{k-1}, x_k)^q a_k^q \bigg)^{\frac{1}{q}} \leq C'' \bigg(\sum_{k=N+1}^{\infty} a_k^p \bigg)^{\frac{1}{p}}
\end{equation}	
hold for every sequence of non-negative numbers $\{a_k\}_{k=N+1}^{\infty}$, where
\begin{equation}\label{H(a,b)}
H(a,b) := \sup_{f\in  \mathcal{M}^+(a,b)} \frac{\bigg(\int_a^b \bigg(\int_a^t f(s)^r v(s) ds \bigg)^{\frac{q}{r}} u(t) dt \bigg)^{\frac{1}{q}}}{\int_a^b f(t) dt}
\end{equation}
and $V_r$ is defined in \eqref{Vr}.
Moreover the least constants $C$, $C'$ and $C''$ respectively in \eqref{main}, \eqref{Vr-inequality} and \eqref{H-inequality} satisfy $C \approx C' + C''$.
\end{lem}

\begin{proof}
We will start with discretizing inequality \eqref{main}. Let $\{x_k\}_{k=N}^{\infty}$ be a discretizing sequence of $\mathcal{W}$. Using \eqref{int.equiv} for $\beta = 1$ and \eqref{dec.sum-sum} for $\alpha =p$, we get
\begin{equation*}
\RHS{\eqref{main}}^p \approx C^p \sum_{k=N+1}^{\infty} 2^{-k} \bigg(\sum_{i=N+1}^{k}\int_{x_{i-1}}^{x_{i}} f\bigg)^p \approx C^p \sum_{k=N+1}^{\infty} 2^{-k} \bigg(\int_{x_{k-1}}^{x_k} f\bigg)^p.
\end{equation*}
On the other hand, we have
\begin{align*}
\LHS\eqref{main}^q &= \sum_{k=N}^{\infty} \int_{x_k}^{x_{k+1}} \bigg( \int_a^t f^r v \bigg)^{\frac{q}{r}} u(t) dt  \\
& \approx \sum_{k=N+1}^{\infty} \bigg( \int_a^{x_k} f^r v\bigg)^{\frac{q}{r}} \int_{x_k}^{x_{k+1}} u  + \sum_{k=N}^{\infty}  \int_{x_k}^{x_{k+1}} \bigg( \int_{x_k}^t f^r v \bigg)^{\frac{q}{r}} u(t) dt.
\end{align*}
Therefore,  there exists a positive constant $C$ such that \eqref{main} holds for all non-negative measurable $f$ on $(a,b)$ if and only if there exist positive constants $ \mathfrak{C'}, \mathfrak{C''}$ such that 
\begin{equation}\label{disc.main.1}
\bigg( \sum_{k=N+1}^{\infty}  \bigg( \int_a^{x_k} f^r v \bigg)^{\frac{q}{r}} \int_{x_k}^{x_{k+1}} u  \bigg)^{\frac{1}{q}} \leq \mathfrak{C'} \bigg( \sum_{k=N+1}^{\infty}  2^{-k} \bigg( 	\int_{x_{k-1}}^{x_{k}} f \bigg)^p \bigg)^{\frac{1}{p}}
\end{equation}
and
\begin{equation}\label{disc.main.2}
\bigg( \sum_{k=N+1}^{\infty} \int_{x_{k-1}}^{x_{k}} \bigg( \int_{x_{k-1}}^t f^r v \bigg)^{\frac{q}{r}} u(t) dt \bigg)^{\frac{1}{q}} \leq \mathfrak{C''} \bigg( \sum_{k=N+1}^{\infty}  2^{-k} \bigg( \int_{x_{k-1}}^{x_{k}} f \bigg)^p \bigg)^{\frac{1}{p}}
\end{equation}
hold for all non-negative measurable $f$ on $(a,b)$. Moreover, $C\approx \mathfrak{C'} + \mathfrak{C''}$.

Next, we will prove that \eqref{disc.main.1} holds for all non-negative measurable $f$ on $(a,b)$ if and only if \eqref{Vr-inequality} holds for every sequence of non-negative numbers $\{a_k\}_{k=N+1}^{\infty}$.Observe that for $0<r\leq 1$ H\"older's inequality and \eqref{Vr} yields, 
\begin{equation}\label{Vr-holder}
\sup_{f\in  \mathcal{M}^+(a, b)} \frac{\bigg(\int_{a}^{b} f^r v \bigg)^{\frac{1}{r}}} {\int_{a}^{b} f}= V_r(a,b),
\end{equation}
see, \cite[(4.12)]{GPU-JFA}.
Then, there exist $f_k\in\mathcal{M}^+(a,b)$, $N+1 \leq k$ such that 	
\begin{equation*}
\supp f_k \subset [x_{k-1}, x_k], \quad \int_{x_{k-1}}^{x_{k}} f_k = 1 \quad \text{and} \quad \int_{x_{k-1}}^{x_k} f_k^r v \gtrsim  V_r(x_{k-1}, x_k)^r.
\end{equation*}
Now, define $f=\sum_{m=N+1}^{\infty} 2^{\frac{m}{p}} a_m f_m $ for any sequence  of non-negative numbers $\{a_k\}_{k=N+1}^{\infty}$. Assume that \eqref{disc.main.1} holds.  Testing \eqref{disc.main.1} with $f$, we get
\begin{align*}
\LHS\eqref{disc.main.1}  &=\bigg( \sum_{k =N+1}^{\infty} \bigg( \sum_{i=N+1}^k \int_{x_{i-1}}^{x_i}  \bigg[\sum_{m=N+1}^{\infty} 2^{\frac{m}{p}} a_m f_m(t) \bigg]^r  v(t) dt \bigg)^{\frac{q}{r}} \int_{x_k}^{x_{k+1}} u \bigg)^{\frac{1}{q}} \\&=\bigg( \sum_{k =N+1}^{\infty} \bigg( \sum_{i=N+1}^k  2^{\frac{ir}{p}}a_i^r \int_{x_{i-1}}^{x_i} f_i^r v \bigg)^{\frac{q}{r}} \int_{x_k}^{x_{k+1}} u \bigg)^{\frac{1}{q}} \\
& \gtrsim \bigg( \sum_{k =N+1}^{\infty} \bigg( \sum_{i=N+1}^k  2^{\frac{ir}{p}} a_i^r V_r(x_{i-1},x_i)^r \bigg)^{\frac{q}{r}} \int_{x_k}^{x_{k+1}} u \bigg)^{\frac{1}{q}} .
\end{align*}
On the other hand,
\begin{align*}
\RHS\eqref{disc.main.1} &= \mathfrak{C'} \bigg( \sum_{k=N+1}^{\infty}  2^{-k} \bigg( 	\int_{x_{k-1}}^{x_{k}} \sum_{m=N+1}^{\infty} 2^{\frac{m}{p}} a_m f_m(t) dt \bigg)^p \bigg)^{\frac{1}{p}} \\ 
&= \mathfrak{C'} \bigg( \sum_{k =N+1}^{\infty} a_k^p \bigg( 	\int_{x_{k-1}}^{x_k} f_k \bigg)^p \bigg)^{\frac{1}{p}} = \mathfrak{C'} \bigg( \sum_{k =N+1}^{\infty}  a_k^p  \bigg)^{\frac{1}{p}}.
\end{align*}
Therefore, \eqref{Vr-inequality} holds for all sequences of non-negative numbers $\{a_k\}_{k=N+1}^{\infty}$ with $C' \lesssim \mathfrak{C}'$.

Conversely, assume that \eqref{Vr-inequality} holds. By \eqref{Vr-holder} we have
\begin{equation*}
V_r(x_{k-1}, x_k)^r \geq \bigg( \int_{x_{k-1}}^{x_k} f^r v \bigg) \bigg(\int_{x_{k-1}}^{x_{k}} f \bigg)^{-r}, \quad N+1 \leq k,
\end{equation*}
for all $f\in \mathcal{M}^+$. Then validity of \eqref{Vr-inequality} states that
\begin{equation*}
\bigg( \sum_{k =N+1}^{\infty} \bigg( \sum_{i=N+1}^{k} 2^{\frac{ir}{p}} a_i^r \bigg( 	\int_{x_{i-1}}^{x_i} f^r v \bigg) \bigg(\int_{x_{i-1}}^{x_i} f \bigg)^{-r} 	\bigg)^{\frac{q}{r}} \int_{x_k}^{x_{k+1}} u \bigg)^{\frac{1}{q}} \leq C' \bigg(\sum_{k =N+1}^{\infty} a_k^p \bigg)^{\frac{1}{p}}
\end{equation*}
holds for all sequences  of non-negative numbers $\{a_k\}_{k=N+1}^{\infty}$. On taking $a_k = 2^{-\frac{k}{p}} \int_{x_{k-1}}^{x_k} f$, we get that \eqref{disc.main.1} holds with $\mathfrak{C'} \leq C'$. Hence, $C' \approx \mathfrak{C'}$.

Now suppose that \eqref{disc.main.2} holds. By \eqref{H(a,b)}, we can choose $h_k\in \mathcal{M}^+(a,b)$, $N+1 \leq k$  such that
\begin{equation*}
\supp h_k \subset [x_{k-1},x_{k}], \quad \int_{x_{k-1}}^{x_{k}} h_k =1 \quad \text{and} \quad \int_{x_{k-1}}^{x_{k}} \bigg( \int_{x_{k-1}}^t h_k^r v\bigg)^\frac{q}{r} u(t) dt \gtrsim H(x_{k-1}, x_{k})^q.
\end{equation*}

Similar to the previous part of the proof testing \eqref{disc.main.2} with $f= \sum_{m=N+1}^{\infty} 2^{\frac{m}{p}}a_m h_m$, where $\{a_m\}_{m=N+1}^{\infty}$ is a sequence of non-negative numbers yields to the inequality \eqref{H-inequality} with $C'' \lesssim \mathfrak{C''}$. 

Conversely, assume that \eqref{H-inequality} holds. Since
\begin{equation*}
H(x_{k-1}, x_{k})^q \geq  \bigg( \int_{x_{k-1}}^{x_{k}} \bigg( \int_{x_{k-1}}^t f^r v\bigg)^\frac{q}{r} u(t) dt \bigg)\bigg(\int_{x_{k-1}}^{x_{k}} f\bigg)^{-q}, \quad  N+1 \leq k,
\end{equation*}
for every $f\in \mathcal{M}^+$, the validity of \eqref{H-inequality} yields that 
\begin{equation*}
\bigg( \sum_{k=N+1}^{\infty} 2^{k\frac{q}{p}} \bigg( \int_{x_{k-1}}^{x_{k}} \bigg( \int_{x_{k-1}}^t f^r v\bigg)^\frac{q}{r} u(t) dt \bigg)\bigg(\int_{x_{k-1}}^{x_{k}} f\bigg)^{-q} a_k^q \bigg)^{\frac{1}{q}} \leq C'' \bigg(\sum_{k=N+1}^{\infty} a_k^p \bigg)^{\frac{1}{p}}
\end{equation*}	
holds for all sequences  of non-negative numbers $\{a_k\}_{k=N+1}^{\infty}$.
Now putting $a_k = 2^{-\frac{k}{p}} \int_{x_{k-1}}^{x_{k}} f$ 
in the latter inequality, we obtain \eqref{disc.main.2}. Moreover, $\mathfrak{C''} \leq C''$. 

Consequently,  there exists a positive constant $C$ such that \eqref{main} holds for all non-negative measurable $f$ on $(a,b)$ if and only if there exist positive constants $ C', C''$ such that \eqref{Vr-inequality} and \eqref{H-inequality} hold for every sequence of non-negative numbers $\{a_k\}_{k=N+1}^{\infty}$.
\end{proof}

We can now state and prove the discrete characterization of inequality \eqref{main}.

\begin{thm}\label{C:discrete solutions}
Let $0 < r \leq 1$, $0 < p, q < \infty$ and $u,v,w$ be weights on $(a,b)$ such that $0 < \mathcal{W}(x)<\infty$ for all $x\in (a,b)$. Assume that $\{x_k\}_{k=N}^{\infty}$ is a discretizing sequence of $\mathcal{W}$. Then \eqref{main} holds for all $f\in \mathcal{M}^+(a,b)$ if and only if 

{\rm(i)} either $p \leq r \leq 1 \leq q$, and $\max\{A_1, B_1\} < \infty$, where
\begin{equation}\label{A1}
A_1 := \sup_{N+1 \leq k} 2^{\frac{k}{p}} V_r(x_{k-1}, x_k) \bigg(\int_{x_k}^b u\bigg)^{\frac{1}{q}}
\end{equation}
and
\begin{equation}\label{B1}
B_1 =  \sup_{N+1 \leq k} 2^{\frac{k}{p}} \esup_{t \in (x_{k-1}, x_{k}) } \bigg( \int_t^{x_{k}} u \bigg)^{\frac{1}{q}}  V_r(x_{k-1}, t)
\end{equation}

{\rm(ii)}  or $p\leq q <1$, $p \leq r\leq 1$ and $\max\{A_1, B_2\} < \infty$, where $A_1$ is defined in \eqref{A1} and
\begin{equation}\label{B2}
B_2 :=  \sup_{N+1 \leq k} 2^{\frac{k}{p}} \bigg( \int_{x_{k-1}}^{x_{k}} \bigg( \int_t^{x_{k}} u \bigg)^{\frac{q}{1-q}} u(t) V_r(x_{k-1}, t)^{\frac{q}{1-q}}  dt \bigg)^{\frac{1-q}{q}}
\end{equation}

{\rm(iii)} or $r < p \leq q$, $r \leq 1 \leq q$ and $\max\{A_2, B_1\} < \infty$, where $B_1$ is defined in \eqref{B1} and
\begin{equation}\label{A2}
A_2 :=  \sup_{N+1 \leq k} \bigg(\int_{x_k}^b u\bigg)^{\frac{1}{q}} \bigg(\sum_{i=N+1}^k 2^{i\frac{r}{p-r}} V_r(x_{i-1}, x_i)^{\frac{pr}{p-r}} \bigg)^{\frac{p-r}{pr}}
\end{equation}

{\rm(iv)} or $r < p \leq q < 1$ and $\max\{A_2, B_2\}< \infty$, where $A_2$ and $B_2$ are defined in \eqref{A2} and \eqref{B2}, respectively

{\rm(v)} or  $q < p \leq r \leq 1$ and  $\max\{A_3, B_3\} < \infty$, where
\begin{equation}\label{A3}
A_3 := \bigg(\sum_{k=N+1}^{\infty} \bigg(\int_{x_k}^{x_{k+1}} u\bigg) \bigg(\int_{x_k}^b u\bigg)^{\frac{q}{p-q}} \sup_{N+1 \leq i\leq k}  2^{i\frac{q}{p-q}} V_r(x_{i-1},x_i)^{\frac{pq}{p-q}} \bigg)^{\frac{p-q}{pq}}
\end{equation}
and
\begin{equation}\label{B3}
B_3 :=   \bigg(\sum_{k=N+1}^{\infty} 2^{k\frac{q}{p-q}} \bigg( \int_{x_{k-1}}^{x_{k}} \bigg( \int_t^{x_{k}} u \bigg)^{\frac{q}{1-q}} u(t) V_r(x_{k-1},t)^{\frac{q}{1-q}} dt \bigg)^{\frac{p(1-q)}{p-q}} \bigg)^{\frac{p-q}{pq}}
\end{equation}

{\rm(vi)} or  $q < p$, $q < 1$, $r < p$, $r \leq 1$, and $\max\{A_4, B_3\} < \infty$, where $B_3$ is defined in \eqref{B3} and
\begin{equation}\label{A4}
A_4 :=  \bigg(\sum_{k=N+1}^{\infty} \bigg(\int_{x_k}^{x_{k+1}} u \bigg) \bigg(\int_{x_k}^b u \bigg)^{\frac{q}{p-q}} \bigg( \sum_{i=N+1}^{k} 2^{i\frac{r}{p-r}} V_r(x_{i-1}, x_i)^{\frac{pr}{p-r}} \bigg)^{\frac{q(p-r)}{r(p-q)}}  \bigg)^{\frac{p-q}{pq}}
\end{equation}

{\rm(vii)} or  $r \leq 1 \leq q < p$, and $\max\{A_4, B_4\} < \infty$, where $A_4$ is defined in \eqref{A4} and 
\begin{equation}\label{B4}
B_4 :=  \bigg( \sum_{k=N+1}^{\infty} 2^{k\frac{q}{p-q}} \esup_{t \in (x_{k-1}, x_{k})} \bigg(\int_t^{x_{k}} u \bigg)^{\frac{p}{p-q}} V_r(x_{k-1},t)^{\frac{pq}{p-q}} \bigg)^{\frac{p-q}{pq}}.
\end{equation}

Moreover, the best constant $C$ in \eqref{main} satisfies 
\begin{equation*}
C\approx
\begin{cases}
	A_1+B_1 &\text{in the case \textup{(i)},} \\
	A_1+B_2 &\text{in the case \textup{(ii)},} \\
	A_2+B_1 &\text{in the case \textup{(iii)},}\\
	A_2+B_2 &\text{in the case \textup{(iv)},}\\
	A_3+B_3 &\text{in the case \textup{(v)},}\\
	A_4+B_3 &\text{in the case \textup{(vi)},}\\
	A_4+B_4 &\text{in the case \textup{(vii)},}
\end{cases}
\end{equation*}
and the multiplicative constants depend only on $p,q,r$.
\end{thm}

\begin{proof}
By Lemma~\ref{L:equiv. ineq.}, the best constant $C$ in \eqref{main} satisfies $C\approx C' + C''$, where $C'$ and $C''$ are the best constant in \eqref{Vr-inequality} and \eqref{H-inequality}, respectively. \eqref{Vr-inequality} is a discrete Hardy inequality. Characterization of discrete Hardy inequality was given in \cite[Theorem~1]{Be:91} and \cite[Theorem 9.2]{Gr:98} by breaking up to several cases depending on $p,q$. The combination of those results is given in \cite[Theorem~4.6]{GPU-JFA}. 

Let us first rewrite inequality \eqref{Vr-inequality} in a more convenient way by setting $a_i^r=y_i$.  We have
\begin{equation*}
\bigg( \sum_{k=N+1}^{\infty} \bigg( \sum_{i=N+1}^{k} 2^{i\frac{r}{p}} V_r(x_{i-1},x_i)^r y_i \bigg)^{\frac{q}{r}} \int_{x_k}^{x_{k+1}} u \bigg)^{\frac{r}{q}} \leq (C')^r \bigg(\sum_{k=N+1}^{\infty} y_i^{\frac{p}{r}} \bigg)^{\frac{r}{p}}.
\end{equation*}	
Now, applying \cite[Theorem~4.6]{GPU-JFA} with parameters $q/r$ and $p/r$, we obtain if $p\leq r$ and $ p\leq q$ then
\begin{equation*}
C' \approx
\sup_{N+1 \leq k} 2^{\frac{k}{p}} V_r(x_{k-1}, x_k) \bigg(\int_{x_k}^b u\bigg)^{\frac{1}{q}}=A_1,
\end{equation*}
if $q<p\leq r$ then
\begin{equation*}
C' \approx
\bigg(\sum_{k=N+1}^{\infty} \bigg(\int_{x_k}^{x_{k+1}} u\bigg) \bigg(\int_{x_k}^b u\bigg)^{\frac{q}{p-q}} \sup_{N+1 \leq i\leq k}  2^{i\frac{q}{p-q}} V_r(x_{i-1},x_i)^{\frac{pq}{p-q}} \bigg)^{\frac{p-q}{pq}}=A_3 
\end{equation*}
if $r<p$ and $q<p$ then
\begin{equation*}
C' \approx
\bigg(\sum_{k=N+1}^{\infty} \bigg(\int_{x_k}^{x_{k+1}} u \bigg) \bigg(\int_{x_k}^b u \bigg)^{\frac{q}{p-q}} \bigg( \sum_{i=-\infty}^{k} 2^{i\frac{r}{p-r}} V_r(x_{i-1}, x_i)^{\frac{pr}{p-r}} \bigg)^{\frac{q(p-r)}{r(p-q)}}  \bigg)^{\frac{p-q}{pq}}=A_4
\end{equation*}
and if $r<p\leq q$ then
\begin{equation*}
C' \approx
\sup_{N+1 \leq k} \bigg(\int_{x_k}^{b} u\bigg)^{\frac{1}{q}} \bigg(\sum_{i=N+1}^k 2^{i\frac{r}{p-r}} V_r(x_{i-1}, x_i)^{\frac{pr}{p-r}} \bigg)^{\frac{p-r}{pr}}=A_2.
\end{equation*}

To characterize $C''$ in \eqref{H-inequality}, we need to make use of the discrete form of Landau's theorem. To this end applying \cite[Theorem~4.5]{GPU-JFA}, we have
\begin{equation} \label{C''-char.}
C'' \approx \begin{cases}
	\sup\limits_{N+1\leq k} 2^{\frac{k}{p}} H(x_{k-1},x_k) , \quad &\text{if \, $p \leq q$}\\
	\bigg(\sum\limits_{N+1}^{\infty} 2^{k\frac{q}{p-q}} H(x_{k-1},x_k)^{\frac{pq}{p-q}} \bigg)^{\frac{p-q}{pq}} \quad &\text{if \, $q < p $}.
\end{cases}
\end{equation}
Thus, it remains to characterize $H(x_{k-1}, x_k)$. Applying \cite[Theorem~4.7]{GPU-JFA}, we have
\begin{equation}\label{H-char.}
H(x_{k-1}, x_k) \approx \begin{cases}
	\esup\limits_{t\in (x_{k-1},x_k)} \bigg( \int\limits_t^{x_k} u \bigg)^{\frac{1}{q}}  V_r(x_{k-1},t) , \quad &\text{if \, $1 \leq q$}\\
	\bigg( \int\limits_{x_{k-1}}^{x_{k}} \bigg( \int\limits_t^{x_{k}} u \bigg)^{\frac{q}{1-q}} u(t) V_r(x_{k-1},t)^{\frac{q}{1-q}} dt \bigg)^{\frac{1-q}{q}} \quad &\text{if \, $q < 1 $}.
\end{cases}
\end{equation} 
Therefore, combining \eqref{C''-char.} with \eqref{H-char.} we obtain if $p\leq q$ and $1\leq q$, then
\begin{equation*}
C'' \approx  \sup_{N+1 \leq k} 2^{\frac{k}{p}} \sup_{t \in (x_{k-1}, x_{k}) } \bigg( \int_t^{x_{k}} u \bigg)^{\frac{1}{q}}  V_r(x_{k-1}, t) =  B_1,
\end{equation*}
if $p\leq q < 1$, then 
\begin{equation*}
C'' \approx  \sup_{N+1 \leq k} 2^{\frac{k}{p}} \bigg( \int_{x_{k-1}}^{x_{k}} \bigg( \int_t^{x_{k}} u \bigg)^{\frac{q}{1-q}} u(t) V_r(x_{k-1}, t)^{\frac{q}{1-q}}  dt \bigg)^{\frac{1-q}{q}} = B_2,
\end{equation*}
if $1\leq q<p$, then
\begin{equation*}
C'' \approx   \bigg( \sum_{k=N+1}^{\infty} 2^{k\frac{q}{p-q}} \sup_{t \in (x_{k-1}, x_{k})} \bigg(\int_t^{x_{k}} u \bigg)^{\frac{p}{p-q}} V_r(x_{k-1},t)^{\frac{pq}{p-q}} \bigg)^{\frac{p-q}{pq}} = B_4,
\end{equation*}
and finally if $q<p$ and $q<1$, then
\begin{equation*}
C'' \approx   \bigg(\sum_{k=N+1}^{\infty} 2^{k\frac{q}{p-q}} \bigg( \int_{x_{k-1}}^{x_{k}} \bigg( \int_t^{x_{k}} u \bigg)^{\frac{q}{1-q}} u(t) V_r(x_{k-1},t)^{\frac{q}{1-q}} dt \bigg)^{\frac{p(1-q)}{p-q}} \bigg)^{\frac{p-q}{pq}} = B_3.
\end{equation*}
Consequently, the results follow.
\end{proof}

We will continue stating and proving some useful relations that will help prove the main result.

\begin{lem}\label{disc-sup} Let $0<q<p<\infty$, $0< r \leq 1$ and $N \in \mathbb{Z}\cup \{-\infty\}$. Assume that $u,v,w$ are weights on $(a,b)$ such that $0 < \mathcal{W}(x)<\infty$ for all $x\in (a,b)$ and $\{x_k\}_{k=N}^{\infty}$ is a discretizing sequence of $\mathcal{W}$.
Then for each $k : N + 1 \le  k$
\begin{align} \label{antidisc_sup-esup}
\sup_{N+1\le i \le k} 2^{i \frac{q}{p-q}} V_r(x_{i-1}, x_i)^{\frac{pq}{p-q}} \approx \esup_{ t\in (a, x_{k})} \mathcal{W}(t)^{-\frac{q}{p-q}}  V_r(t,x_k)^{\frac{pq}{p-q}} 
\end{align}
holds. 
\end{lem}

\begin{proof}
We need to handle the cases $r<1$ and $r=1$, individually. Let us start with the situation when $r<1$. First, using the definition of $V_r$ from \eqref{Vr}, then applying \eqref{inc.sup-sum}, we have 
\begin{align}\label{V-1}
\sup_{N+1\le i \le k} 2^{i \frac{q}{p-q}} V_r(x_{i-1}, x_i)^\frac{pq}{p-q} &= \sup_{N+1\le i \le k} 2^{i \frac{q}{p-q}} \bigg(\int_{x_{i-1}}^{x_i} v^{\frac{1}{1-r}}\bigg)^{\frac{pq(1-r)}{r(p-q)}} \notag \\
& \approx \sup_{N+1\le i \le k} 2^{i \frac{q}{p-q}} \bigg(\sum_{j=i}^k \int_{x_{j-1}}^{x_j} v^{\frac{1}{1-r}}\bigg)^{\frac{pq(1-r)}{r(p-q)}}  \notag \\
&= \sup_{N+1\le i \le k} 2^{i \frac{q}{p-q}} \bigg( \int_{x_{i-1}}^{x_k} v^{\frac{1}{1-r}}\bigg)^{\frac{pq(1-r)}{r(p-q)}}. 
\end{align}
On the other hand, if $r=1$, using the definition of $V_r$ from \eqref{Vr} again, then applying \eqref{inc.sup-sup}, we have 
\begin{align}\label{V-2}
\sup_{N+1\le i \le k} 2^{i \frac{q}{p-q}} V_r(x_{i-1}, x_i)^\frac{pq}{p-q} &= \sup_{N+1\le i \le k} 2^{i \frac{q}{p-q}} \esup_{s\in (x_{i-1}, x_i)} v(s)^\frac{pq}{p-q} \notag \\
& =\sup_{N+1\le i \le k} 2^{i \frac{q}{p-q}} \sup_{i\leq j \leq k} \esup_{s\in (x_{j-1}, x_j)} v(s)^\frac{pq}{p-q} \notag \\
&= \sup_{N+1\le i \le k} 2^{i \frac{q}{p-q}} \esup_{s\in (x_{i-1}, x_k)} v(s)^\frac{pq}{p-q}.
\end{align}
Then, combination of \eqref{V-1} and \eqref{V-2} yields, 
\begin{align}\label{xi-xk}
\sup_{N+1\le i \le k} 2^{i\frac{q}{p-q}} V_r(x_{i-1}, x_i)^\frac{pq}{p-q} &\approx
\sup_{N+1\le i \le k} 2^{i\frac{q}{p-q}} V_r(x_{i-1}, x_k)^\frac{pq}{p-q}.
\end{align}
Finally, applying \eqref{P1} with $h(t)= V_r(t,x_k)$, we arrive at \eqref{antidisc_sup-esup}.
\end{proof}

\begin{lem}
Let $0<r\leq 1$, $0 < p, q < \infty$ and $N \in \mathbb{Z}\cup \{-\infty\}$. Assume that $u,v,w$ are weights on $(a,b)$ such that $0 < \mathcal{W}(x)<\infty$ for all $x\in (a,b)$ and $\{x_k\}_{k=N}^{\infty}$ is a discretizing sequence of $\mathcal{W}$. Then
\begin{equation} \label{N+1<sup}
2^{\frac{N+1}{p}} V_r(a,x_{N+1}) \bigg(\int_{x_{N+1}}^b u\bigg)^{\frac{1}{q}} \leq \sup_{N+1 \leq k} 2^{\frac{k}{p}} \esup_{t \in (x_{k-1}, b) } \bigg( \int_t^b u \bigg)^{\frac{1}{q}}  V_r(x_{k-1}, t)
\end{equation}
holds.
\end{lem}

\begin{proof}
Since $V_r$ is non-decreasing in the second variable, we have
\begin{align*}
2^{\frac{N+1}{p}} V_r(a,x_{N+1}) \bigg(\int_{x_{N+1}}^b u\bigg)^{\frac{1}{q}} & =  2^{\frac{N+1}{p}} V_r(x_N, x_{N+1}) \esup_{t \in (x_{N+1}, b) } \bigg( \int_t^b u \bigg)^{\frac{1}{q}} \\
& \leq 2^{\frac{N+1}{p}}  \esup_{t \in (x_{N}, b) } \bigg( \int_t^b u \bigg)^{\frac{1}{q}} V_r(x_N, t) \\
& \leq \sup_{N+1 \leq k} 2^{\frac{k}{p}}  \esup_{t \in (x_{k-1}, b) } \bigg( \int_t^b u \bigg)^{\frac{1}{q}} V_r(x_{k-1}, t). 
\end{align*}
\end{proof}

\begin{lem} 
Let $0 < r \leq 1$, $0< r < p <\infty$ and $N \in \mathbb{Z}\cup \{-\infty\}$.  Assume that $v,w$ are weights on $(a,b)$ such that $0 < \mathcal{W}(x)<\infty$ for all $x\in (a,b)$ and $\{x_k\}_{k=N}^{\infty}$ is a discretizing sequence of $\mathcal{W}$. Then for each $k: N+2 \leq k$, 
\begin{align} \label{disc_sum<int}
\sum_{i=N+2}^k 2^{i\frac{r}{p-r}} V_r(x_{i-1}, x_i)^{\frac{pr}{p-r}} \lesssim \int_a^{x_{k-1}} \mathcal{W}(t)^{-\frac{p}{p-r}} w(t) V_r(t,x_k)^{\frac{pr}{p-r}} dt
\end{align}
and for each $k: N+1 \leq k$
\begin{align} \label{disc_int<sum}
\int_a^{x_{k}} \mathcal{W}(t)^{-\frac{p}{p-r}} w(t) V_r(t,x_k)^{\frac{pr}{p-r}} dt \lesssim \sum_{i=N+1}^k 2^{i\frac{r}{p-r}} V_r(x_{i-1}, x_i)^{\frac{pr}{p-r}} 
\end{align}
hold. 
\end{lem}

\begin{proof}
Since $\{x_k\}_{k=N}^{\infty}$ is a discretizing sequence of $\mathcal{W}$, we have $\mathcal{W}(x_k) \approx 2^{-k}$, $N \leq k$. Thus, for each $i: N+2 \leq i$, we have
\begin{equation}\label{w-inc}
\int_{x_{i-2}}^{x_{i-1}} \mathcal{W}(t)^{-\frac{p}{p-r}} w(t) dt \approx \int_{x_{i-2}}^{x_{i-1}} d\bigg[\mathcal{W}(t)^{-\frac{r}{p-r}}\bigg] \approx 2^{i\frac{r}{p-r}}.
\end{equation}
Therefore,
\begin{align*}
\sum_{i=N+2}^k 2^{i\frac{r}{p-r}} V_r(x_{i-1}, x_i)^{\frac{pr}{p-r}} \approx & \sum_{i=N+2}^k \int_{x_{i-2}}^{x_{i-1}} \mathcal{W}(t)^{-\frac{p}{p-r}} w(t) dt \,  V_r(x_{i-1}, x_i)^{\frac{pr}{p-r}} \\
\leq & \sum_{i=N+2}^k \int_{x_{i-2}}^{x_{i-1}} \mathcal{W}(t)^{-\frac{p}{p-r}} w(t) V_r(t, x_i)^{\frac{pr}{p-r}} dt \\
\leq & \sum_{i=N+2}^k \int_{x_{i-2}}^{x_{i-1}} \mathcal{W}(t)^{-\frac{p}{p-r}} w(t) V_r(t, x_k)^{\frac{pr}{p-r}} dt \\
= & \int_a^{x_{k-1}} \mathcal{W}(t)^{-\frac{p}{p-r}} w(t) V_r(t,x_k)^{\frac{pr}{p-r}} dt
\end{align*}
holds. Note that the first and the second inequalities follow from the monotonicity of $V_r$ in the first and the second variables, respectively. 

On the other hand, similar definition of the discretizing sequence and the monotonicity of $V_r$ yields
\begin{align*}
\int_a^{x_{k}} \mathcal{W}(t)^{-\frac{p}{p-r}} w(t) V_r(t,x_k)^{\frac{pr}{p-r}} dt 
& = \sum_{i=N+1}^k \int_{x_{i-1}}^{x_{i}} \mathcal{W}(t)^{-\frac{p}{p-r}} w(t) V_r(t, x_k)^{\frac{pr}{p-r}} dt \\
& \lesssim  \sum_{i=N+1}^k 2^{i\frac{r}{p-r}} V_r(x_{i-1}, x_k)^{\frac{pr}{p-r}}. 
\end{align*}
Showing that
\begin{equation}\label{middle}
\sum_{i=N+1}^k 2^{i\frac{r}{p-r}} V_r(x_{i-1}, x_k)^{\frac{pr}{p-r}} \approx \sum_{i=N+1}^k 2^{i\frac{r}{p-r}} V_r(x_{i-1}, x_i)^{\frac{pr}{p-r}}
\end{equation}
holds completes the proof. To this end, we need to handle the cases $r<1$ and $r=1$, separately. If $r<1$, then using \eqref{Vr} and \eqref{inc.sum-sum}, we obtain
\begin{align*}
\sum_{i=N+1}^k 2^{i\frac{r}{p-r}} V_r(x_{i-1}, x_k)^{\frac{pr}{p-r}} = & \sum_{i=N+1}^k 2^{i\frac{r}{p-r}} \bigg(\sum_{j=i}^k \int_{x_{j-1}}^{x_j} v^{\frac{1}{1-r}}\bigg)^{\frac{p(1-r)}{p-r}} \\
\approx & \sum_{i=N+1}^k 2^{i\frac{r}{p-r}} \bigg( \int_{x_{i-1}}^{x_i} v^{\frac{1}{1-r}}\bigg)^{\frac{p(1-r)}{p-r}} \\
= & \sum_{i=N+1}^k 2^{i\frac{r}{p-r}} V_r(x_{i-1}, x_i)^{\frac{pr}{p-r}}.
\end{align*}
On the other hand, if $r=1$, then using \eqref{Vr} and \eqref{inc.sum-sup}, we get
\begin{align*}
\sum_{i=N+1}^k 2^{i\frac{r}{p-r}} V_r(x_{i-1}, x_k)^{\frac{pr}{p-r}} = & \sum_{i=N+1}^k 2^{i\frac{r}{p-r}} \sup_{i\leq j \leq k} \esup_{s\in (x_{j-1}, x_j)} v(s)^{\frac{pr}{p-r}} \\
\approx & \sum_{i=N+1}^k 2^{i\frac{r}{p-r}} \esup_{s\in (x_{i-1}, x_i)} v(s)^{\frac{pr}{p-r}} \\
= & \sum_{i=N+1}^k 2^{i\frac{r}{p-r}} V_r(x_{i-1}, x_i)^{\frac{pr}{p-r}}.
\end{align*}
Thus \eqref{middle} holds for $0<r\leq 1$.
\end{proof}

\section{Proofs of the main result} \label{S:Proofs}

\begin{proof}[Proof of Theorem \ref{T:main}]

\rm{(i)} Let $p \leq r \leq 1 \leq q$. Theorem~\ref{C:discrete solutions}, case (i) states that the optimal constant $C$ in \eqref{main} satisfies $C\approx A_1 + B_1$, where $\{ x_k\}_{k=N}^{\infty}$ is a discretizing sequence of $\mathcal{W}$. Our aim is to show that $C_1 \approx A_1 + B_1$. 

Applying \eqref{P1} with $\beta = \frac{1}{p}$ and $h(x)= \esup_{t\in (x,b)} \big( \int_t^b u \big)^{\frac{1}{q}}  V_r(x,t)$, we have
\begin{align} \label{C1-equiv}
C_1 \approx \sup_{N+1 \leq k} 2^{\frac{k}{p}} \esup_{t\in (x_{k-1},b)} \bigg( \int_t^b u \bigg)^{\frac{1}{q}}  V_r(x_{k-1},t).
\end{align}
Observe that
\begin{align*}
C_1 &  \approx \sup_{N+1 \leq k} 2^{\frac{k}{p}} \esup_{t\in (x_{k-1},x_k)} \bigg( \int_t^b u \bigg)^{\frac{1}{q}}  V_r(x_{k-1},t)
+ \sup_{N+1 \leq k} 2^{\frac{k}{p}} \sup_{k+1\leq  i} \esup_{t\in (x_{i-1}, x_i)} \bigg( \int_t^b u \bigg)^{\frac{1}{q}}  V_r(x_{k-1},t).
\end{align*}
Since 
\begin{equation}\label{V-cut}
V_r(x_{k-1},t) \approx V_r(x_{k-1}, x_{i-1}) + V_r(x_{i-1}, t) \quad\text{for $t \in (x_{i-1}, x_{i}), \, k < i$,}
\end{equation}
we obtain that
\begin{align*}
C_1 &\approx  \sup_{N+1 \leq k} 2^{\frac{k}{p}} \esup_{t\in (x_{k-1},x_k)} \bigg( \int_t^b u \bigg)^{\frac{1}{q}}  V_r(x_{k-1},t) + \sup_{N+1 \leq k} 2^{\frac{k}{p}} \sup_{k+1\leq  i}  \bigg( \int_{x_{i-1}}^b u \bigg)^{\frac{1}{q}}  V_r(x_{k-1},x_{i-1} )\\
& \quad + \sup_{N+1 \leq k} 2^{\frac{k}{p}} \sup_{k+1\leq i} \esup_{t\in (x_{i-1},x_i)} \bigg( \int_t^b u \bigg)^{\frac{1}{q}}  V_r(x_{i-1},t)\\
& \approx \sup_{N+1 \leq k} 2^{\frac{k}{p}} \sup_{k+1\leq  i}  \bigg( \int_{x_{i-1}}^b u \bigg)^{\frac{1}{q}}  V_r(x_{k-1},x_{i-1} ) + \sup_{N+1 \leq k} 2^{\frac{k}{p}} \sup_{k\leq i} \esup_{t\in (x_{i-1},x_i)} \bigg( \int_t^b u \bigg)^{\frac{1}{q}}  V_r(x_{i-1},t) \\
& =: C_{1,1} + C_{1,2}.
\end{align*}
Reindexing $i-1 \mapsto i$ and changing the order of supremum, we have 
\begin{align*}
C_{1,1} = \sup_{N+1 \leq k} 2^{\frac{k}{p}} \sup_{k\leq  i}  \bigg( \int_{x_{i}}^b u \bigg)^{\frac{1}{q}}  V_r(x_{k-1},x_{i} ) = \sup_{N+1 \leq i} \bigg( \int_{x_{i}}^b u \bigg)^{\frac{1}{q}} \sup_{N+1 \leq k\leq  i}  2^{\frac{k}{p}}  V_r(x_{k-1},x_{i}).
\end{align*}
Applying \eqref{xi-xk} and then changing the order of supremum once again, we get
\begin{align} \label{new for A-1}
C_{1,1} & \approx  \sup_{N+1 \leq i} \bigg( \int_{x_i}^b u \bigg)^{\frac{1}{q}} \sup_{N+1 \leq k\leq  i} 2^{\frac{k}{p}} V_r(x_{k-1}, x_k)=  \sup_{N+1 \leq k} 2^{\frac{k}{p}} V_r(x_{k-1}, x_k) \sup_{k\leq  i}  \bigg( \int_{x_i}^b u \bigg)^{\frac{1}{q}}   =A_1.
\end{align}

On the other hand, applying \eqref{inc.sup-sup} we get
\begin{align}\label{C12 equiv}
C_{1,2} & \approx \sup_{N+1 \leq k} 2^{\frac{k}{p}}  \esup_{t\in (x_{k-1},x_k)} \bigg( \int_t^b u \bigg)^{\frac{1}{q}}  V_r(x_{k-1},t).
\end{align}  
Decomposing the integral $\int_t^b$ into the sum $\int_t^{x_k} + \int_{x_k}^b$,
\begin{align}\label{C12<A1+B1}
C_{1,2} & \approx  \sup_{N+1 \leq k} 2^{\frac{k}{p}}  \esup_{t\in (x_{k-1},x_k)} \bigg( \int_t^{x_k} u \bigg)^{\frac{1}{q}}  V_r(x_{k-1},t) + \sup_{N+1 \leq k} 2^{\frac{k}{p}}   \bigg( \int_{x_k}^b u \bigg)^{\frac{1}{q}}  V_r(x_{k-1},x_k)\nonumber\\
& = B_1 + A_1.
\end{align}  
Therefore, we obtain that 
\begin{equation} \label{C1 A1B1}
  C_1\approx A_1 + B_1.   
\end{equation}
Note that this equivalency is independent of the positions of the parameters. The condition $p\leq r \leq 1 \leq q$ affected only the form of the discrete conditions $A_1$ and $B_1$. 

\medskip
\rm{(ii)} Let $p\leq q <1$ and $p\leq r\leq 1$.  Theorem~\ref{C:discrete solutions}, case (ii) states that $C\approx A_1 + B_2$, thus it is enough to prove that $C_2 \approx A_1 + B_2$. 

Using \eqref{P1} with $\beta = \frac{1}{p}$ and
\begin{equation*}
h(x) = \bigg( \int_x^b \bigg( \int_t^b u \bigg)^{\frac{q}{1-q}} u(t) V_r(x,t)^{\frac{q}{1-q}} dt \bigg)^{\frac{1-q}{q}},
\end{equation*}
we have
\begin{equation}\label{C2=C2*}
C_2 \approx \sup_{N+1 \leq k} 2^{\frac{k}{p}} \bigg( \int_{x_{k-1}}^b \bigg( \int_t^b u \bigg)^{\frac{q}{1-q}} u(t) V_r(x_{k-1},t)^{\frac{q}{1-q}} dt \bigg)^{\frac{1-q}{q}}.
\end{equation}
Using Lemma~\ref{cutinglem}, we have
\begin{align*}
C_2  \approx 
&  \sup_{N+1 \leq k} 2^{\frac{k}{p}} \bigg(\sum_{i=k}^{\infty} \int_{x_{i-1}}^{x_i} \bigg( \int_t^b u \bigg)^{\frac{q}{1-q}} u(t) V_r(x_{k-1},t)^{\frac{q}{1-q}}\,dt \bigg)^{\frac{1-q}{q}} \notag\\
& \lesssim \sup_{N+1 \leq k} 2^{\frac{k}{p}} \bigg(\sum_{i=k}^{\infty} \int_{x_{i-1}}^{x_i} \bigg( \int_t^{x_i} u \bigg)^{\frac{q}{1-q}} u(t) V_r(x_{i-1},t)^{\frac{q}{1-q}}\,dt \bigg)^{\frac{1-q}{q}} \notag\\
& \quad + \sup_{N+1 \leq k} 2^{\frac{k}{p}} \bigg(\sum_{i=k}^{\infty}\bigg( \int_{x_i}^b u \bigg)^{\frac{1}{1-q}}  \int_{x_{i-1}}^{x_i} \,d\big(V_r(x_{k-1},t)^{\frac{q}{1-q}}\big) \bigg)^{\frac{1-q}{q}}\notag\\
& \quad + \sup_{N+1 \leq k} 2^{\frac{k}{p}} \bigg(\sum_{i=k}^{\infty} V_r(x_{k-1}, x_{i-1}+)^{\frac{q}{1-q}} \int_{x_{i-1}}^{x_i} \bigg( \int_t^b u \bigg)^{\frac{q}{1-q}} u(t)\,dt \bigg)^{\frac{1-q}{q}}.
\end{align*}
Applying \cite[Lemma 4.4]{GPU-JFA} with  $s=\frac{1}{1-q}$, $
a_i = \int_{x_i}^{x_{i+1}}  u
$ and
$b_i = V_r(x_{k-1}, x_i+)^{\frac{q}{1-q}}$ for the second term, we can see that
\begin{align*}
C_2     
& \lesssim \sup_{N+1 \leq k} 2^{\frac{k}{p}} \bigg(\sum_{i=k}^{\infty} \int_{x_{i-1}}^{x_i} \bigg( \int_t^{x_i} u \bigg)^{\frac{q}{1-q}} u(t) V_r(x_{i-1},t)^{\frac{q}{1-q}} dt \bigg)^{\frac{1-q}{q}}\notag \\    & \quad + \sup_{N+1 \leq k} 2^{\frac{k}{p}} \bigg(\sum_{i=k}^{\infty} V_r(x_{k-1}, x_{i-1}+)^{\frac{q}{1-q}} \bigg(\int_{x_{i-1}}^{x_i} u\bigg) \bigg( \int_{x_{i-1}}^b u \bigg)^{\frac{q}{1-q}} \bigg)^{\frac{1-q}{q}}\\
& \quad + \sup_{N+1 \leq k} 2^{\frac{k}{p}} \bigg(\sum_{i=k}^{\infty} V_r(x_{k-1}, x_{i-1}+)^{\frac{q}{1-q}} \int_{x_{i-1}}^{x_i} \bigg( \int_t^b u \bigg)^{\frac{q}{1-q}} u(t)\,dt \bigg)^{\frac{1-q}{q}}.
\end{align*}
In view of \eqref{u-estimate}, we obtain
\begin{align}
C_2     
& \lesssim \sup_{N+1 \leq k} 2^{\frac{k}{p}} \bigg(\sum_{i=k}^{\infty} \int_{x_{i-1}}^{x_i} \bigg( \int_t^{x_i} u \bigg)^{\frac{q}{1-q}} u(t) V_r(x_{i-1},t)^{\frac{q}{1-q}}\,dt \bigg)^{\frac{1-q}{q}}\notag \\
& \quad + \sup_{N+1 \leq k} 2^{\frac{k}{p}} \bigg(\sum_{i=k}^{\infty} V_r(x_{k-1}, x_{i-1}+)^{\frac{q}{1-q}} \int_{x_{i-1}}^{x_i} \bigg( \int_t^b u \bigg)^{\frac{q}{1-q}} u(t)\,dt \bigg)^{\frac{1-q}{q}} \notag \\
& =: C_{2.1} + C_{2.2}. \label{c2<c21+c22}
\end{align}
 Using \eqref{inc.sup-sum}, we get 
 \begin{equation} \label{c21<}
  C_{2.1}
 \approx \sup_{N+1 \leq k} 2^{\frac{k}{p}} \bigg( \int_{x_{k-1}}^{x_k} \bigg( \int_t^{x_k} u \bigg)^{\frac{q}{1-q}} u(t) V_r(x_{k-1},t)^{\frac{q}{1-q}} dt \bigg)^{\frac{1-q}{q}} = B_2.
 \end{equation}
Since $q<1$, choosing $a_k = 2^{\frac{k}{p}}$, $b_i = \int_{x_{i}}^{x_{i+1}} \big( \int_t^b u \big)^{\frac{q}{1-q}} u(t)\,dt$, $d_{k,i} = V_r(x_{k-1}, x_{i-2}+)^{\frac{q}{1-q}}$, $\beta = \frac{1-q}{q}$ and applying \eqref{kersup} yields 
\begin{align}
C_{2,2} &\approx  \sup_{N+1 \leq k} 2^{\frac{k}{p}} V_r(x_{k-1}, x_{k-1}+) \bigg(\sum_{i=k}^{\infty}  \int_{x_{i-1}}^{x_{i}} \bigg( \int_t^b u \bigg)^{\frac{q}{1-q}} u(t)\,dt \bigg)^{\frac{1-q}{q}} \notag\\  
&\approx  \sup_{N+1 \leq k} 2^{\frac{k}{p}} V_r(x_{k-1}, x_{k-1}+) \bigg( \int_{x_{k-1}}^b u\bigg)^{\frac{1}{q}}.\notag
\end{align}
Decomposing the integral $\int_{x_{k-1}}^b u$ into the sum $\int_{x_{k-1}}^{x_k} u + \int_{x_k}^b u$ and the monotonicity of $V_r$ gives
\begin{align}
C_{2,2}
& \lesssim \sup_{N+1 \leq k} 2^{\frac{k}{p}} V_r(x_{k-1}, x_{k}) \bigg( \int_{x_{k}}^b u\bigg)^{\frac{1}{q}}  + \sup_{N+1 \leq k} 2^{\frac{k}{p}} V_r(x_{k-1}, x_{k-1}+) \bigg( \int_{x_{k-1}}^{x_{k}} u\bigg)^{\frac{1}{q}} \notag\\
&\lesssim  A_1  +  \sup_{N+1 \leq k} 2^{\frac{k}{p}}  \bigg( \int_{x_{k-1}}^{x_{k}} \bigg(\int_{t}^{x_{k}} u \bigg)^{\frac{q}{1-q}} u(t) V_r(x_{k-1}, t)^{\frac{q}{1-q}}\,dt \bigg)^{\frac{1-q}{q}} \notag \\
&= A_1+B_2.  \label{c22<}
\end{align}
Note that in the expressions above, only the assumption $q<1$ was employed. Therefore, combining \eqref{c2<c21+c22} with \eqref{c21<} and \eqref{c22<}, we obtain $C_2 \lesssim A_1 + B_2$ for $q<1$.

Conversely,  as $q<1$,
\begin{align*}
A_1 \approx  \sup_{N+1 \leq k} 2^{\frac{k}{p}} V_r(x_{k-1}, x_k)   \bigg(\int_{x_k}^b \bigg(\int_t^b   u\bigg)^{\frac{q}{1-q}} u(t) dt \bigg)^{\frac{1-q}{q}}.
\end{align*}
Monotonicity of $V_r$ in the second variable and \eqref{C2=C2*} gives
\begin{align*}
A_1 \lesssim  \sup_{N+1 \leq k} 2^{\frac{k}{p}}  \bigg(\int_{x_k}^b \bigg(\int_t^b   u\bigg)^{\frac{q}{1-q}} u(t) V_r(x_{k-1}, t)^{\frac{q}{1-q}}   dt \bigg)^{\frac{1-q}{q}} \lesssim C_2,
\end{align*}
moreover it is clear from \eqref{C2=C2*} that $B_2 \lesssim C_2$. Thus, we have $A_1 + B_2 \lesssim C_2$ for $q<1$. 

Consequently we arrive at $C_2 \approx A_1 + B_2$ when $q<1$. Note that the conditions $p\le q$ and $p\le r\le 1$ affected only the form of discrete conditions $A_1$ and $B_2$.

\medskip
\rm{(iii)} If  $r < p \leq q$, and $r \leq 1 \leq q$, then we have from Theorem~\ref{C:discrete solutions}, case (iii) that $C\approx A_2 + B_1$, thus it is enough to prove that $C_1 + C_3 \approx A_2 + B_1$. 

First of all, we know from \eqref{C1 A1B1} that $C_1 \approx A_1 + B_1$ holds, and this equivalency is independent of the positions of the parameters. Since $r<p$, $A_1 \leq A_2$ is trivial and we have $C_1 \lesssim A_2 + B_1$ for $r<p$.

On the other hand,
\begin{align}\label{C3-equal}
C_3 = \sup_{N+1 \leq k} \sup_{x \in (x_{k-1},x_k)} \bigg(\int_{x}^b u\bigg)^{\frac{1}{q}} \bigg(\int_a^{x} \mathcal{W}(t)^{-\frac{p}{p-r}} w(t) V_r(t,x)^{\frac{pr}{p-r}} dt \bigg)^{\frac{p-r}{pr}}.
\end{align}
Decomposing the integral $\int_a^x$ into sum $\int_a^{x_{k-1}} + \int_{x_{k-1}}^x$ for $N+2 \leq k$, we have
\begin{align*}
C_3 & \approx  \sup_{N+1 \leq k} \sup_{x \in (x_{k-1}, x_k)} \bigg(\int_{x}^b u\bigg)^{\frac{1}{q}} \bigg(\int_{x_{k-1}}^{x} \mathcal{W}(t)^{-\frac{p}{p-r}} w(t) V_r(t,x)^{\frac{pr}{p-r}} dt \bigg)^{\frac{p-r}{pr}}\\
& \quad + \sup_{N+2 \leq k} \sup_{x \in (x_{k-1}, x_k)} \bigg(\int_{x}^b u\bigg)^{\frac{1}{q}} \bigg(\int_a^{x_{k-1}} \mathcal{W}(t)^{-\frac{p}{p-r}} w(t) V_r(t,x)^{\frac{pr}{p-r}} dt \bigg)^{\frac{p-r}{pr}}\\
& =: C_{3,1} + C_{3,2}.
\end{align*}
Since $r<p$ in this case, using the properties of the discretization sequence $\{x_k\}_{k=N}^{\infty}$, we have
\begin{equation}\label{2k equiv-1}
\int_{x_{k-1}}^{x_k} \mathcal{W}(t)^{-\frac{p}{p-r}} w(t) dt \approx 2^{k\frac{r}{p-r}}, \quad N+1 \leq k
\end{equation}
and
\begin{equation}\label{2k equiv 2}
\int_a^{x_{k-1}} \mathcal{W}(t)^{-\frac{p}{p-r}} w(t) dt \lesssim 2^{k\frac{r}{p-r}}, \quad N+2 \leq k.
\end{equation}

Then, first using the monotonicity of $V_r$ in the first variable and then applying \eqref{2k equiv-1}, and \eqref{C12 equiv}, we get 
\begin{align*}
C_{3,1} &\leq \sup_{N+1 \leq k} \esup_{x \in (x_{k-1}, x_k)}  \bigg(\int_{x}^b u\bigg)^{\frac{1}{q}} V_r(x_{k-1},x) \bigg(\int_{x_{k-1}}^{x} \mathcal{W}(t)^{-\frac{p}{p-r}} w(t)  dt \bigg)^{\frac{p-r}{pr}} \\
&  \leq \sup_{N+1 \leq k} \bigg(\int_{x_{k-1}}^{x_k} \mathcal{W}(t)^{-\frac{p}{p-r}} w(t)  dt \bigg)^{\frac{p-r}{pr}} \esup_{x \in (x_{k-1}, x_k)}  \bigg(\int_{x}^b u\bigg)^{\frac{1}{q}} V_r(x_{k-1},x)   \\
& \approx \sup_{N+1 \leq k} 2^{\frac{k}{p}} \esup_{x \in (x_{k-1}, x_k)}  \bigg(\int_{x}^b u\bigg)^{\frac{1}{q}} V_r(x_{k-1},x)\\
&\approx C_{1,2}.
\end{align*}
In view of \eqref{C12<A1+B1}, the fact that $A_1 \leq A_2$ yields
\begin{equation*}
C_{3,1} \lesssim A_2 + B_1, \quad r<p.
\end{equation*}
Moreover, since 
\begin{equation}\label{Vr-cut-t<x}
V_r(t,x) \approx V_r(t, x_{k-1}) + V_r(x_{k-1}, x),\quad t < x_{k-1} < x,
\end{equation}
applying \eqref{2k equiv 2}, we obtain
\begin{align*}
C_{3,2} &\approx \sup_{N+2 \leq k} \bigg(\int_{x_{k-1}}^b u\bigg)^{\frac{1}{q}} \bigg(\int_a^{x_{k-1}} \mathcal{W}(t)^{-\frac{p}{p-r}} w(t) V_r(t,x_{k-1})^{\frac{pr}{p-r}} dt \bigg)^{\frac{p-r}{pr}} \\
& \quad + \sup_{N+2 \leq k} \bigg(\int_a^{x_{k-1}} \mathcal{W}(t)^{-\frac{p}{p-r}} w(t)  dt \bigg)^{\frac{p-r}{pr}}\esup_{x \in (x_{k-1}, x_k)} \bigg(\int_{x}^b u\bigg)^{\frac{1}{q}} V_r(x_{k-1},x)  \\
& \lesssim \sup_{N+1 \leq k} \bigg(\int_{x_k}^b u\bigg)^{\frac{1}{q}} \bigg(\int_a^{x_{k}} \mathcal{W}(t)^{-\frac{p}{p-r}} w(t) V_r(t,x_{k})^{\frac{pr}{p-r}} dt \bigg)^{\frac{p-r}{pr}} \\
& \quad + \sup_{N+2 \leq k}  2^{\frac{k}{p}} \esup_{x \in (x_{k-1}, x_k)} \bigg(\int_{x}^b u\bigg)^{\frac{1}{q}} V_r(x_{k-1},x).
\end{align*}
Applying \eqref{disc_int<sum} to the first term and \eqref{C12 equiv} and \eqref{C12<A1+B1} combined with the fact that $A_1 \leq A_2$, we have
\begin{align*}
C_{3,2} \lesssim A_2+ B_1, \quad r<p.
\end{align*}
Then, we have $C_3 \lesssim A_2 + B_1$ when $r<p$. 
Thus  $C_1 + C_3 \lesssim A_2 + B_1$ when $r<p$.

Conversely, let us now prove that $A_2 + B_1 \lesssim C_1 + C_3$. Observe that
\begin{align*}
A_2 & \approx  \bigg(\int_{x_{N+1}}^b u\bigg)^{\frac{1}{q}} 2^{\frac{N+1}{p}} V_r(a, x_{N+1}) + \sup_{N+2 \leq k} \bigg(\int_{x_k}^b u\bigg)^{\frac{1}{q}} 2^{\frac{N+1}{p}} V_r(a, x_{N+1}) \\
& \quad + \sup_{N+2 \leq k} \bigg(\int_{x_k}^b u\bigg)^{\frac{1}{q}} \bigg(\sum_{i=N+2}^k 2^{i\frac{r}{p-r}} V_r(x_{i-1}, x_i)^{\frac{pr}{p-r}} \bigg)^{\frac{p-r}{pr}}\\
& \lesssim  \bigg(\int_{x_{N+1}}^b u\bigg)^{\frac{1}{q}} 2^{\frac{N+1}{p}} V_r(a, x_{N+1}) + \sup_{N+2 \leq k} \bigg(\int_{x_k}^b u\bigg)^{\frac{1}{q}} \bigg(\sum_{i=N+2}^k 2^{i\frac{r}{p-r}} V_r(x_{i-1}, x_i)^{\frac{pr}{p-r}} \bigg)^{\frac{p-r}{pr}}.
\end{align*}
Applying \eqref{N+1<sup} to the first term and since $r<p$, applying \eqref{disc_sum<int} to the second term, we get
\begin{align*}
A_2 &\lesssim \sup_{N+1 \leq k} 2^{\frac{k}{p}} \esup_{t \in (x_{k-1}, b) } \bigg( \int_t^b u \bigg)^{\frac{1}{q}}  V_r(x_{k-1}, t)\\ 
& \quad + \sup_{N+1 \leq k} \bigg(\int_{x_k}^b u\bigg)^{\frac{1}{q}} \bigg(\int_a^{x_{k}} \mathcal{W}(t)^{-\frac{p}{p-r}} w(t) V_r(t,x_k)^{\frac{pr}{p-r}} dt\bigg)^{\frac{p-r}{pr}}\\
&\leq \sup_{N+1 \leq k} 2^{\frac{k}{p}} \esup_{t \in (x_{k-1}, b) } \bigg( \int_t^b u \bigg)^{\frac{1}{q}}  V_r(x_{k-1}, t)\\ 
& \quad + \sup_{N+1 \leq k} \sup_{x\in(x_{k-1}, x_k)} \bigg(\int_{x}^b u\bigg)^{\frac{1}{q}} \bigg(\int_a^{x} \mathcal{W}(t)^{-\frac{p}{p-r}} w(t) V_r(t,x)^{\frac{pr}{p-r}} dt\bigg)^{\frac{p-r}{pr}}.
\end{align*}
Thus, in view of \eqref{C1-equiv} and \eqref{C3-equal}, we have $A_2 \lesssim C_1 + C_3$ when $r<p$. 

Moreover, using \eqref{C1 A1B1} we have $C_1 \approx A_1 + B_1$, therefore $B_1\lesssim C_1$. Consequently, $A_2 + B_1 \lesssim C_1 + C_3$ holds when $r<p$. Note that the conditions $p\leq q$ and $r \leq 1 \leq q$ affected only the form of discrete conditions $A_2$ and $B_1$.

\medskip
\rm{(iv)} If  $r < p \leq q <1$, then we have from Theorem~\ref{C:discrete solutions}, case (iv) that $ C\approx A_2 + B_2$, thus it is enough to prove that $C_2 + C_3 \approx A_2 + B_2$. 

To begin with, we know from case (ii) that $C_2 \approx A_1 + B_2$ holds when $q<1$. Therefore 
\begin{equation}\label{B2<C2}
B_2 \lesssim C_2, \quad q<1, 
\end{equation}
and since when $r<p$, $A_1 \leq A_2$ is trivial, 
\begin{equation}\label{C2<A2+B2}
C_2 \lesssim A_2 + B_2, \quad q<1, \, \, r<p.
\end{equation}  
Furthermore, since $r<p$, we have from the case (iii) that $A_2 \lesssim C_1 + C_3$. On the other hand, the following estimate 
\begin{equation}\label{GPU-6.3}
\esup_{s\in (x,y)} \bigg(\int_s^y u\bigg)^{\frac{1}{q}} V_r(x, s) \leq \bigg(\int_x^y \bigg(\int_t^y u\bigg)^{\frac{q}{1-q}} u(t) V_r(x, t)^{\frac{q}{1-q}} dt \bigg)^{\frac{1-q}{q}}
 \end{equation}
yields $C_1 \lesssim C_2$ for $q<1$. Thus we have $A_2 \lesssim C_2 + C_3$ for $q<1$ and $r<p$. Then combining this with \eqref{B2<C2} we have $A_2 + B_2 \lesssim C_2 + C_3$, when $q<1$ and $r<p$.

Conversely, in the proof of case (iii), we have shown that $C_3\lesssim A_2 + B_1$ when $r<p$. Additionally \eqref{GPU-6.3} yields that $B_1 \leq B_2$ when $q<1$. Thus, $C_3\lesssim A_2 + B_2$, when $r<p$ and $q<1$. Combining this with \eqref{C2<A2+B2}, we obtain that $C_2 + C_3 \lesssim A_2 + B_2$ when $r<p$ and $q<1$. Note that the condition $p\leq q$ affected only the form of discrete conditions $A_2$ and $B_2$.

\medskip
\rm{(v)}  If $ q <p \le r \le 1$,  then we have from Theorem~\ref{C:discrete solutions}, case (v) that $C\approx A_3 + B_3$, thus it is enough to prove that $C_4 + C_5 \approx A_3 + B_3$. Observe that $\esup_{t\in (a,x)} \mathcal{W}(t)^{-\frac{q}{p-q}}  V_r(t,x)^{\frac{pq}{p-q}}$ is a regular kernel, indeed for $a \le  x_{k-1} < x < x_k \le b$, first decomposing the supremum, then applying \eqref{Vr-cut-t<x},  we have 
\begin{align*}
\esup_{t\in (a,x)}  \mathcal{W} & \,(t)^{-\frac{q}{p-q}}  V_r(t,x)^{\frac{pq}{p-q}} \\
& \approx \esup_{t\in (a,x_{k-1})}  \mathcal{W}(t)^{-\frac{q}{p-q}}  V_r(t,x)^{\frac{pq}{p-q}}  + \esup_{t\in (x_{k-1},x)}  \mathcal{W}(t)^{-\frac{q}{p-q}}  V_r(t,x)^{\frac{pq}{p-q}}   \\
& \approx \esup_{t\in (a,x_{k-1})}  \mathcal{W}(t)^{-\frac{q}{p-q}}  V_r(t,x_{k-1})^{\frac{pq}{p-q}}  + \mathcal{W}(x_{k-1})^{-\frac{q}{p-q}}  V_r(x_{k-1},x)^{\frac{pq}{p-q}}  \\
&\hskip+1cm + \esup_{t\in (x_{k-1},x)}  \mathcal{W}(t)^{-\frac{q}{p-q}}  V_r(t,x)^{\frac{pq}{p-q}}.
\end{align*}
Moreover, using the fact that $\mathcal{W}(t)\approx W(x_{k-1})$ for each $x_{k-1} < t < x_k$
\begin{align}\label{sup-cut}
\esup_{t\in (a,x)}  \mathcal{W} (t)^{-\frac{q}{p-q}}  V_r(t,x)^{\frac{pq}{p-q}}
& \approx \esup_{t\in (a,x_{k-1})} \mathcal{W}(t)^{-\frac{q}{p-q}}  V_r(t,x_{k-1})^{\frac{pq}{p-q}} \notag\\
& \qquad +  \esup_{t\in (x_{k-1},x)}  \mathcal{W}(t)^{-\frac{q}{p-q}}  V_r(t,x)^{\frac{pq}{p-q}}. 
\end{align}
Since $q<p$,  using Lemma~\ref{cutinglem} we get
\begin{align}
C_4^{\frac{pq}{p-q}} &= \sum_{k=N+1}^{\infty} \int_{x_{k-1}}^{x_k} \bigg(\int_{x}^b u\bigg)^{\frac{q}{p-q}} u(x) \bigg(\esup_{t\in (a,x)} \mathcal{W}(t)^{-\frac{q}{p-q}}  V_r(t,x)^{\frac{pq}{p-q}}\Bigg) dx\\
& \lesssim  \sum_{k=N+1}^{\infty} \int_{x_{k-1}}^{x_k} \bigg(\int_{x}^{x_k} u\bigg)^{\frac{q}{p-q}} u(x) \esup_{t\in (x_{k-1},x)} \mathcal{W}(t)^{-\frac{q}{p-q}}  V_r(t,x)^{\frac{pq}{p-q}} dx \notag\\
& \quad + \sum_{k=N+1}^{\infty}  \bigg(\int_{x_k}^{b} u\bigg)^{\frac{p}{p-q}} \Bigg( \esup_{t\in (a,x_k)} \mathcal{W}(t)^{-\frac{q}{p-q}}  V_r(t,x_k)^{\frac{pq}{p-q}} \notag\\
& \hskip+5cm- \esup_{t\in (a,x_{k-1})} \mathcal{W}(t)^{-\frac{q}{p-q}}  V_r(t,x_{k-1})^{\frac{pq}{p-q}}\Bigg) \notag\\
& \quad + \sum_{k=N+1}^{\infty} \Bigg[ \bigg(\int_{x_{k-1}}^b u\bigg)^{\frac{p}{p-q}} - \bigg(\int_{x_{k}}^b u\bigg)^{\frac{p}{p-q}} \Bigg] \lim_{x\to x_{k-1}+} \esup_{t\in (a,x)} \mathcal{W}(t)^{-\frac{q}{p-q}}  V_r(t,x)^{\frac{pq}{p-q}}\notag\\
& := C_{4,1} + C_{4,2} +  C_{4,3}. 
\label{C4-I-IV}
\end{align}
We will find appropriate upper estimates for $C_{4,1}, C_{4,2},  C_{4,3}$. Let us start with $C_{4,1}$. Using the fact that $\mathcal{W}(t)\approx 2^{-k}$ for each  $t\in (x_{k-1},x_{k})$, we obtain 
\begin{align*}
C_{4,1} \approx \sum_{k=N+1}^{\infty} 2^{k\frac{q}{p-q}} \int_{x_{k-1}}^{x_k} \bigg(\int_{x}^{x_k} u\bigg)^{\frac{q}{p-q}} u(x) V_r(x_{k-1},x)^{\frac{pq}{p-q}} dx.    
\end{align*}
Since $V_r(x_{k-1}, t)$ is an increasing function in $t$ and $\frac{p-q}{p(1-q)} <1$ holds for $p<1$, applying a special case of \cite[Proposition 2.1]{He-St:93}, we get
\begin{align}
&\int_{x_{k-1}}^{x_{k}} \left(\int_x^{x_{k}} u\right)^{\frac{q}{p-q}} u(x) V_r(x_{k-1}, x)^{\frac{pq}{p-q}} dx \notag \\
& \hspace{3cm}\lesssim  \left(\int_{x_{k-1}}^{x_{k}} \left(\int_x^{x_{k}} u\right)^{\frac{q}{1-q}} u(x) V_r(x_{k-1},x)^{\frac{q}{1-q}} dx \right)^{\frac{p(1-q)}{p-q}}.\label{p-q<p-pq}
\end{align}
Note that if $p=1$, then \eqref{p-q<p-pq} is an equality. Thus,
\begin{equation} \label{I-A3B3}
C_{4,1} \lesssim B_3^{\frac{pq}{p-q}}, \quad q<p\leq 1.
\end{equation}

Let us move on to $C_{4,2}$. It is evident that
\begin{align*}
C_{4,2}& \lesssim  \sum_{k=N+2}^{\infty}  \bigg(\int_{x_k}^{b} u\bigg)^{\frac{p}{p-q}} \Bigg( \esup_{t\in (a,x_k)} \mathcal{W}(t)^{-\frac{q}{p-q}}  V_r(t,x_k)^{\frac{pq}{p-q}} \\
& \hskip+5cm - \esup_{t\in (a,x_{k-1})} \mathcal{W}(t)^{-\frac{q}{p-q}}  V_r(t,x_{k-1})^{\frac{pq}{p-q}}\Bigg) \\
& \qquad +  \bigg(\int_{x_{N+1}}^{b} u\bigg)^{\frac{p}{p-q}}  \esup_{t\in (a,x_{N+1})} \mathcal{W}(t)^{-\frac{q}{p-q}}  V_r(t,x_{N+1})^{\frac{pq}{p-q}} \\
& := C_{4,21} + C_{4,22}.
\end{align*}
Further, using \cite[Lemma 4.4]{GPU-JFA} with  $s=\frac{q}{p-q}$
\begin{equation*}
a_k = \int_{x_k}^{x_{k+1}}  u(t)dt
\end{equation*}
and
\begin{equation}\label{bk}
b_k = \esup_{t\in (a,x_k)} \mathcal{W}(t)^{-\frac{q}{p-q}} V_r(t,x_k)^{\frac{pq}{p-q}},
\end{equation}
we can see that
\begin{align*}
C_{4,21}& \lesssim  \sum_{k=N+1}^{\infty}  \bigg(\int_{x_k}^{x_{k+1}} u\bigg)\bigg(\int_{x_k}^{b} u\bigg)^{\frac{q}{p-q}} \esup_{t\in (a,x_k)} \mathcal{W}(t)^{-\frac{q}{p-q}}  V_r(t,x_k)^{\frac{pq}{p-q}}.
\end{align*}
Applying \eqref{antidisc_sup-esup}, we obtain $C_{4,21}  \lesssim A_3^{\frac{pq}{p-q}}$ when $q<p$.

On the other hand, observe that
\begin{align*}
C_{4,22}& \approx \int_{x_{N+1}}^{b} \bigg(\int_{x}^{b} u\bigg)^{\frac{q}{p-q}} u(x)dx \esup_{t\in (a,x_{N+1})} \mathcal{W}(t)^{-\frac{q}{p-q}}  V_r(t,x_{N+1})^{\frac{pq}{p-q}}\\  
&  =\sum_{k=N+1}^\infty \int_{x_{k}}^{x_{k+1}}
\bigg(\int_{x}^{b} u\bigg)^{\frac{q}{p-q}} u(x)dx \esup_{t\in (a,x_{N+1})} \mathcal{W}(t)^{-\frac{q}{p-q}}  V_r(t,x_{N+1})^{\frac{pq}{p-q}}.
\end{align*}
Applying first \eqref{u-estimate}, then \eqref{antidisc_sup-esup}, we have
\begin{align} \label{1}
C_{4,22} \leq \sum_{k=N+1}^{\infty}  \bigg(\int_{x_k}^{x_{k+1}} u\bigg)\bigg(\int_{x_k}^{b} u\bigg)^{\frac{q}{p-q}} \esup_{t\in (a,x_k)} \mathcal{W}(t)^{-\frac{q}{p-q}}  V_r(t,x_k)^{\frac{pq}{p-q}} \approx A_3^{\frac{pq}{p-q}}.
\end{align}
Therefore,
\begin{equation}\label{II-A3}
C_{4,2} \lesssim A_3^{\frac{pq}{p-q}}, \quad q<p.
\end{equation}
Finally, since for $k \geq N+2$
\begin{align*}
\lim_{x\to x_{k-1}+}& \esup_{t\in (a,x)} \mathcal{W}(t)^{-\frac{q}{p-q}}  V_r(t,x)^{\frac{pq}{p-q}} \\
&= \lim_{x\to a+} \esup_{t\in (a,x)} \mathcal{W}(t)^{-\frac{q}{p-q}}  V_r(t,x)^{\frac{pq}{p-q}} \\
&+ \sum_{i=N+2}^{k} \bigg(\lim_{x\to x_{i-1}+} \esup_{t\in (a,x)} \mathcal{W}(t)^{-\frac{q}{p-q}}  V_r(t,x)^{\frac{pq}{p-q}} - \lim_{x\to x_{i-2}+} \esup_{t\in (a,x)} \mathcal{W}(t)^{-\frac{q}{p-q}}  V_r(t,x)^{\frac{pq}{p-q}}\bigg)
\end{align*}
holds, we have
\begin{align}
C_{4,3} &\approx \sum_{k=N+2}^{\infty} \Bigg[ \bigg(\int_{x_{k-1}}^b u\bigg)^{\frac{p}{p-q}} - \bigg(\int_{x_{k}}^b u\bigg)^{\frac{p}{p-q}} \Bigg] \notag\\
& \quad\times \sum_{i=N+2}^{k} \bigg(\lim_{x\to x_{i-1}+} \esup_{t\in (a,x)} \mathcal{W}(t)^{-\frac{q}{p-q}}  V_r(t,x)^{\frac{pq}{p-q}} - \lim_{x\to x_{i-2}+} \esup_{t\in (a,x)} \mathcal{W}(t)^{-\frac{q}{p-q}}  V_r(t,x)^{\frac{pq}{p-q}}\bigg)   \notag\\
&\quad +\sum_{k=N+1}^{\infty} \Bigg[ \bigg(\int_{x_{k-1}}^b u\bigg)^{\frac{p}{p-q}} - \bigg(\int_{x_{k}}^b u\bigg)^{\frac{p}{p-q}} \Bigg]
\lim_{x\to a+} \esup_{t\in (a,x)} \mathcal{W}(t)^{-\frac{q}{p-q}}  V_r(t,x)^{\frac{pq}{p-q}}\notag\\
&= \sum_{k=N+2}^{\infty} \Bigg[ \bigg(\int_{x_{k-1}}^b u\bigg)^{\frac{p}{p-q}} - \bigg(\int_{x_{k}}^b u\bigg)^{\frac{p}{p-q}} \Bigg]\notag \\
& \quad\times \sum_{i=N+2}^{k} \bigg(\lim_{x\to x_{i-1}+} \esup_{t\in (a,x)} \mathcal{W}(t)^{-\frac{q}{p-q}}  V_r(t,x)^{\frac{pq}{p-q}} - \lim_{x\to x_{i-2}+} \esup_{t\in (a,x)} \mathcal{W}(t)^{-\frac{q}{p-q}}  V_r(t,x)^{\frac{pq}{p-q}}\bigg)  \notag \\
&\quad +\bigg(\int_a^b u\bigg)^{\frac{p}{p-q}}
\lim_{x\to a+} \esup_{t\in (a,x)} \mathcal{W}(t)^{-\frac{q}{p-q}}  V_r(t,x)^{\frac{pq}{p-q}}\notag\\
& =: C_{4,31} + C_{4,32}. \label{C43-C431-C432}
\end{align}
Changing the order of sums yields
\begin{align*}
C_{4,31} 
&=   \sum_{i=N+2}^{\infty} \bigg(\int_{x_{i-1}}^b u\bigg)^{\frac{p}{p-q}} \bigg(\lim_{x\to x_{i-1}+} \esup_{t\in (a,x)} \mathcal{W}(t)^{-\frac{q}{p-q}}  V_r(t,x)^{\frac{pq}{p-q}}  \notag\\
&\hspace{5cm} - \lim_{x\to x_{i-2}+} \esup_{t\in (a,x)} \mathcal{W}(t)^{-\frac{q}{p-q}}  V_r(t,x)^{\frac{pq}{p-q}}\bigg).
\end{align*}
Decomposing the integral $\int_{x_{i-1}}^{b}u$ into sum $\int_{x_{i-1}}^{x_i}u + \int_{x_i}^{b}u$, we obtain
\begin{align*}
C_{4,31} 
& \lesssim \sum_{i=N+2}^{\infty} \bigg(\int_{x_{i}}^{b} u\bigg)^{\frac{p}{p-q}} \bigg(\lim_{x\to x_{i-1}+} \esup_{t\in (a,x)} \mathcal{W}(t)^{-\frac{q}{p-q}}  V_r(t,x)^{\frac{pq}{p-q}}  \notag \\
&\hspace{5cm} - \lim_{x\to x_{i-2}+} \esup_{t\in (a,x)} \mathcal{W}(t)^{-\frac{q}{p-q}}  V_r(t,x)^{\frac{pq}{p-q}}\bigg)  \notag \\
& + \sum_{i=N+2}^{\infty} \bigg(\int_{x_{i-1}}^{x_i} u\bigg)^{\frac{p}{p-q}} \lim_{x\to x_{i-1}+} \esup_{t\in (a,x)} \mathcal{W}(t)^{-\frac{q}{p-q}}  V_r(t,x)^{\frac{pq}{p-q}}.     
\end{align*}
Applying \cite[Lemma 4.4]{GPU-JFA} with  $s=\frac{q}{p-q}$, $b_i = \lim_{x\to x_{i-1}+} \esup_{t\in (a,x)} \mathcal{W}(t)^{-\frac{q}{p-q}}  V_r(t,x)^{\frac{pq}{p-q}}$ and $
a_i = \int_{x_i}^{x_{i+1}} u$ for the first term, we have that
\begin{align*}
C_{4,31} 
& \lesssim   \sum_{i=N+1}^{\infty} \bigg(\int_{x_{i}}^{x_{i+1}} u\bigg) \bigg(\int_{x_{i}}^{b} u\bigg)^{\frac{q}{p-q}} \lim_{x\to x_{i-1}+} \esup_{t\in (a,x)} \mathcal{W}(t)^{-\frac{q}{p-q}}  V_r(t,x)^{\frac{pq}{p-q}}  \\
& \quad + \sum_{i=N+2}^{\infty} \bigg(\int_{x_{i-1}}^{x_i} u\bigg)^{\frac{p}{p-q}}\lim_{x\to x_{i-1}+} \esup_{t\in (a,x)} \mathcal{W}(t)^{-\frac{q}{p-q}}  V_r(t,x)^{\frac{pq}{p-q}}. \end{align*}
Next, in view of \eqref{sup-cut}, we have
\begin{align*}
C_{4,31} 
& \lesssim   \sum_{i=N+1}^{\infty} \bigg(\int_{x_{i}}^{x_{i+1}} u\bigg) \bigg(\int_{x_{i}}^{b} u\bigg)^{\frac{q}{p-q}} \esup_{t\in (a,x_i)} \mathcal{W}(t)^{-\frac{q}{p-q}}  V_r(t,x_i)^{\frac{pq}{p-q}}  \\
& \quad + \sum_{i=N+2}^{\infty} \bigg(\int_{x_{i-1}}^{x_i} u\bigg)^{\frac{p}{p-q}} \esup_{t\in (a,x_{i-1})} \mathcal{W}(t)^{-\frac{q}{p-q}}  V_r(t,x_{i-1})^{\frac{pq}{p-q}}\\
& \quad +\sum_{i=N+2}^{\infty} \bigg(\int_{x_{i-1}}^{x_i} u\bigg)^{\frac{p}{p-q}}\lim_{x\to x_{i-1}+} \esup_{t\in (x_{i-1},x)} \mathcal{W}(t)^{-\frac{q}{p-q}}  V_r(t,x)^{\frac{pq}{p-q}}\\
& \lesssim  \sum_{i=N+1}^{\infty} \bigg(\int_{x_{i}}^{x_{i+1}} u\bigg) \bigg(\int_{x_{i}}^{b} u\bigg)^{\frac{q}{p-q}} \esup_{t\in (a,x_i)} \mathcal{W}(t)^{-\frac{q}{p-q}}  V_r(t,x_i)^{\frac{pq}{p-q}}  \\
& \quad +\sum_{i=N+2}^{\infty} \bigg(\int_{x_{i-1}}^{x_i} u\bigg)^{\frac{p}{p-q}}\lim_{x\to x_{i-1}+} \esup_{t\in (x_{i-1},x)} \mathcal{W}(t)^{-\frac{q}{p-q}}  V_r(t,x)^{\frac{pq}{p-q}}.  
\end{align*}
Applying \eqref{antidisc_sup-esup} for the first term and using the properties of discretizing sequence for the second term, we obtain
\begin{align}
C_{4,31} & \lesssim A_3^{\frac{pq}{p-q}} + \sum_{i=N+2}^{\infty} \bigg(\int_{x_{i-1}}^{x_i} \bigg(\int_t^{x_i} u\bigg)^{\frac{q}{1-q}} u(t) dt \bigg)^{\frac{p(1-q)}{p-q}}  2^{i\frac{q}{p-q}}   V_r(x_{i-1},x_{i-1}+)^{\frac{pq}{p-q}} \notag \\
 & \leq A_3^{\frac{pq}{p-q}} + \sum_{i=N+2}^{\infty} 2^{i\frac{q}{p-q}}   \bigg(\int_{x_{i-1}}^{x_i} \bigg(\int_t^{x_i} u\bigg)^{\frac{q}{1-q}} u(t) V_r(x_{i-1},t)^{\frac{q}{1-q}} dt \bigg)^{\frac{p(1-q)}{p-q}} \notag \\
 & \leq A_3^{\frac{pq}{p-q}} + B_3^{\frac{pq}{p-q}}. \label{C431-A3-B3}
\end{align}
Lastly, monotonicity of $V_r$ and $\mathcal{W}$, gives
\begin{align*}
C_{4,32} &\leq  \bigg(\int_{a}^b u\bigg)^{\frac{p}{p-q}}  \lim_{x\to a+} \mathcal{W}(x)^{-\frac{q}{p-q}}  V_r(a,x)^{\frac{pq}{p-q}} = \bigg(\int_{a}^{b}  u\bigg)^{\frac{p}{p-q}} \mathcal{W}(a)^{-\frac{q}{p-q}} \lim_{x\to a+}    V_r(a,x)^{\frac{pq}{p-q}}.
\end{align*}
Decomposing the integral $\int_a^b u$ into sum $\int_a^{x_{N+1}} u + \int_{x_{N+1}}^b u$, together with the fact that $\mathcal{W}(a)\approx 2^{-N}$  gives
\begin{align*}
C_{4,32}  & \lesssim  \bigg(\int_{x_{N+1}}^b  u \bigg)^{\frac{p}{p-q}} \mathcal{W}(a)^{-\frac{q}{p-q}} V_r(a,x_{N+1})^{\frac{pq}{p-q}} + \bigg(\int_a^{x_{N+1}}  u\bigg)^{\frac{p}{p-q}}  2^{N\frac{q}{p-q}}  \lim_{x\to a+} V_r(a,x)^{\frac{pq}{p-q}} 
	\\
	& \lesssim C_{4,22} + 2^{N\frac{q}{p-q}}  \bigg(\int_a^{x_{N+1}}  \bigg(\int_t^{x_{N+1}} u\bigg)^{\frac{q}{1-q}} u(t) dt \bigg)^{\frac{p(1-q)}{p-q}} \lim_{x\to a+} V_r(a,x)^{\frac{pq}{p-q}}  \\
	& \leq C_{4,22} + 2^{N\frac{q}{p-q}} \bigg(\int_a^{x_{N+1}}  \bigg(
	\int_t^{x_{N+1}} u\bigg)^{\frac{q}{1-q}} u(t)   V_r(a,t)^{\frac{q}{1-q}} dt \bigg)^{\frac{p(1-q)}{p-q}}.
\end{align*} 
Then in view of \eqref{1}, we have
\begin{align}\label{C432-A3-B3}
	C_{4,32} & \lesssim A_3^{\frac{pq}{p-q}}+ B_3^{\frac{pq}{p-q}}, \quad q<p\le 1.
\end{align}
For future reference note that we have proved that
\begin{equation}\label{new-1}
\bigg(\int_{a}^{b}  u\bigg)^{\frac{p}{p-q}} \mathcal{W}(a)^{-\frac{q}{p-q}} \lim_{x\to a+}    V_r(a,x)^{\frac{pq}{p-q}} \lesssim    A_3^{\frac{pq}{p-q}}+ B_3^{\frac{pq}{p-q}} 
\end{equation}
when $q<p\le 1$. Combining   \eqref{C43-C431-C432}, \eqref{C431-A3-B3} and  \eqref{C432-A3-B3}  we obtain 
\begin{align}\label{C43-A3-B3}
C_{4,3} & \lesssim A_3^{\frac{pq}{p-q}}+ B_3^{\frac{pq}{p-q}}, \quad q<p\le 1.
\end{align}
Therefore, combining  \eqref{C4-I-IV}, \eqref{I-A3B3}, \eqref{II-A3} and \eqref{C43-A3-B3}, we arrive at
\begin{equation} \label{C4<A3 A4}
C_4\lesssim A_3+ B_3, \quad q<p\leq 1.
\end{equation} 

We will now prove that $C_5 \lesssim A_3 + B_3$. Using the fact that $\int_{x_{k-1}}^{x_k}w \approx 2^{-k}, k\geq N+1$, we get
\begin{align}
C_5^{\frac{pq}{p-q}} & = \sum_{k= N+1}^\infty  \int_{x_{k-1}}^{x_{k}}  w(x)  \esup_{y\in 
	(a, x)} \mathcal{W}(y)^{-\frac{p}{p-q}}\bigg( \int_{y}^x \bigg( \int_t^x u \bigg)^{\frac{q}{1-q}} u(t) V_r(y, t)^{\frac{q}{1-q}} dt \bigg)^{\frac{p(1-q)}{p-q}} dx \notag\\
& \lesssim \sum_{k= N+1}^\infty 2^{-k} \esup_{y\in 
	(a, x_{k})} \mathcal{W}(y)^{-\frac{p}{p-q}}\bigg( \int_y^{x_{k}} \bigg( \int_t^{x_k} u \bigg)^{\frac{q}{1-q}} u(t) V_r(y, t)^{\frac{q}{1-q}} dt \bigg)^{\frac{p(1-q)}{p-q}}  \notag\\
&= \sum_{k=N+1}^{\infty }  2^{-k} \sup_{N+1 \le i\le k} 
\esup_{y\in(x_{i-1},x_i)} \mathcal{W}(y)^{-\frac{p}{p-q}}  \bigg( \int_y^{x_{k}} \Bigg( \int_t^{x_k}u \bigg)^{\frac{q}{1-q}} u(t) V_r(y,t)^{\frac{q}{1-q}} dt \Bigg)^{\frac{p(1-q)}{p-q}}. \notag 
\end{align}
Similarly, properties of the discretizing sequence yields,
\begin{align}
C_5^{\frac{pq}{p-q}} & \lesssim \sum_{k=N+2}^{\infty }  2^{-k} \sup_{N+1\le  i  \le k} 2^{i\frac{p}{p-q}} \left( \int_{x_{i-1}}^{x_k} \left( \int_t^{x_{k}}u \right)^{\frac{q}{1-q}} u(t) V_r(x_{i-1},t)^{\frac{q}{1-q}} dt \right)^{\frac{p(1-q)}{p-q}}\notag\\
& \quad +  2^{N\frac{q}{p-q}} \left( \int_{x_{N}}^{x_{N+1}} \left( \int_t^{x_{N+1}}u \right)^{\frac{q}{1-q}} u(t) V_r(x_N,t)^{\frac{q}{1-q}} dt \right)^{\frac{p(1-q)}{p-q}}\notag\\
& = \sum_{k=N+2}^{\infty }  2^{-k} \sup_{N+1\le  i \le k} 2^{i\frac{p}{p-q}}\left(
\sum_{j=i}^{k} \int_{x_{j-1}}^{x_{j}} \left( \int_t^{x_{k}}u \right)^{\frac{q}{1-q}} u(t) V_r(x_{i-1},t)^{\frac{q}{1-q}} dt \right)^{\frac{p(1-q)}{p-q}}\notag\\
& \quad +  2^{N\frac{q}{p-q}} \left( \int_{x_{N}}^{x_{N+1}} \left( \int_t^{x_{N+1}}u \right)^{\frac{q}{1-q}} u(t) V_r(x_N,t)^{\frac{q}{1-q}} dt \right)^{\frac{p(1-q)}{p-q}}.\label{C5-1}
\end{align}
Applying Lemma~\ref{cutinglem} for $ a=x_{i-1} $,  $ b=x_{k} $,  $ c=x_{j-1} $, $ d=x_{j}$ for $N+1 \leq i \leq j \leq k$, $k(x,y)=V_r(x,y)^{\frac{q}{1-q}}$ and $\beta = \frac{q}{1-q}$, we obtain
\begin{align}\label{cuting}
&\int_{x_{j-1}}^{x_{j}} \left( \int_t^{x_{k}}u \right)^{\frac{q}{1-q}} u(t) V_r(x_{i-1},t)^{\frac{q}{1-q}} dt \notag\\
& \hskip+1cm\lesssim \int_{x_{j-1}}^{x_{j}} \left( \int_t^{x_{j}}u \right)^{\frac{q}{1-q}}u(t) V_r(x_{j-1},t)^{\frac{q}{1-q}}dt + \left( \int_{x_{j}}^
{x_{k}}u \right)^{\frac{1}{1-q}}\int_{x_{j-1}}^{x_{j}}  d\big(V_r(x_{i-1},t)^{\frac{q}{1-q}}\big) \notag \\
&  \hskip+3.5cm +  V_r(x_{i-1}, x_{j-1}+)^{\frac{q}{1-q}}  
\Bigg[\bigg( \int_{x_{j-1}}^{x_k}u \bigg)^{\frac{1}{1-q}}-\bigg( \int_{x_{j}}^{x_k}u \bigg)^{\frac{1}{1-q}}\Bigg]. 
\end{align}
Then, applying \eqref{cuting}, we get
\begin{align*}
\sum_{j=i}^{k} \int_{x_{j-1}}^{x_{j}} & \left( \int_t^{x_{k}}u \right)^{\frac{q}{1-q}} u(t)V_r(x_{i-1},t)^{\frac{q}{1-q}}dt \\
&\lesssim\sum_{j=i}^{k} \int_{x_{j-1}}^{x_{j}} \left( \int_t^{x_{j}}u \right)^{\frac{q}{1-q}} u(t)V_r(x_{j-1},t)^{\frac{q}{1-q}}dt \notag\\
& \quad +\sum_{j=i}^{k-1} 
\left( \int_{x_{j}}^{x_{k}}u \right)^{\frac{1}{1-q}}\int_{x_{j-1}}^{x_{j}}d\big(V_r(x_{i-1},t)^{\frac{q}{1-q}}\big) \notag \\
&\quad  +  \sum_{j=i}^{k}  V_r(x_{i-1}, x_{j-1}+)^{\frac{q}{1-q}} \Bigg[\bigg(\int_{x_{j-1}}^{x_{k}} u \bigg)^{\frac{1}{1-q}}-\bigg(\int_{x_{j}}^{x_{k}} u \bigg)^{\frac{1}{1-q}}\Bigg]\\
& \lesssim  \sum_{j=i}^{k} \int_{x_{j-1}}^{x_{j}} \left( \int_t^{x_{j}}u \right)^{\frac{q}{1-q}} u(t)V_r(x_{j-1},t)^{\frac{q}{1-q}}dt\notag \\
& +  \sum_{j=i}^{k} 
\left( \int_{x_{j-1}}^{x_j} u \right) \left( \int_{x_{j-1}}^{x_k} u \right)^{\frac{q}{1-q}} V_r(x_{i-1}, x_{j-1}+)^{\frac{q}{1-q}}.
\end{align*}
Note that to obtain the last inequality, we have applied \cite[Lemma~4.4]{GPU-JFA} for the second term and  \eqref{u-estimate}  for the third term. 

Now, plugging the last estimate in \eqref{C5-1}, then applying \eqref{inc.sup-sum} in the first term, we have
\begin{align*}
C_{5}^{\frac{pq}{p-q}}
& \lesssim  \sum_{k=N+2}^{\infty }  2^{-k} \sup_{N+1\le  i \le k}  2^{i\frac{p}{p-q}} \Bigg( \sum_{j=i}^{k}  \int_{x_{j-1}}^{x_{j}} \left( \int_t^{x_{j}}u \right)^{\frac{q}{1-q}} u(t)V_r(x_{j-1},t)^{\frac{q}{1-q}}dt\Bigg)^{\frac{p(1-q)}{p-q}} \notag \\
& \quad +  \sum_{k=N+2}^{\infty }  2^{-k}\sup_{N+1\le  i \le k}  2^{i\frac{p}{p-q}}  \Bigg(\sum_{j=i}^{k} 
\left( \int_{x_{j-1}}^{x_j} u \right) \left( \int_{x_{j-1}}^{b} u \right)^{\frac{q}{1-q}}  V_r(x_{i-1}, x_{j-1}+)^{\frac{q}{1-q}} \Bigg)^{\frac{p(1-q)}{p-q}}\\
& \quad +  2^{N\frac{q}{p-q}} \left( \int_{x_{N}}^{x_{N+1}} \left( \int_t^{x_{N+1}}u \right)^{\frac{q}{1-q}} u(t) V_r(x_N,t)^{\frac{q}{1-q}} dt \right)^{\frac{p(1-q)}{p-q}}\\
& \approx \sum_{k=N+1}^{\infty }  2^{-k} \sup_{N+1\le  i \le k}  2^{i\frac{p}{p-q}} \Bigg( \int_{x_{i-1}}^{x_{i}} \left( \int_t^{x_{i}}u \right)^{\frac{q}{1-q}} u(t)V_r(x_{i-1},t)^{\frac{q}{1-q}}dt\Bigg)^{\frac{p(1-q)}{p-q}} \notag \\
& \quad +  \sum_{k=N+2}^{\infty }  2^{-k}\sup_{N+1\le  i \le k}  2^{i\frac{p}{p-q}}  \Bigg(\sum_{j=i}^{k} 
\left( \int_{x_{j-1}}^{x_j} u \right) \left( \int_{x_{j-1}}^{b} u \right)^{\frac{q}{1-q}}  V_r(x_{i-1}, x_{j-1}+)^{\frac{q}{1-q}} \Bigg)^{\frac{p(1-q)}{p-q}}\\
&:=C_{5,1} + C_{5,2}.
\end{align*}
It is easy to see that, \eqref{dec.sum-sup} yields
\begin{align*}
C_{5,1}  \approx \sum_{k=N+1}^{\infty }  2^{k\frac{q}{p-q}} \Bigg(  \int_{x_{k-1}}^{x_{k}} \left( \int_t^{x_{k}}u \right)^{\frac{q}{1-q}} u(t)V_r(x_{k-1},t)^{\frac{q}{1-q}}dt\Bigg)^{\frac{p(1-q)}{p-q}}  = B_3^{\frac{pq}{p-q}}.
\end{align*}
On the other hand,
\begin{align*}
C_{5,2} & \leq   \sum_{k=N+2}^{\infty }  2^{-k} \sup_{N+1 \leq i \leq k}\Bigg(\sum_{j=i}^{k} 
\left( \int_{x_{j-1}}^{x_j} u \right) \left( \int_{x_{j-1}}^{b} u \right)^{\frac{q}{1-q}} \\
&\hskip+5cm  \times \sup_{N+1 \le m \le j} 2^{m\frac{1}{1-q}}  V_r(x_{m-1}, x_{j-1}+)^{\frac{q}{1-q}} \Bigg)^{\frac{p(1-q)}{p-q}}\\
& =   \sum_{k=N+2}^{\infty }  2^{-k} \Bigg(\sum_{j=N+1}^{k} 
\left( \int_{x_{j-1}}^{x_j} u \right) \left( \int_{x_{j-1}}^{b} u \right)^{\frac{q}{1-q}} \sup_{N+1 \le m \le j} 2^{m\frac{1}{1-q}}  V_r(x_{m-1}, x_{j-1}+)^{\frac{q}{1-q}} \Bigg)^{\frac{p(1-q)}{p-q}}\\
&\approx \sum_{k=N+2}^{\infty }  2^{-k} \Bigg(\sum_{j=N+2}^{k} 
\left( \int_{x_{j-1}}^{x_j} u \right) \left( \int_{x_{j-1}}^{b} u \right)^{\frac{q}{1-q}} \! \sup_{N+1 \le m \le j} 2^{m\frac{1}{1-q}}  V_r(x_{m-1}, x_{j-1}+)^{\frac{q}{1-q}} \Bigg)^{\frac{p(1-q)}{p-q}} \\
& \quad + \sum_{k=N+2}^{\infty }  2^{-k} \Bigg(\left( \int_{x_N}^{x_{N+1}} u \right) \left( \int_{x_N}^{b} u \right)^{\frac{q}{1-q}} 2^{(N+1)\frac{1}{1-q}}  V_r(x_N, x_{N}+)^{\frac{q}{1-q}} \Bigg)^{\frac{p(1-q)}{p-q}}\\
&=: C_{5,21} + C_{5,22}.
\end{align*}
Note that if $0<r<1$ then $C_{5,22}=0$. If $r=1$, we have 
\begin{equation}\label{inc-sum-2N}
\sum_{k=N+2}^{\infty }  2^{-k} \approx 2^{-N}.
\end{equation}
Then, \eqref{inc-sum-2N} together with the fact that $\mathcal{W}(a)\approx 2^{-N}$ and \eqref{new-1} yields
\begin{align}\label{new-2}
C_{5,22} \lesssim  2^{N\frac{q}{p-q}} \left( \int_{x_N}^{b} u \right)^{\frac{p}{p-q}}  V_r(x_N, x_{N}+)^{\frac{pq}{p-q}} \lesssim A_3^{\frac{pq}{p-q}} + B_3^{\frac{pq}{p-q}}.
\end{align}
Now, applying \eqref{dec.sum-sum}, we obtain
\begin{align*}
C_{5,21}& \approx \sum_{k=N+2}^{\infty }  2^{-k} 
\left( \int_{x_{k-1}}^{x_k} u \right)^{\frac{p(1-q)}{p-q}}\left( \int_{x_{k-1}}^{b} u \right)^{\frac{pq}{p-q}} \sup_{N+1 \le m \le k} 2^{m\frac{p}{p-q}}  V_r(x_{m-1}, x_{k-1}+)^{\frac{pq}{p-q}}.
\end{align*}
Since $\frac{p(1-q)}{p-q} \ge 1$ in this case, 
\begin{align*}
C_{5,21} 
& = \sum_{k=N+2}^{\infty }  2^{-k} \left( \int_{x_{k-1}}^{x_k} u \right)^{\frac{p(1-q)}{p-q}-1} \left( \int_{x_{k-1}}^{x_k} u \right)\left( \int_{x_{k-1}}^{b} u \right)^{\frac{pq}{p-q}}\\
&\hskip+5cm \times \sup_{N+1 \le m \le k} 2^{m\frac{p}{p-q}}  V_r(x_{m-1}, x_{k-1}+)^{\frac{pq}{p-q}} \\
& \approx  \sum_{k=N+2}^{\infty }  
\left( \int_{x_{k-1}}^{x_k} u \right)\left( \int_{x_{k-1}}^{b} u \right)^{\frac{q}{p-q}} \sup_{N+1 \le m \le k} 2^{m\frac{q}{p-q}}  V_r(x_{m-1}, x_{k-1}+)^{\frac{pq}{p-q}}.
\end{align*}
Choosing $a_k = \int_{x_{k-1}}^{x_k}u$, $b_k=\sup_{N+1 \le m \le k} 2^{m\frac{q}{p-q}}  V_r(x_{m-1},x_{k-1}+)^{\frac{pq}{p-q}}$, $s=\frac{q}{p-q}$ and applying \cite[Lemma~4.4]{GPU-JFA}, we get that
\begin{align*}
C_{5,21} &\approx  \sum_{k=N+3}^{\infty }  
\left( \sum_{i=k}^{\infty} \int_{x_{i-1}}^{x_i} u \right)^{\frac{p}{p-q}} \Bigg[ \sup_{N+1 \le m \le k} 2^{m\frac{q}{p-q}}  V_r(x_{m-1},x_{k-1}+)^{\frac{pq}{p-q}} \\
& \hspace{5cm} -\sup_{N+1 \le m \le k-1} 2^{m\frac{q}{p-q}}  V_r(x_{m-1},x_{k-2}+)^{\frac{pq}{p-q}}\Bigg]\\
& \quad + \bigg(\sum_{k=N+2}^{\infty} \int_{x_{k-1}}^{x_k} u \bigg)^{\frac{p}{p-q}} \sup_{N+1 \le m \le N+2} 2^{m\frac{q}{p-q}}  V_r(x_{m-1},x_{N+1}+)^{\frac{pq}{p-q}}\\
&=\sum_{k=N+3}^{\infty }  
\left(  \int_{x_{k-1}}^b u \right)^{\frac{p}{p-q}} \Bigg[ \sup_{N+1 \le m \le k} 2^{m\frac{q}{p-q}}  V_r(x_{m-1},x_{k-1}+)^{\frac{pq}{p-q}} \\
& \hspace{5cm} -\sup_{N+1 \le m \le k-1} 2^{m\frac{q}{p-q}}  V_r(x_{m-1},x_{k-2}+)^{\frac{pq}{p-q}}\Bigg]\\
& \quad + \bigg(\int_{x_{N+1}}^{b} u \bigg)^{\frac{p}{p-q}} \sup_{N+1 \le m \le N+2} 2^{m\frac{q}{p-q}}  V_r(x_{m-1},x_{N+1}+)^{\frac{pq}{p-q}}.
\end{align*}
Decomposing the integral $\int_{x_{k-1}}^b u$ into sum $\int_{x_{k-1}}^{x_k} u + \int_{x_{k}}^b u$ for $k\geq N+2$ gives
\begin{align*}
C_{5,21} & \lesssim \sum_{k=N+3}^{\infty }  
\left(  \int_{x_{k-1}}^{x_k} u \right)^{\frac{p}{p-q}} \sup_{N+1 \le m \le k} 2^{m\frac{q}{p-q}}  V_r(x_{m-1},x_{k-1}+)^{\frac{pq}{p-q}} \\
&\quad +\sum_{k=N+3}^{\infty }  
\left(  \int_{x_{k}}^b u \right)^{\frac{p}{p-q}} \Bigg[ \sup_{N+1 \le m \le k} 2^{m\frac{q}{p-q}}  V_r(x_{m-1},x_{k-1}+)^{\frac{pq}{p-q}} \\
& \hspace{5cm} -\sup_{N+1 \le m \le k-1} 2^{m\frac{q}{p-q}}  V_r(x_{m-1},x_{k-2}+)^{\frac{pq}{p-q}}\Bigg]\\
& \quad + \bigg(\int_{x_{N+1}}^{x_{N+2}} u \bigg)^{\frac{p}{p-q}} \sup_{N+1 \le m \le N+2} 2^{m\frac{q}{p-q}}  V_r(x_{m-1},x_{N+1}+)^{\frac{pq}{p-q}}\\
& \quad + \bigg(\int_{x_{N+2}}^{b} u \bigg)^{\frac{p}{p-q}} \sup_{N+1 \le m \le N+2} 2^{m\frac{q}{p-q}}  V_r(x_{m-1},x_{N+1}+)^{\frac{pq}{p-q}}\\
 & \approx \sum_{k=N+2}^{\infty }  
\left(  \int_{x_{k-1}}^{x_k} u \right)^{\frac{p}{p-q}} \sup_{N+1 \le m \le k} 2^{m\frac{q}{p-q}}  V_r(x_{m-1},x_{k-1}+)^{\frac{pq}{p-q}} \\
&\quad  +\sum_{k=N+3}^{\infty }  
\left(  \int_{x_{k}}^b u \right)^{\frac{p}{p-q}} \Bigg[ \sup_{N+1 \le m \le k} 2^{m\frac{q}{p-q}}  V_r(x_{m-1},x_{k-1}+)^{\frac{pq}{p-q}} \\
& \hspace{5cm} -\sup_{N+1 \le m \le k-1} 2^{m\frac{q}{p-q}}  V_r(x_{m-1},x_{k-2}+)^{\frac{pq}{p-q}}\Bigg]\\
& \quad  + \bigg(\sum_{i=N+2}^{\infty}\int_{x_{i}}^{x_{i+1}}  u \bigg)^{\frac{p}{p-q}} \sup_{N+1 \le m \le N+2} 2^{m\frac{q}{p-q}}  V_r(x_{m-1},x_{N+1}+)^{\frac{pq}{p-q}}.
\end{align*}
Decomposing the supremum in the first term and applying power rules (see for instance \cite[Lemma~1]{Be-Gr:05}) in the third term yields 
\begin{align*}
C_{5,21}& \lesssim \sum_{k=N+2}^{\infty }  
\left(  \int_{x_{k-1}}^{x_k} u \right)^{\frac{p}{p-q}} \sup_{N+1 \le m \le k-1} 2^{m\frac{q}{p-q}}  V_r(x_{m-1},x_{k-1})^{\frac{pq}{p-q}} \\
& \quad  + \sum_{k=N+2}^{\infty }  
\left(  \int_{x_{k-1}}^{x_k} u \right)^{\frac{p}{p-q}}  2^{k\frac{q}{p-q}}  V_r(x_{k-1},x_{k-1}+)^{\frac{pq}{p-q}}\\
&\quad  +\sum_{k=N+3}^{\infty }  
\left(  \int_{x_{k}}^b u \right)^{\frac{p}{p-q}} \Bigg[ \sup_{N+1 \le m \le k} 2^{m\frac{q}{p-q}}  V_r(x_{m-1},x_{k-1}+)^{\frac{pq}{p-q}} \\
& \hspace{5cm} -\sup_{N+1 \le m \le k-1} 2^{m\frac{q}{p-q}}  V_r(x_{m-1},x_{k-2}+)^{\frac{pq}{p-q}}\Bigg]\\
& \quad  + \sum_{i=N+2}^{\infty}\bigg(\int_{x_{i}}^{x_{i+1}}  u \bigg)\bigg(\int_{x_{i}}^{b}  u \bigg)^{\frac{q}{p-q}} \sup_{N+1 \le m \le N+2} 2^{m\frac{q}{p-q}}  V_r(x_{m-1},x_{N+2})^{\frac{pq}{p-q}}.
\end{align*}
Observe that reindexing $k \to k+1$ in the first term, we get
\begin{align*}
C_{5,21}& \lesssim \sum_{k=N+1}^{\infty }  
\left(  \int_{x_{k}}^{x_{k+1}} u \right)\left(  \int_{x_{k}}^{b} u \right)^{\frac{q}{p-q}} \sup_{N+1 \le m \le k} 2^{m\frac{q}{p-q}}  V_r(x_{m-1},x_{k})^{\frac{pq}{p-q}} \\
&\quad  + \sum_{k=N+2}^{\infty }  
\left(  \int_{x_{k-1}}^{x_k} u \right)^{\frac{p}{p-q}}  2^{k\frac{q}{p-q}}  V_r(x_{k-1},x_{k-1}+)^{\frac{pq}{p-q}}\\
&\quad +\sum_{k=N+3}^{\infty }  
\left(  \int_{x_{k}}^b u \right)^{\frac{p}{p-q}} \Bigg[ \sup_{N+1 \le m \le k} 2^{m\frac{q}{p-q}}  V_r(x_{m-1},x_{k-1}+)^{\frac{pq}{p-q}} \\
& \hspace{5cm} -\sup_{N+1 \le m \le k-1} 2^{m\frac{q}{p-q}}  V_r(x_{m-1},x_{k-2}+)^{\frac{pq}{p-q}}\Bigg].
\end{align*}
Applying \cite[Lemma~4.4]{GPU-JFA} with $a_i = \int_{x_{i}}^{x_{i+1}}u$, $b_k=\sup_{N+1 \le m\le k} 2^{m\frac{q}{p-q}}  V_r(x_{m-1},x_{k-1}+)^{\frac{pq}{p-q}}$ and $s=\frac{q}{p-q}$ in the third term, we get that
\begin{align*}
C_{5,21}
& \lesssim \sum_{k=N+1}^{\infty }  
\left(  \int_{x_{k}}^{x_{k+1}} u \right)\left(  \int_{x_{k}}^{b} u \right)^{\frac{q}{p-q}} \sup_{N+1 \le m \le k} 2^{m\frac{q}{p-q}}  V_r(x_{m-1},x_{k})^{\frac{pq}{p-q}} \\
& \quad + \sum_{k=N+2}^{\infty }  
\bigg(\int_{x_{k-1}}^{x_k} \left(  \int_{x_{k-1}}^{x_k} u \right)^{\frac{q}{1-q}} u(t) dt \bigg)^{\frac{p(1-q)}{p-q}}  2^{k\frac{q}{p-q}}  V_r(x_{k-1},x_{k-1}+)^{\frac{pq}{p-q}}\\
& \lesssim A_3^{\frac{pq}{p-q}} + B_3^{\frac{pq}{p-q}}.
\end{align*}
Thus 
\begin{equation}\label{C5<A3+B3}
C_5 \lesssim A_3 + B_3, \quad q< \min\{p,1\}.
\end{equation}

It remains to prove that $A_3 + B_3 \lesssim C_4 + C_5$. 

Using \eqref{u-estimate} together with \eqref{antidisc_sup-esup}, we obtain when $q<p$,
\begin{align*}
A_3^{\frac{pq}{p-q}} &\lesssim \sum_{k=N+1}^{\infty} \bigg(\int_{x_k}^{x_{k+1}}\bigg(\int_{x}^b u\bigg)^{\frac{q}{p-q}} u(x)dx \bigg)  \esup_{t\in (a,x_k)}  \mathcal{W}(t)^{-\frac{q}{p-q}} V_r(t,x_k)^{\frac{pq}{p-q}} \\
&\leq \sum_{k=N+1}^{\infty} \int_{x_k}^{x_{k+1}}\bigg(\int_{x}^b u\bigg)^{\frac{q}{p-q}} u(x) \esup_{t\in(a, x)}  \mathcal{W}(t)^{-\frac{q}{p-q}} V_r(t,x)^{\frac{pq}{p-q}} dx \\
&\leq  C_4^{\frac{pq}{p-q}}.
\end{align*}

In addition, using the properties of the discretizing sequence,
\begin{align}
B_3^{\frac{pq}{p-q}} 
& \approx  \sum_{k=N+1}^{\infty} \bigg(\int_{x_{k}}^{x_{k+1}} w\bigg) \mathcal{W}(x_{k-1})^{-\frac{p}{p-q}} 
\bigg(\int_{x_{k-1}}^{x_k} \bigg( \int_t^{x_k} u \bigg)^{\frac{q}{1-q}} u(t) V_r(x_{k-1},t)^{\frac{q}{1-q}} dt \bigg)^{\frac{p(1-q)}{p-q}}\notag\\
& \leq    \sum_{k=N+1}^{\infty} \bigg(\int_{x_{k}}^{x_{k+1}} w\bigg) \esup_{y \in (a,x_k)} \mathcal{W}(y)^{-\frac{p}{p-q}}
\bigg(\int_{y}^{x_k} \bigg( \int_t^{x_k} u \bigg)^{\frac{q}{1-q}} u(t) V_r(y,t)^{\frac{q}{1-q}} dt \bigg)^{\frac{p(1-q)}{p-q}} \notag\\
& \leq  \sum_{k=N+1}^{\infty} \int_{x_{k}}^{x_{k+1}}  w(x)  \esup_{y \in (a, x)} \mathcal{W}(y)^{-\frac{p}{p-q}}
\bigg(\int_{y}^{x} \bigg( \int_t^{x} u \bigg)^{\frac{q}{1-q}} u(t) V_r(y,t)^{\frac{q}{1-q}} dt \bigg)^{\frac{p(1-q)}{p-q}} dx \notag\\
& \lesssim C_5^{\frac{pq}{p-q}} \label{B3<C5}
\end{align}
holds when $q<\min\{p,1\}$.

Consequently, we have
\begin{equation*}
A_3 + B_3 \lesssim C_4 + C_5, \quad q<\min\{p,1\},
\end{equation*}
which completes the proof in this case.

\medskip
\rm{(vi)} By Theorem~\ref{C:discrete solutions}, case (vi), we have that $C\approx A_4 + B_3$. We will prove that $ A_4+B_3  \approx C_1 + C_5 + C_6 $.

Let us start by finding an appropriate upper bound for $C_1$. Applying \eqref{GPU-6.3}, it is easy to see that $B_1\lesssim B_3$. On the other hand, in view of \eqref{new for A-1}, 
\begin{align*}
A_1 &= \sup_{N+1 \leq i} \bigg( \int_{x_i}^b u \bigg)^{\frac{1}{q}} \sup_{N+1 \leq k\leq  i} 2^{\frac{k}{p}} V_r(x_{k-1}, x_k) \\
& = \bigg(\sup_{N+1 \leq i} \bigg( \sum_{m=i}^{\infty} \int_{x_m}^{x_{m+1}} u \bigg)^{\frac{p}{p-q}} \sup_{N+1 \leq k\leq  i} 2^{k\frac{q}{p-q}} V_r(x_{k-1}, x_k)^{\frac{pq}{p-q}} \bigg)^{\frac{p-q}{pq}}.
\end{align*}
Since $q<p$, applying power rules, we get
\begin{align*}
A_1 & \approx \bigg(\sup_{N+1 \leq i}  \sum_{m=i}^{\infty} \bigg(\int_{x_m}^{x_{m+1}} u\bigg) \bigg(\int_{x_m}^{b} u\bigg)^{\frac{q}{p-q}} \sup_{N+1 \leq k\leq  i} 2^{k\frac{q}{p-q}} V_r(x_{k-1}, x_k)^{\frac{pq}{p-q}} \bigg)^{\frac{p-q}{pq}}\\
& \leq \bigg(\sup_{N+1 \leq i}  \sum_{m=i}^{\infty} \bigg(\int_{x_m}^{x_{m+1}} u\bigg) \bigg(\int_{x_m}^{b} u\bigg)^{\frac{q}{p-q}} \sup_{N+1 \leq k\leq  m} 2^{k\frac{q}{p-q}} V_r(x_{k-1}, x_k)^{\frac{pq}{p-q}} \bigg)^{\frac{p-q}{pq}}\\
& =A_3.
\end{align*}
Moreover, since $\max\{r,q\}<p$, it is clear that $A_3 \leq A_4$. Thus
\begin{equation}\label{A1<A4}
A_1 \lesssim A_4, \quad   \max\{r,q\}<p.
\end{equation}
Therefore, by \eqref{C1 A1B1} 
\begin{equation}  \label{C1 A4B3}
C_1\lesssim A_4 + B_3,\quad  \max\{r,q\}<p.
\end{equation}

Our second step is to find an upper bound for $C_5$. Using   \eqref{C5<A3+B3} and the fact that $A_3 \leq A_4$, we obtain that 
\begin{equation} \label{C5 A4B3}
   C_5 \lesssim   B_3+ A_3 \leq B_3 + A_4, \quad \max\{r,q\}<p, \, \, q<1.
\end{equation} 
Now, we shall find an upper estimate for $C_6$. Observe that
\begin{align}\label{C6-new}
C_6^{\frac{pq}{p-q}} &= \sum_{k=N+1}^{\infty} \int_{x_{k-1}}^{x_k} w(x) \esup_{y\in (a,x)} \mathcal{W}(y)^{-1} \bigg( \int_y^x \bigg(\int_{t}^b u\bigg)^{\frac{q}{p-q}} u(t)dt\bigg)
\notag\\
& \hskip+5cm \times 
\bigg(\int_{a}^{y} \mathcal{W}(t)^{-\frac{p}{p-r}} w(t)V_r(t,y)^{\frac{pr}{p-r}}dt\Bigg)^{\frac{q(p-r)}{r(p-q)}} dx \notag\\
& \lesssim   \sum_{k=N+1}^{\infty} 2^{-k} \esup_{y\in (a, x_k)} 
\mathcal{W}(y)^{-1} 
\bigg(\int_y^{x_k} \bigg(\int_{t}^b u\bigg)^{\frac{q}{p-q}} u(t)dt\bigg) \notag\\
& \hskip+5cm \times 
\bigg(\int_{a}^{y} \mathcal{W}(t)^{-\frac{p}{p-r}} w(t)V_r(t,y)^{\frac{pr}{p-r}}dt\Bigg)^{\frac{q(p-r)}{r(p-q)}}   \notag\\
& \approx \sum_{k=N+1}^{\infty} 2^{-k} \sup_{N+1\le j \le k} 2^{j}\esup_{y\in (x_{j-1}, x_j)} 
\bigg(\int_y^{x_k} \bigg(\int_{t}^b u\bigg)^{\frac{q}{p-q}} u(t)dt\bigg)  \notag\\
& \hskip+5cm \times 
\bigg(\int_{a}^{y} \mathcal{W}(t)^{-\frac{p}{p-r}} w(t)V_r(t,y)^{\frac{pr}{p-r}}dt\Bigg)^{\frac{q(p-r)}{r(p-q)}}.
\end{align}
Decomposing the integral $\int_y^{x_k}$ into sum $\int_y^{x_j} + \int_{x_j}^{x_k}$, we obtain
\begin{align}
C_6^{\frac{pq}{p-q}}&\lesssim \sum_{k=N+2}^{\infty} 2^{-k} \sup_{N+1\le j \le k-1} 2^{j} 
\bigg(\int_{x_j}^{x_k} \bigg(\int_{t}^b u\bigg)^{\frac{q}{p-q}} u(t)dt\bigg)  \notag\\
& \hskip+5cm \times 
\bigg(\int_{a}^{x_j} \mathcal{W}(t)^{-\frac{p}{p-r}} w(t)V_r(t,x_j)^{\frac{pr}{p-r}}dt\Bigg)^{\frac{q(p-r)}{r(p-q)}}  \notag \\
&  \quad + \sum_{k=N+1}^{\infty} 2^{-k} \sup_{N+1\le j \le k} 2^{j}\esup_{y\in (x_{j-1}, x_j)} \bigg(\int_y^{x_j} \bigg(\int_{t}^b u\bigg)^{\frac{q}{p-q}} u(t)dt\bigg) \notag \\
& \hskip+5cm \times 
\bigg(\int_{a}^{y} \mathcal{W}(t)^{-\frac{p}{p-r}} w(t)V_r(t,y)^{\frac{pr}{p-r}}dt\Bigg)^{\frac{q(p-r)}{r(p-q)}} \notag\\
& =: C_{6,1} +C_{6,2}.\label{C6<C61+C62}
\end{align}
Using \eqref{disc_int<sum}, we get 
\begin{align*}
C_{6,1}  &\lesssim   \sum_{k=N+2}^{\infty} 2^{-k} \sup_{N+1\le j \le k-1} 2^{j} 
\bigg(\int_{x_j}^{x_k} \bigg(\int_{t}^b u\bigg)^{\frac{q}{p-q}} u(t)dt\bigg)\\
& \hskip+5cm \times 
\bigg(\sum_{i=N+1}^{j}  2^{i\frac{r}{p-r}}V_r(x_{i-1},x_i)^{\frac{pr}{p-r}}\Bigg)^{\frac{q(p-r)}{q(p-q)}}\\
&=   \sum_{k=N+2}^{\infty} 2^{-k} \sup_{N+1\le j \le k-1} 2^{j} 
\bigg(\sum_{m=j}^{k-1}\int_{x_{m}}^{x_{m+1}} \bigg(\int_{t}^b u\bigg)^{\frac{q}{p-q}} u(t) dt\bigg)  \notag\\
& \hskip+5cm \times 
\bigg(\sum_{i=N+1}^{j}  2^{i\frac{r}{p-r}}V_r(x_{i-1},x_i)^{\frac{pr}{p-r}} \Bigg)^{\frac{q(p-r)}{r(p-q)}}.
\end{align*}
Then, using the fact that $\bigg\{ 2^{j}  \bigg(\sum_{i=N+1}^{j}  2^{i\frac{r}{p-r}}V_r(x_{i-1},x_i)^{\frac{pr}{p-r}}dt\Bigg)^{\frac{q(p-r)}{r(p-q)}} \bigg\}_{j={N+1}}^k$ is a strongly increasing sequence and applying first \eqref{inc.sup-sum}, then reindexing $k\to k+1$ and using \eqref{dec.sum-sup}, we obtain that
\begin{align}
C_{6,1}  &\lesssim \sum_{k= N+1 }^{\infty} 2^{-k} \sup_{N+1\le j \le k} 2^{j} \bigg( \int_{x_j}^{x_{j+1}} \bigg(\int_{t}^b u\bigg)^{\frac{q}{p-q}} u(t)dt\bigg)  \notag\\
& \hskip+5cm \times 
\bigg(\sum_{i=N+1}^{j}  2^{i\frac{r}{p-r}}V_r(x_{i-1},x_i)^{\frac{pr}{p-r}} \Bigg)^{\frac{q(p-r)}{r(p-q)}}  \notag\\
&\approx   \sum_{k=N+1}^{\infty} 
\bigg(\int_{x_k}^{x_{k+1}} \bigg(\int_{t}^b u\bigg)^{\frac{q}{p-q}} u(t)dt\bigg)\bigg(\sum_{i=N+1}^{k}  2^{i\frac{r}{p-r}}V_r(x_{i-1},x_i)^{\frac{pr}{p-r}} \Bigg)^{\frac{q(p-r)}{r(p-q)}}  \notag\\
& \lesssim A_4^{\frac{pq}{p-q}}.  \label{C61<A4}
\end{align}
Note that we have used \eqref{u-estimate} in the last inequality. 

Moreover, first using \eqref{dec.sum-sup}, then decomposing the integral $\int_a^y$ into the sum $\int_a^{x_{k-1}} + \int_{x_{k-1}}^y$, we have 
\begin{align*}
C_{6,2} &\approx  \sum_{k=N+1}^{\infty}  \esup_{y\in (x_{k-1}, x_k)} 
\bigg(\int_y^{x_k} \bigg(\int_{t}^b u\bigg)^{\frac{q}{p-q}} u(t)dt\bigg) 
\bigg(\int_{a}^{y} \mathcal{W}(t)^{-\frac{p}{p-r}} w(t)V_r(t,y)^{\frac{pr}{p-r}}dt\Bigg)^{\frac{q(p-r)}{r(p-q)}}  \notag\\
&\approx  \sum_{k=N+1}^{\infty}  \esup_{y\in (x_{k-1}, x_k)} 
\bigg(\int_y^{x_k} \bigg(\int_{t}^b u\bigg)^{\frac{q}{p-q}} u(t)dt\bigg)
\bigg(\int_{x_{k-1}}^{y} \mathcal{W}(t)^{-\frac{p}{p-r}} w(t)V_r(t,y)^{\frac{pr}{p-r}}dt\Bigg)^{\frac{q(p-r)}{r(p-q)}}  \notag\\
&\quad  + \sum_{k=N+2}^{\infty}  \esup_{y\in (x_{k-1}, x_k)} 
\bigg(\int_y^{x_k} \bigg(\int_{t}^b u\bigg)^{\frac{q}{p-q}} u(t)dt\bigg) \bigg(\int_{a}^{x_{k-1}} \mathcal{W}(t)^{-\frac{p}{p-r}} w(t) V_r(t,y)^{\frac{pr}{p-r}} dt\Bigg)^{\frac{q(p-r)}{r(p-q)}}.
\end{align*}
Since $V_r(t,y) \approx V_r(t,x_{k-1}) + V_r(x_{k-1},y)$ when $t<x_{k-1}<y$, we get
\begin{align*}
C_{6,2} &\approx \sum_{k=N+1}^{\infty}  \esup_{y\in (x_{k-1}, x_k)} 
\bigg(\int_y^{x_k} \bigg(\int_{t}^b u\bigg)^{\frac{q}{p-q}} u(t)dt\bigg)
\bigg(\int_{x_{k-1}}^{y} \mathcal{W}(t)^{-\frac{p}{p-r}} w(t)V_r(t,y)^{\frac{pr}{p-r}}dt\Bigg)^{\frac{q(p-r)}{r(p-q)}}  \notag\\
& \quad + \sum_{k=N+2}^{\infty}  \esup_{y\in (x_{k-1}, x_k)} 
\bigg(\int_y^{x_k} \bigg(\int_{t}^b u\bigg)^{\frac{q}{p-q}} u(t)dt\bigg) V_r(x_{k-1},y)^{\frac{pq}{p-q}} \notag\\
& \hskip+5cm \times 
\bigg(\int_{a}^{x_{k-1}} \mathcal{W}(t)^{-\frac{p}{p-r}} w(t)dt\Bigg)^{\frac{q(p-r)}{r(p-q)}}  \notag\\
&\quad  +  \sum_{k=N+2}^{\infty}   \bigg(\int_{x_{k-1}}^{x_k} \bigg(\int_{t}^b u\bigg)^{\frac{q}{p-q}} u(t)dt\bigg)\bigg(\int_{a}^{x_{k-1}} \mathcal{W}(t)^{-\frac{p}{p-r}} w(t)V_r(t,x_{k-1})^{\frac{pr}{p-r}}dt\Bigg)^{\frac{q(p-r)}{r(p-q)}}.  
\end{align*}
Next, monotonicity of $V_r$ together with \eqref{2k equiv-1} and \eqref{2k equiv 2} yields
\begin{align}
C_{6,2} &\lesssim \sum_{k=N+1}^{\infty}  2^{k\frac{q}{p-q}} \esup_{y\in (x_{k-1}, x_k)} 
\bigg(\int_y^{x_k} \bigg(\int_{t}^b u\bigg)^{\frac{q}{p-q}} u(t)dt\bigg) V_r(x_{k-1},y)^{\frac{pq}{p-q}}  \notag\\
& \quad  + \sum_{k=N+2}^{\infty}   \bigg(\int_{x_{k-1}}^{x_k} \bigg(\int_{t}^b u\bigg)^{\frac{q}{p-q}} u(t)dt\bigg)\bigg(\int_{a}^{x_{k-1}} \mathcal{W}(t)^{-\frac{p}{p-r}} w(t)V_r(t,x_{k-1})^{\frac{pr}{p-r}}dt\Bigg)^{\frac{q(p-r)}{r(p-q)}}  \notag\\
& := C_{6,21}+C_{6,22}. \label{C62<C621+C622}
\end{align}
It is easy to see that  \eqref{u-estimate} and  \eqref{disc_int<sum} yields
\begin{equation} \label{C622<A4}
C_{6,22} \lesssim A_4^{\frac{pq}{p-q}},
\quad \max\{r,q\}<p.
\end{equation}

Let us proceed with $C_{6.21}$.  
Let $y_k \in [x_{k-1}, x_{k}]$, $k \ge N+1$ be such that if $N > -\infty$ then $y_{N} : =a$, otherwise  $y_{-\infty}:= \lim_{k\rightarrow -\infty} y_k = a$ and
\begin{align}\label{sup-yk}
&\esup_{y \in (x_{k-1}, x_{k})} \bigg(\int_y^{x_{k}} \bigg(\int_t^{b} u \bigg)^{\frac{q}{p-q}} u(t) dt\bigg)  V_r(x_{k-1},y)^{\frac{pq}{p-q}} \notag \\
& \hspace{3cm}\lesssim
\bigg(\int_{y_k}^{x_{k}} \bigg(\int_t^{b} u \bigg)^{\frac{q}{p-q}}  u(t) dt\bigg) V_r(x_{k-1},y_k)^{\frac{pq}{p-q}}.
\end{align}
Observe that
$$
2^{-k} \approx \mathcal{W}(x_k) \leq \mathcal{W}(y_k) \leq \mathcal{W}(x_{k-1}) \approx 2^{-k+1} \quad\text{for $k \geq N+1$.}
$$
Thus, $\{y_k\}_{k=N+1}^\infty$ is also a discretizing sequence of $\mathcal{W}$. 

Now, using \eqref{sup-yk}, we get that
\begin{align*}
C_{6,21} & \lesssim  \sum_{k=N+1}^\infty 2^{k\frac{q}{p-q}} \bigg(\int_{y_k}^{x_{k}} \bigg(\int_t^{b} u \bigg)^{\frac{q}{p-q}} 
u(t) dt\bigg) V_r(x_{k-1},y_k)^{\frac{pq}{p-q}}  \\
& \leq \sum_{k=N+1}^\infty  2^{k\frac{q}{p-q}} \bigg(\int_{y_k}^{y_{k+1}} \bigg(\int_t^{b} u \bigg)^{\frac{q}{p-q}} u(t) dt\bigg) V_r(x_{k-1},y_k)^{\frac{pq}{p-q}} \\
& \leq \sum_{k=N+1}^\infty  \bigg(\int_{y_k}^{y_{k+1}} \bigg(\int_t^{b} u \bigg)^{\frac{q}{p-q}} u(t) dt\bigg) 
\bigg(\sum_{i=N+1}^k 2^{i\frac{r}{p-r}} V_r(x_{i-1},y_k)^{\frac{pr}{p-r}}\bigg)^{\frac{q(p-r)}{r(p-q)}}\\
& \leq \sum_{k=N+1}^\infty  \bigg(\int_{y_k}^{y_{k+1}} \bigg(\int_t^{b} u \bigg)^{\frac{q}{p-q}} u(t) dt\bigg) 
\bigg(\sum_{i=N+1}^k 2^{i\frac{r}{p-r}} V_r(y_{i-1},y_k)^{\frac{pr}{p-r}}\bigg)^{\frac{q(p-r)}{r(p-q)}}.
\end{align*}
Applying \eqref{middle} and \eqref{u-estimate} we get that
\begin{align}\label{C-6,21-yk-estimate}
C_{6,21} & \lesssim  \sum_{k=N+1}^{\infty} \bigg(\int_{y_k}^{y_{k+1}} u \bigg)\bigg(\int_{y_k}^{b} u \bigg)^{\frac{q}{p-q}}  \bigg(\sum_{i=N+1}^k 2^{i\frac{r}{p-r}} V_r(y_{i-1},y_i)^{\frac{pr}{p-r}}\bigg)^{\frac{q(p-r)}{r(p-q)}}.
\end{align}
Since $\{y_k\}_{k=N}^{\infty}$ is a discretizing sequence of $\mathcal{W}$, we obtain
\begin{equation}\label{C621<A4}
C_{6,21} \lesssim A_4^{\frac{pq}{p-q}} + B_3^{\frac{pq}{p-q}},\quad \max\{r,q\}<p.
\end{equation}
Therefore, in view of \eqref{C6<C61+C62}, \eqref{C61<A4}, \eqref{C62<C621+C622}, \eqref{C621<A4} and \eqref{C622<A4}, we arrive at
\begin{equation}\label{C6 A4}
    C_6 \lesssim A_4 + B_3, \quad \max\{r,q\}<p.
\end{equation}

Consequently, using \eqref{C1 A4B3}, \eqref{C5 A4B3} and \eqref{C6 A4}, we obtain that
\begin{equation}\label{C5+C6<B3+A4}
C_1 + C_5 + C_6  \lesssim B_3 + A_4, \quad \max\{r,q\}<p, \, \, q<1.
\end{equation}

It remains to show that $B_3 + A_4 \lesssim C_1 + C_5 + C_6$. 

Using \eqref{u-estimate}, we get
\begin{align*}
A_4^{\frac{pq}{p-q}} &\lesssim \sum_{k= N+1}^{\infty} \bigg(\int_{x_{k}}^{x_{k+1}}  \bigg(\int_{t}^{b} u \bigg)^{\frac{q}{p-q}} u(t) dt\bigg) \bigg(\sum_{i=N+1}^k 2^{i\frac{r}{p-r}} V_r(x_{i-1},x_i)^{\frac{pr}{p-r}}\bigg)^{\frac{q(p-r)}{r(p-q)}} \\
&\approx \bigg(\int_{x_{N+1}}^{b}  \bigg(\int_{t}^{b} u \bigg)^{\frac{q}{p-q}} u(t) dt\bigg) 2^{N\frac{q}{p-q}} V_r(x_N,x_{N+1})^{\frac{pq}{p-q}}\\
&\quad + \sum_{k= N+2}^{\infty} \bigg(\int_{x_{k}}^{x_{k+1}}  \bigg(\int_{t}^{b} u \bigg)^{\frac{q}{p-q}} u(t) dt\bigg) \bigg(\sum_{i=N+2}^k 2^{i\frac{r}{p-r}} V_r(x_{i-1},x_i)^{\frac{pr}{p-r}}\bigg)^{\frac{q(p-r)}{r(p-q)}}.
\end{align*}
Since $q<p$ and $r<p$, applying \eqref{disc_sum<int},
\begin{align}
A_4^{\frac{pq}{p-q}} & \lesssim \bigg(\int_{x_{N+1}}^{b}   u\bigg)^{\frac{p}{p-q}} 2^{N\frac{q}{p-q}} V_r(x_N,x_{N+1})^{\frac{pq}{p-q}} \notag\\
&\quad  + \sum_{k= N+2}^{\infty} \bigg(\int_{x_{k}}^{x_{k+1}}  \bigg(\int_{t}^{b} u \bigg)^{\frac{q}{p-q}} u(t) dt\bigg) \bigg(\int_a^{x_{k-1}} \mathcal{W}(t)^{-\frac{p}{p-r}} w(t) V_r(t,x_k)^{\frac{pr}{p-r}} dt\bigg)^{\frac{q(p-r)}{r(p-q)}} \notag\\
& =: A_{4,1} + A_{4,2}. \label{A4<A41+A42}
\end{align}
Let us start with the $A_{4,1}$ estimate. In view of \eqref{C1 A1B1}, it is easy to see that,
\begin{equation}\label{A41 C1}
A_{4,1} \lesssim A_1^{\frac{pq}{p-q}} \lesssim C_1^{\frac{pq}{p-q}}.
\end{equation}

Let us continue with $A_{4,2}$. 
\begin{align}
A_{4,2} & \leq \sum_{k= N+2}^{\infty} \bigg(\int_{x_{k}}^{x_{k+1}}  \bigg(\int_{t}^{b} u \bigg)^{\frac{q}{p-q}} u(t) dt\bigg) \bigg(\int_a^{x_{k}} \mathcal{W}(t)^{-\frac{p}{p-r}} w(t) V_r(t,x_k)^{\frac{pr}{p-r}} dt\bigg)^{\frac{q(p-r)}{r(p-q)}} \notag \\
&\lesssim \sum_{k=N+2}^{\infty} \bigg(
\int_{x_{k+1}}^{x_{k+2}} w\bigg) \esup_{y\in (a,x_k)} \mathcal{W}(y)^{-1} \bigg(\int_y^{x_{k+1}}  \bigg(\int_{t}^{b} u \bigg)^{\frac{q}{p-q}} u(t) dt\bigg) \notag\\
& \hskip+4cm \times \bigg(\int_a^y \mathcal{W}(t)^{-\frac{p}{p-r}} w(t) V_r(t,y)^{\frac{pr}{p-r}} dt\bigg)^{\frac{q(p-r)}{r(p-q)}}\notag \\
&\lesssim \sum_{k=N+2}^{\infty} 
\int_{x_{k+1}}^{x_{k+2}} w(x)  \esup_{y\in (a,x)} \mathcal{W}(y)^{-1} \bigg(\int_y^{x}  \bigg(\int_{t}^{b} u \bigg)^{\frac{q}{p-q}} u(t) dt\bigg) \notag\\
& \hskip+4cm \times \bigg(\int_a^y \mathcal{W}(t)^{-\frac{p}{p-r}} w(t) V_r(t,y)^{\frac{pr}{p-r}} dt\bigg)^{\frac{q(p-r)}{r(p-q)}} dx \notag \\
& \lesssim C_6^{\frac{pq}{p-q}}. \label{A42<C6}
\end{align}
Thus, by using \eqref{A41 C1} and \eqref{A42<C6} together with  \eqref{A4<A41+A42}, we get 
\begin{equation}\label{A4<C1C6}
A_4 \lesssim C_1 + C_6,  \quad  \max\{q,r\}<p.
\end{equation}
Therefore \eqref{A4<C1C6} together with \eqref{B3<C5} yields $B_3 + A_4 \lesssim C_1 + C_5 + C_6$, establishing the result in this case.
\medskip

\rm{(vii)} By Theorem~\ref{C:discrete solutions}, case (vii), we have that $C\approx A_4 + B_4$. We will prove that $ A_4+B_4  \approx C_1 + C_6 + C_7$.

Let us start with the estimate 
\begin{equation}\label{C1+C6+C7 A4B4}
    C_1+C_6 + C_7 \lesssim A_4 + B_4.
\end{equation} 
It is easy to see that $B_1\lesssim B_4$. Therefore it follows from  \eqref{C1 A1B1} together with \eqref{A1<A4} that
\begin{equation}\label{C1 A4B4}
    C_1\lesssim A_4+ B_4,  \quad \max\{r,q\}<p.
\end{equation}
On the other hand using \eqref{C6<C61+C62}, \eqref{C61<A4}, \eqref{C62<C621+C622}, \eqref{C622<A4}, we have $C_6^{\frac{pq}{p-q}} \lesssim C_{6,21}+ A_4^{\frac{pq}{p-q}}$. Choosing the sequence $\{y_k\}_{k=N+1}^\infty$ as in case (vi), we have \eqref{C-6,21-yk-estimate} in this case as well. Thus, since $\{y_k\}_{k=N+1}^\infty$ is also a discretizing sequence $C_{6,21} \lesssim A_4^{\frac{pq}{p-q}} + B_4^{\frac{pq}{p-q}}$ holds. Therefore, 
\begin{equation}\label{C6 < A4+B4 (vii)}
C_6 \lesssim A_4+ B_4, \quad \max\{r,q\}<p.    
\end{equation}
Finally, since $q<p$, using the properties of the discretizing sequence, we have
\begin{align*}
C_7^{\frac{pq}{p-q}} &= \sum_{k= N+1}^\infty \int_{x_{k-1}}^{x_k} w(x)\esup_{y\in (a, x)} \mathcal{W}(y)^{-\frac{p}{p-q}} \esup_{t \in (y, x)} \bigg( \int_t^{x} u \bigg)^{\frac{p}{p-q}} V_r(y, t)^{\frac{pq}{p-q}} dx \\
&\lesssim \sum_{k= N+1}^\infty 2^{-k} \esup_{y\in (a, x_{k})} \mathcal{W}(y)^{-\frac{p}{p-q}} \esup_{t\in (y, x_k)} \bigg( \int_t^{x_k} u \bigg)^{\frac{p}{p-q}} V_r(y, t)^{\frac{pq}{p-q}} \\ 
& \approx \sum_{k= N+2}^\infty 2^{-k} \sup_{N+1 \leq i \leq k-1} \esup_{y\in (x_{i-1}, x_{i})} \mathcal{W}(y)^{-\frac{p}{p-q}} \esup_{t\in (y, x_k)} \bigg( \int_t^{x_k} u \bigg)^{\frac{p}{p-q}} V_r(y, t)^{\frac{pq}{p-q}} \\
& \quad + \sum_{k= N+1}^\infty 2^{-k}  \esup_{y\in (x_{k-1}, x_{k})} \mathcal{W}(y)^{-\frac{p}{p-q}} \esup_{t\in (y, x_k)} \bigg( \int_t^{x_k} u \bigg)^{\frac{p}{p-q}} V_r(y, t)^{\frac{pq}{p-q}}.
\end{align*}
Decomposing the supremum we get
\begin{align*}
C_7^{\frac{pq}{p-q}}
& \lesssim \sum_{k= N+2}^\infty 2^{-k} \sup_{N+1 \leq i \leq k-1} \esup_{y\in (x_{i-1}, x_{i})} \mathcal{W}(y)^{-\frac{p}{p-q}} \esup_{t\in (y, x_i)} \bigg( \int_t^{x_k} u \bigg)^{\frac{p}{p-q}} V_r(y, t)^{\frac{pq}{p-q}} \\
& \quad + \sum_{k= N+2}^\infty 2^{-k} \sup_{N+1 \leq i \leq k-1} \esup_{y\in (x_{i-1}, x_{i})} \mathcal{W}(y)^{-\frac{p}{p-q}} \esup_{t\in (x_i, x_k)} \bigg( \int_t^{x_k} u \bigg)^{\frac{p}{p-q}} V_r(y, t)^{\frac{pq}{p-q}} \\
& \quad + \sum_{k= N+1}^\infty 2^{-k}  \esup_{y\in (x_{k-1}, x_{k})} \mathcal{W}(y)^{-\frac{p}{p-q}} \esup_{t\in (y, x_k)} \bigg( \int_t^{x_k} u \bigg)^{\frac{p}{p-q}} V_r(y, t)^{\frac{pq}{p-q}}.
\end{align*}
As $\mathcal{W}(y) \approx 2^{-i}$ for $y\in (x_{i-1}, x_{i})$,
\begin{align*}
C_7^{\frac{pq}{p-q}} & \lesssim   \sum_{k= N+2}^\infty 2^{-k} \sup_{N+1 \leq i \leq k-1} 2^{i\frac{p}{p-q}} 
 \esup_{t\in (x_{i-1}, x_{i})} \bigg( \int_t^{x_k} u \bigg)^{\frac{p}{p-q}} V_r(x_{i-1}, t)^{\frac{pq}{p-q}} \\
& \quad + \sum_{k= N+2}^\infty 2^{-k} \sup_{N+1 \leq i \leq k-1}  2^{i\frac{p}{p-q}}  \esup_{t\in (x_i, x_k)} \bigg( \int_t^{x_k} u \bigg)^{\frac{p}{p-q}} V_r(x_{i-1}, t)^{\frac{pq}{p-q}} \\
& \quad + \sum_{k= N+1}^\infty 2^{k\frac{q}{p-q}} \esup_{t\in ({x_{k-1}}, x_k)} \bigg( \int_t^{x_k} u \bigg)^{\frac{p}{p-q}} V_r(x_{k-1}, t)^{\frac{pq}{p-q}}.  
\end{align*}
Since $\int_t^{x_k} u = \int_t^{x_i} u + \int_{x_i}^{x_k} u$ for $t\in(x_{i-1}, x_i)$, $i<k$, $ k\ge N+2$, we have
\begin{align*}
C_7^{\frac{pq}{p-q}} & \lesssim   \sum_{k= N+2}^\infty 2^{-k} \sup_{N+1 \leq i \leq k-1} 2^{i\frac{p}{p-q}} 
 \esup_{t\in (x_{i-1}, x_{i})} \bigg( \int_t^{x_i} u \bigg)^{\frac{p}{p-q}} V_r(x_{i-1}, t)^{\frac{pq}{p-q}} \\
 & \quad + \sum_{k= N+2}^\infty 2^{-k} \sup_{N+1 \leq i \leq k-1} 2^{i\frac{p}{p-q}} \bigg( \int_{x_i}^{x_k} u \bigg)^{\frac{p}{p-q}} V_r(x_{i-1}, x_i)^{\frac{pq}{p-q}}\\
& \quad + \sum_{k= N+2}^\infty 2^{-k} \sup_{N+1 \leq i \leq k-1}  2^{i\frac{p}{p-q}}  \esup_{t\in (x_i, x_k)} \bigg( \int_t^{x_k} u \bigg)^{\frac{p}{p-q}} V_r(x_{i-1}, t)^{\frac{pq}{p-q}} \\
& \quad + \sum_{k= N+1}^\infty 2^{k\frac{q}{p-q}} \esup_{t\in ({x_{k-1}}, x_k)} \bigg( \int_t^{x_k} u \bigg)^{\frac{p}{p-q}} V_r(x_{k-1}, t)^{\frac{pq}{p-q}} \\
& \lesssim \sum_{k= N+1}^\infty 2^{-k} \sup_{N+1 \leq i \leq k} 2^{i\frac{p}{p-q}}  \esup_{t\in (x_{i-1}, x_{i})} \bigg( \int_t^{x_i} u \bigg)^{\frac{p}{p-q}} V_r(x_{i-1}, t)^{\frac{pq}{p-q}} \\
 & \quad + \sum_{k= N+2}^\infty 2^{-k} \sup_{N+1 \leq i \leq k-1} 2^{i\frac{p}{p-q}} \bigg( \int_{x_i}^{x_k} u \bigg)^{\frac{p}{p-q}} V_r(x_{i-1}, x_i)^{\frac{pq}{p-q}}\\
& \quad + \sum_{k= N+2}^\infty 2^{-k} \sup_{N+1 \leq i \leq k-1}  2^{i\frac{p}{p-q}}  \esup_{t\in (x_i, x_k)} \bigg( \int_t^{x_k} u \bigg)^{\frac{p}{p-q}} V_r(x_{i-1}, t)^{\frac{pq}{p-q}}.
\end{align*}
Since $V_r(x_{i-1},t) \approx V_r(x_{i-1}, x_{i}) + V_r(x_{i}, t)$ for $t> x_i$, we get
\begin{align}\label{C7-C71+C72+C73}
C_7^{\frac{pq}{p-q}} & \lesssim \sum_{k= N+1}^\infty 2^{-k} \sup_{N+1 \leq i \leq k} 2^{i\frac{p}{p-q}}  \esup_{t\in (x_{i-1}, x_{i})} \bigg( \int_t^{x_i} u \bigg)^{\frac{p}{p-q}} V_r(x_{i-1}, t)^{\frac{pq}{p-q}} \notag\\
 & \quad + \sum_{k= N+2}^\infty 2^{-k} \sup_{N+1 \leq i \leq k-1} 2^{i\frac{p}{p-q}} \bigg( \int_{x_i}^{x_k} u \bigg)^{\frac{p}{p-q}} V_r(x_{i-1}, x_i)^{\frac{pq}{p-q}} \notag\\
& \quad + \sum_{k= N+2}^\infty 2^{-k} \sup_{N+1 \leq i \leq k-1}  2^{i\frac{p}{p-q}}  \esup_{t\in (x_i, x_k)} \bigg( \int_t^{x_k} u \bigg)^{\frac{p}{p-q}} V_r(x_{i}, t)^{\frac{pq}{p-q}} \notag\\ 
&:= C_{7,1} + C_{7,2} + C_{7,3}.  
\end{align}
First, applying \eqref{dec.sum-sup} we get
\begin{equation} \label{C71-B4} 
C_{7,1} \approx \sum_{k= N+1}^\infty  2^{k\frac{q}{p-q}} \esup_{t\in (x_{k-1}, x_k)} \bigg( \int_t^{x_k} u \bigg)^{\frac{p}{p-q}} V_r(x_{k-1}, t)^{\frac{pq}{p-q}}=  B_4^{\frac{pq}{p-q}}. 
\end{equation}
Next, applying power rules, observe that
\begin{align*}
C_{7,2} & \leq \sum_{k= N+2}^\infty 2^{-k} \sup_{N+1 \leq i \leq k-1} \bigg( \sum_{m=i}^{k-1} \int_{x_m}^{x_{m+1}} u \bigg)^{\frac{p}{p-q}} \sup_{N+1\le j\le i} 2^{j\frac{p}{p-q}}V_r(x_{j-1}, x_j)^{\frac{pq}{p-q}}\\
& \approx \sum_{k= N+2}^\infty 2^{-k} \sup_{N+1 \leq i \leq k-1} \sum_{m=i}^{k-1} \bigg(  \int_{x_m}^{x_{m+1}} u \bigg) \bigg(  \int_{x_m}^{x_{k}} u \bigg)^{\frac{q}{p-q}} \sup_{N+1\le j\le i} 2^{j\frac{p}{p-q}}V_r(x_{j-1}, x_j)^{\frac{pq}{p-q}} \\
& \leq \sum_{k= N+2}^\infty 2^{-k} \sum_{m=N+1}^{k-1} \bigg(  \int_{x_m}^{x_{m+1}} u \bigg) \bigg(  \int_{x_m}^{b} u \bigg)^{\frac{q}{p-q}} \sup_{N+1\le j\le m} 2^{j\frac{p}{p-q}} V_r(x_{j-1}, x_j)^{\frac{pq}{p-q}}.
\end{align*}
Now, applying \eqref{dec.sum-sum}, we obtain that
\begin{align}\label{C72<A4}
C_{7,2} & \lesssim \sum_{k= N+1}^\infty 2^{-k} \bigg(  \int_{x_k}^{x_{k+1}} u \bigg) \bigg(  \int_{x_k}^{b} u \bigg)^{\frac{q}{p-q}} \sup_{N+1\le j\le k} 2^{j\frac{p}{p-q}} V_r(x_{j-1}, x_j)^{\frac{pq}{p-q}}\notag\\
& \leq \sum_{k= N+1}^\infty \bigg(  \int_{x_k}^{x_{k+1}} u \bigg) \bigg(  \int_{x_k}^{b} u \bigg)^{\frac{q}{p-q}} \sup_{N+1\le j\le k} 2^{j\frac{q}{p-q}} V_r(x_{j-1}, x_j)^{\frac{pq}{p-q}} \notag \\
& \leq A_4^{\frac{pq}{p-q}}. 
\end{align}
Finally, 
\begin{align*}
C_{7,3} &= 
\sum_{k= N+2}^\infty 2^{-k} \sup_{N+1 \leq i \leq k-1}  2^{i\frac{p}{p-q}}  \sup_{i \leq j \le k-1} \esup_{t\in (x_{j}, x_{j+1})} \bigg( \int_t^{x_k} u \bigg)^{\frac{p}{p-q}} V_r(x_{i}, t)^{\frac{pq}{p-q}}.
\end{align*}
First decomposing the integral $\int_t^{x_k}$ into sum $\int_t^{x_{j+1}} + \int_{x_{j+1}}^{x_k}$ for $j<k$, and then decomposing $V_r$, we have
\begin{align}\label{C73-C731+C732}
C_{7,3}
& \approx \sum_{k= N+2}^\infty 2^{-k} \sup_{N+1 \leq i \leq k-1}  2^{i\frac{p}{p-q}}  \sup_{i \leq j \le k-1} \esup_{t\in (x_{j}, x_{j+1})} \bigg( \int_t^{x_{j+1}} u \bigg)^{\frac{p}{p-q}} V_r(x_{i}, t)^{\frac{pq}{p-q}} \notag\\
& \quad +\sum_{k= N+2}^\infty 2^{-k} \sup_{N+1 \leq i \leq k-1}  2^{i\frac{p}{p-q}}  \sup_{i \leq j < k-1} \bigg( \int_{x_{j+1}}^{x_k} u \bigg)^{\frac{p}{p-q}} V_r(x_{i}, x_{j+1})^{\frac{pq}{p-q}} \notag \\
& \approx \sum_{k= N+2}^\infty 2^{-k} \sup_{N+1 \leq i \leq k-1}  2^{i\frac{p}{p-q}}  \sup_{i \le j \le k-1} \esup_{t\in (x_{j}, x_{j+1})} \bigg( \int_t^{x_{j+1}} u \bigg)^{\frac{p}{p-q}} V_r(x_{j}, t)^{\frac{pq}{p-q}} \notag \\
& \quad +\sum_{k= N+2}^\infty 2^{-k} \sup_{N+1 \leq i \leq k-1}  2^{i\frac{p}{p-q}}  \sup_{i < j \le k-1} \bigg( \int_{x_j}^{x_{j+1}} u \bigg)^{\frac{p}{p-q}} V_r(x_{i}, x_j)^{\frac{pq}{p-q}} \notag \\
& \quad +\sum_{k= N+2}^\infty 2^{-k} \sup_{N+1 \leq i \leq k-1}  2^{i\frac{p}{p-q}}  \sup_{i \leq j < k-1} \bigg( \int_{x_{j+1}}^{x_k} u \bigg)^{\frac{p}{p-q}} V_r(x_{i}, x_{j+1})^{\frac{pq}{p-q}} \notag \\
& \leq \sum_{k= N+2}^\infty 2^{-k} \sup_{N+1 \le j \le k-1} 2^{j\frac{p}{p-q}}   \esup_{t\in (x_{j}, x_{j+1})} \bigg( \int_t^{x_{j+1}} u \bigg)^{\frac{p}{p-q}} V_r(x_{j}, t)^{\frac{pq}{p-q}} \notag \\
& \quad +\sum_{k= N+2}^\infty 2^{-k} \sup_{N+1 \leq i \leq k-1}  2^{i\frac{p}{p-q}}  \sup_{i \leq j < k-1} \bigg( \int_{x_{j+1}}^{x_k} u \bigg)^{\frac{p}{p-q}} V_r(x_{i}, x_{j+1})^{\frac{pq}{p-q}} \notag \\
&= C_{7,31} + C_{7,32}.
\end{align}
Analogous to the estimation of $C_{7,1}$ we have 
\begin{equation}\label{C731-B4}
C_{7,31} \lesssim B_4^{\frac{pq}{p-q}}.
\end{equation}
On the other hand,
\begin{align*}
C_{7,32} 
&\leq \sum_{k= N+2}^\infty 2^{-k}  \sup_{N+1 \leq j <  k-1} \bigg( \int_{x_{j+1}}^{x_k} u \bigg)^{\frac{p}{p-q}} \sup_{N+1 \le m\le j }2^{m\frac{p}{p-q}} V_r(x_{m}, x_{j+1})^{\frac{pq}{p-q}} \\
& = \sum_{k= N+2}^\infty 2^{-k}  \sup_{N+1 \leq j <  k-1} \bigg( \int_{x_{j+1}}^{x_k} u \bigg)^{\frac{p}{p-q}} \sup_{N+2 \le m\le j+1 }2^{m\frac{p}{p-q}} V_r(x_{m-1}, x_{j+1})^{\frac{pq}{p-q}}
\end{align*}
holds. Then, in view of \eqref{xi-xk}, we have
\begin{align*}
C_{7,32}
&\lesssim \sum_{k= N+2}^\infty 2^{-k}  \sup_{N+1 \leq j < k-1} \bigg( \int_{x_{j+1}}^{x_k} u \bigg)^{\frac{p}{p-q}} \sup_{N+2 \le m\le j+1 }2^{m\frac{p}{p-q}} V_r(x_{m-1}, x_{m})^{\frac{pq}{p-q}}.
\end{align*}
Thus, analogously to the estimate of $C_{7,2}$ we have 
\begin{equation}\label{C732-A4}
C_{7,32} \lesssim A_4^{\frac{pq}{p-q}}.
\end{equation}
Consequently, combining \eqref{C7-C71+C72+C73}, \eqref{C71-B4}, \eqref{C72<A4}, \eqref{C73-C731+C732}, \eqref{C731-B4} and \eqref{C732-A4}, we arrive at 
\begin{equation} \label{C7 A4B4}
   C_7 \lesssim  B_4 + A_4. 
\end{equation} 
Therefore, \eqref{C1 A4B4}, \eqref{C6 < A4+B4 (vii)} and \eqref{C7 A4B4} results in \eqref{C1+C6+C7 A4B4}. 

It remains to show that $B_4 + A_4 \lesssim C_1+C_6 + C_7$. By \eqref{A4<C1C6}, we have $A_4 \lesssim C_1 + C_6$ , when $\max\{q,r\}<p$. On the other hand, it is easy to see that,
\begin{align*}
B_4^{\frac{pq}{p-q}} 
& \lesssim \sum_{k=N+1}^{\infty} 2^{-k} \esup_{y \in (x_{k-1}, x_k)} \mathcal{W}(y)^{-\frac{p}{p-q}} \esup_{t \in (y, x_{k})} \bigg(\int_t^{x_{k}} u \bigg)^{\frac{p}{p-q}} V_r(y,t)^{\frac{pq}{p-q}} \notag \\
& \lesssim \sum_{k=N+1}^{\infty} \int_{x_{k}}^{x_{k+1}} w(x)\,dx \esup_{y \in (a,x_{k})} \mathcal{W}(y)^{-\frac{p}{p-q}}
\esup_{t\in(y,x_{k})} \bigg( \int_t^{x_k} u \bigg)^{\frac{p}{p-q}} V_r(y,t)^{\frac{pq}{p-q}}  \notag\\
& \lesssim  \sum_{k=N+1}^{\infty} \int_{x_{k}}^{x_{k+1}}  w(x) \esup_{y \in (a, x)} \mathcal{W}(y)^{-\frac{p}{p-q}}
\esup_{t\in (y,x)} \bigg( \int_t^{x} u \bigg)^{\frac{p}{p-q}}  V_r(y,t)^{\frac{pq}{p-q}}\,dx \notag \\
& \leq C_7^{\frac{pq}{p-q}}. \label{B4<C7}
\end{align*}

Consequently, we have
\begin{equation}
A_4+ B_4 \lesssim C_1 + C_6 + C_7,  \quad \max\{q,r\}<p
\end{equation}
which completes the proof in this case.
\end{proof}

\begin{proof}[Proof of Theorem~\ref{T:equiv.solut.}] 
By Theorem~\ref{C:discrete solutions}, case (vi), we have $C\approx A_4 + B_3$. Observe that in addition to the positions of parameters in Theorem~\ref{C:discrete solutions}, case (vi), we now also have the restriction $r\leq q$. We will prove that $ A_4+B_3  \approx \mathcal{C}_5 + \mathcal{C}_6$. We know from \eqref{B3<C5} and \eqref{A4<C1C6} that $A_4 + B_3 \lesssim C_1 + C_5 + C_6$. 

Conversely, in view of \eqref{GPU-6.3}, it is clear that
\begin{align*}
C_1^{\frac{pq}{p-q}} & \leq  \esup_{x\in (a,b)} \mathcal{W}(x)^{-\frac{q}{p-q}} \bigg(\int_x^b \bigg( \int_t^b u \bigg)^{\frac{q}{1-q}} u(t)  V_r(x,t)^{\frac{q}{1-q}} dt \bigg)^{\frac{p(1-q)}{p-q}}\\
& \leq  \esup_{x\in (a,b)} \bigg(\mathcal{W}(x)^{-\frac{q}{p-q}} -\mathcal{W}(a)^{-\frac{q}{p-q}}\bigg) \bigg(\int_x^b \bigg( \int_t^b u \bigg)^{\frac{q}{1-q}} u(t)  V_r(x,t)^{\frac{q}{1-q}} dt \bigg)^{\frac{p(1-q)}{p-q}}\\
& +\mathcal{W}(a)^{-\frac{q}{p-q}}\bigg(\int_a^b \bigg( \int_t^b u \bigg)^{\frac{q}{1-q}} u(t)  V_r(a,t)^{\frac{q}{1-q}} dt \bigg)^{\frac{p(1-q)}{p-q}}\\
& \leq  \esup_{x\in (a,b)} \bigg(\int_a^x \mathcal{W}(y)^{-\frac{p}{p-q}} w(y)\,dy \bigg)\bigg(\int_x^b \bigg( \int_t^b u \bigg)^{\frac{q}{1-q}} u(t)  V_r(x,t)^{\frac{q}{1-q}} dt \bigg)^{\frac{p(1-q)}{p-q}}\\
& +\mathcal{W}(a)^{-\frac{q}{p-q}}\bigg(\int_a^b \bigg( \int_t^b u \bigg)^{\frac{q}{1-q}} u(t)  V_r(a,t)^{\frac{q}{1-q}} dt \bigg)^{\frac{p(1-q)}{p-q}}\\
& \leq \mathcal{C}_5^{\frac{pq}{p-q}}.
\end{align*}
Since $q<p<1$ and $V_r$ is non-increasing in the first variable, it is
easy to see that 
\begin{align*}
C_5^{\frac{pq}{p-q}} & =  \int_a^b  w(x) \sup_{y\in (a,x)} \mathcal{W}(y)^{-\frac{p}{p-q}}\bigg( \int_{y}^x \bigg( \int_t^x u \bigg)^{\frac{q}{1-q}} u(t) V_r(y,t)^{\frac{q}{1-q}} dt \bigg)^{\frac{p(1-q)}{p-q}} dx \\
& \leq \int_a^b  w(x) \sup_{y\in (a,x)} \mathcal{W}(y)^{-\frac{p}{p-q}}\bigg( \int_{y}^b \bigg( \int_t^b u \bigg)^{\frac{q}{1-q}} u(t) V_r(y,t)^{\frac{q}{1-q}} dt \bigg)^{\frac{p(1-q)}{p-q}} dx \\
&  \approx \int_a^b  w(x) \sup_{y\in (a,x)} \bigg( \int_a^y \mathcal{W}(s)^{-\frac{p}{p-q}-1} w(s) ds  \bigg) \\
& \hskip+4cm \times \bigg( \int_{y}^b \bigg( \int_t^b u \bigg)^{\frac{q}{1-q}} u(t) V_r(y,t)^{\frac{q}{1-q}} dt \bigg)^{\frac{p(1-q)}{p-q}} dx \\
& \quad + \int_a^b  w(x) \sup_{y\in (a,x)} \mathcal{W}(a)^{-\frac{p}{p-q}} \bigg( \int_{y}^b \bigg( \int_t^b u \bigg)^{\frac{q}{1-q}} u(t) V_r(y,t)^{\frac{q}{1-q}} dt \bigg)^{\frac{p(1-q)}{p-q}} dx \\
&  \leq \int_a^b  w(x)  \int_a^x \mathcal{W}(s)^{-\frac{p}{p-q}-1} w(s) \bigg( \int_{s}^b \bigg( \int_t^b u \bigg)^{\frac{q}{1-q}} u(t) V_r(s,t)^{\frac{q}{1-q}}   dt \bigg)^{\frac{p(1-q)}{p-q}} ds dx \\
& \quad + \mathcal{W}(a)^{-\frac{q}{p-q}}  \bigg( \int_{a}^b \bigg( \int_t^b u \bigg)^{\frac{q}{1-q}} u(t) V_r(a,t)^{\frac{q}{1-q}} dt \bigg)^{\frac{p(1-q)}{p-q}}.
\end{align*} 
Now, Fubini's theorem yields,
\begin{align*}
C_5^{\frac{pq}{p-q}} & \lesssim \int_a^b  \mathcal{W}(s)^{-\frac{p}{p-q}} w(s) \bigg( \int_{s}^b \bigg( \int_t^b u \bigg)^{\frac{q}{1-q}} u(t) V_r(s,t)^{\frac{q}{1-q}}   dt \bigg)^{\frac{p(1-q)}{p-q}} ds \\
& \quad +\mathcal{W}(a)^{-\frac{q}{p-q}}\bigg( \int_{a}^b \bigg( \int_t^b u \bigg)^{\frac{q}{1-q}} u(t) V_r(a,t)^{\frac{q}{1-q}} dt \bigg)^{\frac{p(1-q)}{p-q}}\\
& \approx \mathcal{C}_5^{\frac{pq}{p-q}}.
\end{align*} 

On the other hand, using \eqref{C6-new}, \eqref{inc.sup-sum} and \eqref{dec.sum-sup}, we have
\begin{align*}
C_6^{\frac{pq}{p-q}} & \lesssim  \sum_{k=N+1}^\infty 2^{-k} \! \sup_{ N+1\le i \leq k} 2^i \int_{x_{i-1}}^{x_k} \bigg(\int_t^b u\bigg)^{\frac{q}{p-q}} u(t) \bigg(\int_a^t  \mathcal{W}(\tau)^{-\frac{p}{p-r}} w(\tau) V_r(\tau,t)^{\frac{pr}{p-r}} d\tau \bigg)^{\frac{q(p-r)}{r(p-q)}} \! dt\\
& \approx \sum_{k=N+1}^\infty 2^{-k} \! \sup_{ N+1\le i \le k} 2^i \int_{x_{i-1}}^{x_{i}} 
\bigg(\int_t^b u\bigg)^{\frac{q}{p-q}} u(t) \bigg(\int_a^t  \mathcal{W}(\tau)^{-\frac{p}{p-r}} w(\tau) V_r(\tau,t)^{\frac{pr}{p-r}} d\tau \bigg)^{\frac{q(p-r)}{r(p-q)}}\!  dt  \\
& \approx \sum_{k=N+1}^\infty \int_{x_{k-1}}^{x_{k}} 
\bigg(\int_t^b u\bigg)^{\frac{q}{p-q}} u(t) \bigg(\int_a^t  \mathcal{W}(\tau)^{-\frac{p}{p-r}} w(\tau) V_r(\tau,t)^{\frac{pr}{p-r}} d\tau \bigg)^{\frac{q(p-r)}{r(p-q)}} dt \\
& = \mathcal{C}_6^{\frac{pq}{p-q}}.
\end{align*}    
Thus, we get $B_3 + A_4 \lesssim \mathcal{C}_5 + \mathcal{C}_6$.

Now, we will prove that $\mathcal{C}_5 + \mathcal{C}_6 \lesssim B_3 + A_4$. Using the fact that $ \mathcal{W}(x_k)  \approx 2^{-k}$, $N \leq k$ together with the monotonicity of $V_r$, one has
\begin{align*}
\mathcal{C}_5^{\frac{pq}{p-q}} & \lesssim  \sum_{k=N+1}^\infty \int_{x_{k-1}}^{x_{k}} \mathcal{W}(x)^{-\frac{p}{p-q}} w(x)dx  
\bigg( \int_{x_{k-1}}^b \bigg( \int_t^b u \bigg)^{\frac{q}{1-q}} u(t) V_r(x_{k-1}, t)^{\frac{q}{1-q}} dt \bigg)^{\frac{p(1-q)}{p-q}}  \notag\\
&\quad + \mathcal{W}(a)^{-\frac{q}{p-q}} \bigg(\int_{a}^{b} \bigg(\int_t^b u \bigg)^{\frac{q}{1-q}} u(t) V_r(a,t)^{\frac{q}{1-q}}dt\Bigg)^{\frac{p(1-q)}{p-q}}\\
&\approx  \sum_{k=N+1}^\infty 2^{k\frac{q}{p-q}} 
\bigg( \int_{x_{k-1}}^b \bigg( \int_t^b u \bigg)^{\frac{q}{1-q}} u(t) V_r(x_{k-1}, t)^{\frac{q}{1-q}} dt \bigg)^{\frac{p(1-q)}{p-q}}\\
& = \sum_{k=N+1}^\infty 2^{k\frac{q}{p-q}} 
\bigg( \sum_{j=k}^{\infty} \int_{x_{j-1}}^{x_j} \bigg( \int_t^b u \bigg)^{\frac{q}{1-q}} u(t) V_r(x_{k-1}, t)^{\frac{q}{1-q}} dt \bigg)^{\frac{p(1-q)}{p-q}}.
\end{align*}

Then, applying  Lemma~\ref{cutinglem} for $ a=x_{k-1} $,  $ c=x_{j-1}$, $ d=x_{j}$ for $N+1 \leq k \leq j$, $\beta = \frac{q}{1-q}$ and $k(x,y)= V_r(x,y)^{\frac{q}{1-q}}$, we obtain
\begin{align*}
\mathcal{C}_5^{\frac{pq}{p-q}} & \lesssim  \sum_{k=N+1}^\infty 2^{k\frac{q}{p-q}} \Bigg(\sum_{j=k}^{\infty}
\int_{x_{j-1}}^{x_j} \bigg( \int_t^{x_j} u \bigg)^{\frac{q}{1-q}} u(t) V_r(x_{j-1}, t)^{\frac{q}{1-q}} dt \Bigg)^{\frac{p(1-q)}{p-q}}  \\
&\quad + \sum_{k=N+1}^\infty 2^{k\frac{q}{p-q}} \Bigg(  \sum_{j=k}^{\infty}
\bigg( \int_{x_{j}}^b u \bigg)^{\frac{1}{1-q}} \int_{x_{j-1}}^{x_j} d\bigg[V_r(x_{k-1}, t)^{\frac{q}{1-q}} \bigg]  \Bigg)^{\frac{p(1-q)}{p-q}}  \\
&\quad + \sum_{k=N+1}^\infty 2^{k\frac{q}{p-q}} \Bigg(   \sum_{j=k}^{\infty} V_r(x_{k-1},x_{j-1}+)^{\frac{q}{1-q}} \Bigg[ \bigg( \int_{x_{j-1}}^{b} u \bigg)^{\frac{1}{1- q}} - \bigg( \int_{x_{j}}^{b} u \bigg)^{\frac{1}{1-q}}\Bigg]\Bigg)^{\frac{p(1-q)}{p-q}}.
\end{align*}
Note that in this case $0<r<1$, therefore $V_r(x_{k-1},x_{k-1}+) = 0$ and $V_r(x_{k-1},x_{j-1}+) = V_r(x_{k-1},x_{j-1})$, $j>k$. Next, applying \cite[Lemma~4.4]{GPU-JFA} to the second term and \eqref{u-estimate} to the last term, we obtain
\begin{align}\label{NC5 C51+C52}
\mathcal{C}_5^{\frac{pq}{p-q}} & \lesssim  \sum_{k=N+1}^\infty 2^{k\frac{q}{p-q}} \Bigg(\sum_{j=k}^{\infty}
\int_{x_{j-1}}^{x_j} \bigg( \int_t^{x_j} u \bigg)^{\frac{q}{1-q}} u(t) V_r(x_{j-1}, t)^{\frac{q}{1-q}} dt \Bigg)^{\frac{p(1-q)}{p-q}}  \notag\\
&\quad + \sum_{k=N+1}^\infty 2^{k\frac{q}{p-q}} \Bigg(   \sum_{j=k+1}^{\infty} V_r(x_{k-1},x_{j-1})^{\frac{q}{1-q}} \bigg( \int_{x_{j-1}}^{x_j} u \bigg) \bigg( \int_{x_{j-1}}^{b} u \bigg)^{\frac{q}{1- q}} \Bigg)^{\frac{p(1-q)}{p-q}} \notag\\
& =: \mathcal{C}_{5,1} + \mathcal{C}_{5,2}.
\end{align}

In view of \eqref{inc.sum-sum}, we get
\begin{equation} \label{NC51<B3}
\mathcal{C}_{5,1} \approx B_3^{\frac{pq}{p-q}}.
\end{equation}

On the other hand, it is clear that $d_{k,j+1}=V_r(x_{k-1},x_{j-1})^{\frac{q}{1-q}}$ is a regular kernel,  using \eqref{kersum} and power rules, we obtain 
\begin{align}
\mathcal{C}_{5,2} & \approx \sum_{k=N+1}^\infty 2^{k\frac{q}{p-q}}
V_r(x_{k-1},x_{k})^{\frac{pq}{p-q}} \bigg(\sum_{j=k+1}^{\infty} \bigg( \int_{x_{j-1}}^{x_j} u \bigg) \bigg( \int_{x_{j-1}}^{b} u \bigg)^{\frac{q}{1- q}} \bigg)^{\frac{p(1-q)}{p-q}} \notag\\
& \approx \sum_{k=N+1}^\infty 2^{k\frac{q}{p-q}}
V_r(x_{k-1},x_{k})^{\frac{pq}{p-q}}  \bigg( \int_{x_{k}}^b u \bigg)^{\frac{p}{p-q}}. \label{mid}
\end{align}
Moreover, applying power rules once more and using Fubini's theorem, we have
\begin{align}
\mathcal{C}_{5,2} &\approx  \sum_{k=N+1}^\infty  \bigg(\sum_{i=k}^\infty \int_{x_{i}}^{x_{i+1}} u \bigg)^{\frac{p}{p-q}} 2^{k\frac{q}{p-q}} V_r( x_{k-1}, x_k)^{\frac{pq}{p-q}} \notag\\ 
& \approx  \sum_{k=N+1}^\infty  \sum_{i=k}^\infty\bigg( \int_{x_{i}}^{b} u \bigg)^{\frac{q}{p-q}}\bigg( \int_{x_{i}}^{x_{i+1}} u \bigg) 2^{k\frac{q}{p-q}}
V_r( x_{k-1}, x_k)^{\frac{pq}{p-q}}   \notag\\ 
&=  \sum_{i=N+1}^\infty 
\bigg( \int_{x_{i}}^{b} u \bigg)^{\frac{q}{p-q}} \bigg( \int_{x_{i}}^{x_{i+1}} u \bigg) \sum_{k=N+1}^{i}2^{k\frac{q}{p-q}}V_r(x_{k-1}, x_{k})^{\frac{pq}{p-q}}.  \label{C52<}
\end{align}
Since  $1 \leq \frac{q(p-r)}{r(p-q)}$ in this case, we get
\begin{align}
\sum_{k=N+1}^{i}2^{k\frac{q}{p-q}}V_r(x_{k-1}, x_{k})^{\frac{pq}{p-q}}  \le \Bigg( \sum_{k=N+1}^{i} 2^{k\frac{r}{p-r}} V_r(x_{k-1}, x_{k})^{\frac{pr}{p-r}} \Bigg)^{\frac{q(p-r)}{r(p-q)}}. \label{5.34}
\end{align}
Therefore, from \eqref{C52<} it follows that 
\begin{equation} \label{NC52<A4}
{\mathcal{C}}_{5,2} \lesssim A_4^{\frac{pq}{p-q}}.
\end{equation}
Thus, combining \eqref{NC5 C51+C52}, \eqref{NC51<B3} and \eqref{NC52<A4}, we arrive at
\begin{equation} \label{NC5 < B3+A4}
\mathcal{C}_5 \lesssim  B_3 + A_4.  
\end{equation}

Let us now find the upper estimate of $\mathcal{C}_6$. 

Decomposing the integral $\int_a^x$ into sum $\int_a^{x_{k-1}} + \int_{x_{k-1}}^x$ for $x_{k-1}<x<x_k$, $k\ge N+2$, we get
\begin{align*}
\mathcal{C}_6^{\frac{pq}{p-q}} &= \sum_{k=N+1}^{\infty} \int_{x_{k-1}}^{x_k} \bigg(\int_x^{b} u\bigg)^{\frac{q}{p-q}} u(x) \bigg(\int_a^{x} \mathcal{W}(t)^{-\frac{p}{p-r}} w(t) V_r(t, x)^{\frac{pr}{p-r}} dt \bigg)^{\frac{q(p-r)}{r(p-q)}} dx \\
& \approx \sum_{k=N+2}^{\infty} \int_{x_{k-1}}^{x_k} \bigg(\int_x^{b} u\bigg)^{\frac{q}{p-q}} u(x)  \bigg(\int_a^{x_{k-1}} \mathcal{W}(t)^{-\frac{p}{p-r}} w(t) V_r(t, x)^{\frac{pr}{p-r}} dt \bigg)^{\frac{q(p-r)}{r(p-q)}} dx \\
& \quad + \sum_{k=N+2}^{\infty} \int_{x_{k-1}}^{x_k} \bigg(\int_x^{b} u\bigg)^{\frac{q}{p-q}} u(x)  \bigg(\int_{x_{k-1}}^x \mathcal{W}(t)^{-\frac{p}{p-r}} w(t) V_r(t, x)^{\frac{pr}{p-r}} dt \bigg)^{\frac{q(p-r)}{r(p-q)}} dx \\
& \quad + \int_{x_{N}}^{x_{N+1}} \bigg(\int_x^{b} u\bigg)^{\frac{q}{p-q}} u(x)  \bigg(\int_a^x \mathcal{W}(t)^{-\frac{p}{p-r}} w(t) V_r(t, x)^{\frac{pr}{p-r}} dt \bigg)^{\frac{q(p-r)}{r(p-q)}} dx. 
\end{align*}
Then, applying \eqref{Vr-cut-t<x} in the first term, using the monotonicity of $V_r$ in the second and third terms and the properties of the discretizing sequence gives
\begin{align*}
\mathcal{C}_6^{\frac{pq}{p-q}} & \lesssim \sum_{k=N+2}^{\infty} \bigg( \int_{x_{k-1}}^{x_k} \bigg(\int_x^{b} u\bigg)^{\frac{q}{p-q}} u(x) dx \bigg) \bigg(\int_a^{x_{k-1}} \mathcal{W}(t)^{-\frac{p}{p-r}} w(t) V_r(t, x_{k-1})^{\frac{pr}{p-r}} dt \bigg)^{\frac{q(p-r)}{r(p-q)}} \\
& \quad + \sum_{k=N+2}^{\infty} 2^{k\frac{q}{p-q}} \int_{x_{k-1}}^{x_k} \bigg(\int_x^{b} u\bigg)^{\frac{q}{p-q}} u(x) V_r(x_{k-1}, x)^{\frac{pq}{p-q}} dx \\
& \quad + 2^{N\frac{q}{p-q}} \int_{x_{N}}^{x_{N+1}} \bigg(\int_x^{b} u\bigg)^{\frac{q}{p-q}} u(x) V_r(a, x)^{\frac{pq}{p-q}}  dx\\
& \approx \sum_{k=N+2}^{\infty} \bigg( \int_{x_{k-1}}^{x_k} \bigg(\int_x^{b} u\bigg)^{\frac{q}{p-q}} u(x) dx \bigg)\bigg(\int_a^{x_{k-1}} \mathcal{W}(t)^{-\frac{p}{p-r}} w(t) V_r(t, x_{k-1})^{\frac{pr}{p-r}} dt \bigg)^{\frac{q(p-r)}{r(p-q)}} \\
& \quad + \sum_{k=N+1}^{\infty} 2^{k\frac{q}{p-q}} \int_{x_{k-1}}^{x_k} \bigg(\int_x^{b} u\bigg)^{\frac{q}{p-q}} u(x) V_r(x_{k-1}, x)^{\frac{pq}{p-q}} dx.
\end{align*}
Next, applying \eqref{u-estimate} for the first term and using Lemma~\ref{cutinglem} with  $a= c=x_{k-1} $, $ d=x_{k}$, $\beta = \frac{q}{p-q}$ and $ k(x,y)= V_r(x,y)^{\frac{pq}{p-q}}$ for the second term, using the fact that $V_r(x_{k-1},x_{k-1}  +)=0$ since $r<1$, we obtain
\begin{align*}
\mathcal{C}_6^{\frac{pq}{p-q}} & \lesssim \sum_{k=N+2}^{\infty}  \bigg(\int_{x_{k-1}}^{x_k} u\bigg) \bigg(\int_{x_{k-1}}^b u\bigg)^{\frac{q}{p-q}} \Bigg( \int_{a}^{x_{k-1}} \mathcal{W}(t)^{-\frac{p}{p-r}}  w(t) V_r(t,x_{k-1})^{\frac{pr}{p-r}}dt  \Bigg)^{\frac{q(p-r)}{r(p-q)}}\\
& \quad + \sum_{k=N+1}^{\infty} 2^{k\frac{q}{p-q}}\int_{x_{k-1}}^{x_k} \bigg(\int_{x}^{x_k} u\bigg)^{\frac{q}{p-q}} u(x) V_r(x_{k-1},x)^{\frac{pq}{p-q}} dx \\
&\quad + \sum_{k=N+1}^{\infty} 2^{k\frac{q}{p-q}} \bigg(\int_{x_{k}}^{b}  u\bigg)^{\frac{p}{p-q}} V_r(x_{k-1},x_{k})^{\frac{pq}{p-q}} \\
& =: \mathcal{C}_{6,1} + \mathcal{C}_{6,2} + \mathcal{C}_{6,3}.
\end{align*}
Using \eqref{disc_int<sum}, we can see that
\begin{align*}
\mathcal{C}_{6,1}  &\lesssim  \sum_{k=N+2}^{\infty}   \bigg(\int_{x_{k-1}}^{x_k} u\bigg) \bigg(\int_{x_{k-1}}^b u\bigg)^{\frac{q}{p-q}}
\bigg(\sum_{j=N+1}^{k-1} 2^{j\frac{r}{p-r}}
V_r(x_{j-1},x_{j})^{\frac{pr}{p-r}}  \Bigg)^{\frac{q(p-r)}{r(p-q)}}= A_4^{\frac{pq}{p-q}}.
\end{align*}
Moreover, since $V_r(x_{k-1}, t)$ is an increasing function in $t$, and  $\frac{p-q}{p(1-q)} <1$, applying a special case of \cite[Proposition 2.1]{He-St:93}, we get
\begin{align*}
\int_{x_{k-1}}^{x_k} \bigg(\int_{x}^{x_k} u\bigg)^{\frac{q}{p-q}} &u(x) V_r(x_{k-1},x)^{\frac{pq}{p-q}} dx\\
& \lesssim \bigg(\int_{x_{k-1}}^{x_k} \bigg(\int_{x}^{x_k} u\bigg)^{\frac{q}{1-q}} u(x) V_r(x_{k-1},x)^{\frac{q}{1-q}}dx\bigg)^{\frac{p(1-q)}{p-q}}.
\end{align*}
Therefore, 
\begin{equation*}
\mathcal{C}_{6,2} \lesssim  B_3^{\frac{pq}{p-q}}.    
\end{equation*}
Furthermore, using \eqref{mid}  and \eqref{NC52<A4}, we get 
\begin{equation*}
\mathcal{C}_{6,3}\lesssim A_4^{\frac{pq}{p-q}}.   
\end{equation*}
Therefore, 
\begin{equation*}
\mathcal{C}_6\lesssim A_4+B_3.    
\end{equation*}    
and the proof is complete.
\end{proof}

\bibliographystyle{abbrv}

%\bibliography{GU-Ces-Ces}

\end{document}